%% file: Microlocal_partition_of_energy_for_fractional-type_dispersive_equations.tex
\documentclass[12pt,a4paper,leqno]{amsart}
\usepackage{amsmath}
\usepackage{amssymb}
\usepackage{bbm}
\usepackage{amsfonts}

\usepackage{amsthm}
\newtheorem{definition}{Definition}[section]
\newtheorem{theorem}[definition]{Theorem}
\newtheorem{proposition}[definition]{Proposition}
\newtheorem{remark}[definition]{Remark}
\newtheorem{corollary}[definition]{Corollary}
\newtheorem{lemma}[definition]{Lemma}

\usepackage{indentfirst}
\usepackage{multirow}
\numberwithin{equation}{section}
\allowdisplaybreaks[4] 

\usepackage{geometry}
\usepackage{tikz}

\usepackage[foot]{amsaddr} 

\title{Microlocal partition of energy for fractional-type dispersive equations}
\author{Haocheng Yang$^{1,2}$}
\thanks{$^{1}$Ecole Normale Supérieure Paris-Saclay, CNRS Centre Borelli UMR9010, 4 Avenue des Sciences, F-91190 Gif-sur-Yvette}
\thanks{$^{2}$Universit{\'e} Paris XIII (Sorbonne Paris-Nord), LAGA, CNRS (UMR 7539), 99 Avenue J.-B. Cl{\'e}ment, F-93430 Villetaneuse}

\usepackage{cite}

\usepackage{verbatim} 

\usepackage[colorlinks,linkcolor=blue,citecolor=red,urlcolor=black]{hyperref}

\begin{document}
\pagenumbering{arabic}

\newcommand{\Op}[2][]{\operatorname{Op}^{#1}\left(#2\right)}
\newcommand{\Supp}[1]{\operatorname{Supp} #1}
\newcommand{\sgn}[1]{\operatorname{sgn}\left(#1\right)}
\newcommand{\Real}{\operatorname{Re}}
\newcommand{\Imaginary}{\operatorname{Im}}

\newcommand{\mar}[1]{\marginpar{\textcolor{red}{#1}}}

\begin{abstract}
This paper is devoted to the proof of microlocal partition of energy for fractional-type dispersive equations including Schrödinger equation, linearized gravity or capillary water-wave equation and half-Klein-Gordon equation. Roughly speaking, a quarter of the $L^2$ energy lies inside or outside the ``light cone'' $|x| = |tP'(\xi)|$ for large time. In addition, based on the study of half-Klein-Gordon equation, the microlocal partition of energy will also be proved for Klein-Gordon equation.
\end{abstract}

\maketitle


\input{main/Introduction}

\section*{Acknowledgment}
We thank Luis Vega and Carlos Kenig for some discussions related to Proposition \ref{prop_nullity_of_critical_limit} above.

\input{main/L2_bound}

\input{main/L2_bound_alter}

\input{main/Limit_of_energy}

\input{main/Examples}

\appendix

\input{appendix/Appendix_technical_lemma}

\input{appendix/Appendix_stationary_phase_lemma}

\input{appendix/Appendix_Schrodinger_equation}

\input{appendix/Appendix_Klein-Gordon_classical_truncation}

\input{appendix/Appendix_limit_of_integral}

\bibliographystyle{abbrv}
\bibliography{ref}

\end{document}

%% file: main/Introduction.tex
\section{Introduction}\label{sect_intro}

\subsection{Background}

The classical partition of energy states that the energy of the solution $w$ to linear wave equation
\begin{equation}\tag{W}\label{eq_wave}
\left\{
\begin{aligned}
&(\partial_t^2 - \Delta)w = 0, \\
&w|_{t=0} = w_0 \in \dot{H}^1(\mathbb{R}^d), \\
&\partial_t w|_{t=0} = w_1 \in L^2(\mathbb{R}^d),
\end{aligned}
\right.
\end{equation}
inside and outside the light cone $|x|=|t|$ satisfies, in \textit{odd} dimension $d$,
\begin{equation}\label{eq_partition_of_energy_wave}
\begin{aligned}
&\lim_{t\rightarrow +\infty} \left( E^\text{in}(w_0,w_1,t) + E^\text{in}(w_0,w_1,-t) \right) = \|\partial_t w\|_{L^2}^2 + \|\nabla w\|_{L^2}^2, \\
&\lim_{t\rightarrow +\infty} \left( E^\text{out}(w_0,w_1,t) + E^\text{out}(w_0,w_1,-t) \right) = \|\partial_t w\|_{L^2}^2 + \|\nabla w\|_{L^2}^2,
\end{aligned}
\end{equation}
where
\begin{equation*}
\begin{aligned}
E^\text{in}(w_0,w_1,t) &:= \int_{|x|<|t|} \left( |\nabla w|^2 + |\partial_t w|^2 \right) dx, \\
E^\text{out}(w_0,w_1,t) &:= \int_{|x|>|t|} \left( |\nabla w|^2 + |\partial_t w|^2 \right) dx.
\end{aligned}
\end{equation*}
A proof of this can be found in \cite{duyckaerts2011universality} and \cite{duyckaerts2012universality}, where the authors applied this result to study the soliton of focusing energy-critical nonlinear wave equation in dimension $d=3,5$ via some nonlinear analysis on small data solutions. One may also refer to \cite{cote2015characterization1} and \cite{cote2015characterization2} for the application in equivariant wave maps. In \textit{even} dimension $d$, the limits \eqref{eq_partition_of_energy_wave} do not hold in general settings. It is essential to add some extra corrections, which have been calculated in detail for radial solutions in \cite{cote2014energy} and for general data in \cite{côte2021concentration}. In these references, the authors also discovered some special data such that the corrections vanish, an application of which to $4$-dimensional focusing energy-critical wave equation can be found in \cite{cote2018profile}.

The above results have been recently revisited in \cite{delort2022microlocal} using the tools of microlocal analysis. Consider first a solution $u$
of the half-wave equation
\begin{equation}\tag{HW}\label{eq_half_wave}
\left\{
\begin{aligned}
&\left( \frac{\partial_t}{i} - |D_x| \right)u = 0, \\
&u|_{t=0} = u_0 \in L^2.
\end{aligned}\right.
\end{equation}
Since $u = e^{it|D_x|}u_0$, stationary phase formula shows that one expects the microlocalized energy of the solution outside a convenient neighborhood of $\{(x,\xi): x = t\frac{\xi}{|\xi|}\}$ at time $t$ to vanish when $t$ goes to infinity. Because of that, it is natural to ask whether the microlocalized energy close to the preceding point, truncated outside the wave cone, gives rise to lower bound of the form \eqref{eq_partition_of_energy_wave}. More precisely, if one defines this microlocalized truncated energy
as
\begin{gather}
E_{\chi,\tilde{\chi},\delta}^\text{HW}(u_0,t) := \left\| \Op{a^\text{HW}_{\chi,\tilde{\chi},\delta}(t)} u(t) \right\|_{L^2}^2, \label{eq_HW_truncation_energy} \\ 
a^\text{HW}_{\chi,\tilde{\chi},\delta}(t,x,\xi) := \chi\left( \frac{x+t\frac{\xi}{|\xi|}}{|t|^{\frac{1}{2}+\delta}} \right) \tilde{\chi}\left( \frac{|x|-|t|}{|t|^\delta} \right) \mathbbm{1}_{|x|>|t|}, \label{eq_HW_truncation_symbol}
\end{gather}
where $\chi,\tilde{\chi}\in C_c^\infty(\mathbb{R}^d)$ are chosen to be real, radial and equal to $1$ near zero with $\delta\in]0,\frac{1}{2}]$, then it has been proved in \cite{delort2022microlocal} that in any dimension $d$,
\begin{equation}\label{eq_HW_partition_of_energy}
\lim_{t\rightarrow +\infty} \left( E^\text{HW}_{\chi,\tilde{\chi},\delta}(u_0,t) + E^\text{HW}_{\chi,\tilde{\chi},\delta}(u_0,-t) \right) = \|u_0\|_{L^2}^2.
\end{equation}
This result may be used to recover \eqref{eq_partition_of_energy_wave} in odd dimension by taking $u = \left( -i\partial_t + |D_x| \right) w$. Actually, the truncated energy in \eqref{eq_HW_truncation_energy} may be expressed from the microlocalized truncated energy for the solution $u$ of the half wave equation and from extra terms. These extra terms give a zero contribution at the limit $t$ tending to infinity in \emph{odd} dimensions, but not in even ones.

The heuristics underlying estimate \eqref{eq_HW_partition_of_energy} for the half wave equation are as follows. Define the quantization $\Op{a^\text{HW}_{\chi,\tilde{\chi},\delta}}$ of the symbol \eqref{eq_HW_truncation_symbol} by 
$$ \left(\Op{a^\text{HW}_{\chi,\tilde{\chi},\delta}}f\right)(t,x) = \frac{1}{(2\pi)^d} \int e^{ix\cdot\xi} a^\text{HW}_{\chi,\tilde{\chi},\delta}(t,x,\xi) \hat{f}(t,\xi) d\xi. $$
Then \eqref{eq_HW_truncation_energy} may be written as
\begin{equation}\label{eq_HW_truncation_energy_formal_calcul}
\begin{aligned}
E^\text{HW}_{\chi,\tilde{\chi},\delta}(u_0,t) =& \left\langle \Op{a^\text{HW}_{\chi,\tilde{\chi},\delta}}e^{it|D_x|} u_0, \Op{a^\text{HW}_{\chi,\tilde{\chi},\delta}}e^{it|D_x|} u_0 \right\rangle_{L^2} \\
=& \left\langle u_0, e^{-it|D_x|} \Op{a^\text{HW}_{\chi,\tilde{\chi},\delta}}^*\Op{a^\text{HW}_{\chi,\tilde{\chi},\delta}}e^{it|D_x|} u_0 \right\rangle_{L^2}.
\end{aligned}
\end{equation}
If the symbols were smooth ones, so that symbolic calculus (whose details can be found in \cite{zworski2022semiclassical}) could be used, one would expect the composition $\Op{a^\text{HW}_{\chi,\tilde{\chi},\delta}}^*\Op{a^\text{HW}_{\chi,\tilde{\chi},\delta}}$ to be equal, modulo negligible remainders, to $\Op{b}$, with $b=\left|a^\text{HW}_{\chi,\tilde{\chi},\delta}\right|^2$, and the conjugation $e^{-it|D_x|}\Op{b}e^{it|D_x|}$ to be equal, up to remainders, to $\Op{c}$ where $c(x,\xi) = b(x-t\frac{\xi}{|\xi|})$. Applying this to \eqref{eq_HW_truncation_energy_formal_calcul}, one could write this quantity as $\langle u_0,\Op{e}u_0\rangle$ modulo a term tending to zero when $t$ goes to infinity, where $e$ is given by
$$ e(t,x,\xi) = \chi^2\left( \frac{x}{|t|^{\frac{1}{2}+\delta}} \right) \tilde{\chi}^2\left( \frac{|x-t\frac{\xi}{|\xi|}|-|t|}{|t|^\delta} \right) \mathbbm{1}_{\left|x-t\frac{\xi}{|\xi|}\right|>|t|} $$

This symbol $e$ roughly cuts-off the phase space on the domain
$$ \left\{ (x,\xi) : |x|\lesssim |t|^{\frac{1}{2}+\delta},\ \left|x-t\frac{\xi}{|\xi|}\right|>|t| \right\}. $$
When $t\to +\infty$, the truncated domain tends to the half-space
\begin{equation}\label{eq_intro_lim_of_truncation_domain}
	\left\{ (x,\xi) : \operatorname{sgn}(t) x\cdot\frac{\xi}{|\xi|}<0 \right\}.
\end{equation}
As a consequence, the sum of truncated energy at time $t$ and $-t$ covers the whole phase space.

We emphasize that all the arguments above are merely formal. They hold only if all the involved functions are regular enough. In fact, it has also been proved in \cite{delort2022microlocal} that the cut-off operator $\Op{a^\text{HW}_{\chi,\tilde{\chi},\delta}}$ may not even be bounded if the singular cut-off $|x|>|t|$ is replaced by $x\cdot\frac{\xi}{|\xi|}>t$, which seems to work formally.

The above formal point of view, though purely heuristic, makes expect that the classical result \eqref{eq_partition_of_energy_wave} might be extended to a large class of dispersive equations. In general system
$$ \left(\frac{\partial_t}{i} - P(D_x)\right)u = 0, $$
as in half-wave equation, one may expect that the energy concentrates in the phase space around $x+tP'(\xi)=0$ and the partition of energy holds with generalized ``light cone''
$$ |x| = |tP'(\xi)|. $$
The first result for Schrödinger equation has been given in \cite{delort2022microlocal} with truncation
\begin{gather}
E^\text{Schr}_{\chi,\delta}(u_0,t) := \left\| \Op{a^\text{Schr}_{\chi,\delta}(t)} u(t) \right\|_{L^2}^2, \label{eq_Schr_truncation_energy} \\ 
a^\text{Schr}_{\chi,\delta}(t,x,\xi) := \chi\left( \frac{x+t\xi}{|t\xi|\langle\sqrt{|t|}|\xi|\rangle^{-\frac{1}{2}+\delta}} \right) \mathbbm{1}_{|x|>|t\xi|}. \label{eq_Schr_truncation_symbol}
\end{gather}
The result is similar:
\begin{equation}\label{eq_Schr_partition_of_energy}
\lim_{t\rightarrow +\infty} \left( E^\text{Schr}_{\chi,\delta}(u_0,t) + E^\text{Schr}_{\chi,\delta}(u_0,-t) \right) = \frac{1}{2}\|u_0\|_{L^2}^2.
\end{equation}
Here the extra factors in the cut-off $\chi$ is only for technical use, and the loss of half of total energy $\|u_0\|_{L^2}^2$ is due to the convexity of $P(\xi)= \frac{|\xi|^2}{2}$.

The goal of this paper is to examine if the microlocal partition of energy results \eqref{eq_HW_partition_of_energy}, \eqref{eq_Schr_partition_of_energy} may be extended to a large class of dispersive equations. In particular, this generalized result covers the system of linearized gravity or capillary water-wave with infinite depth
\begin{align}
\frac{\partial_t u}{i} - |D_x|^\frac{1}{2} u = 0,\label{eq_lin_gra_ww}\tag{LGWW} \\
\frac{\partial_t u}{i} - |D_x|^\frac{3}{2} u = 0,\label{eq_lin_cap_ww}\tag{LCWW}
\end{align}
or with finite and constant depth $h$,
\begin{align}
&\frac{\partial_t u}{i} - |D_x|^\frac{1}{2}\tanh(h|D_x|) u = 0,\label{eq_lin_gra_ww_finite_depth}\tag{$\text{LGWW}_h$} \\
&\frac{\partial_t u}{i} - |D_x|^\frac{3}{2}\tanh(h|D_x|) u = 0.\label{eq_lin_cap_ww_finite_depth}\tag{$\text{LCWW}_h$}
\end{align}
Moreover, the system of half-Klein-Gordon
\begin{equation}\label{eq_half_KG}\tag{HKG}
\frac{\partial_t u}{i} - \langle D_x \rangle u = 0
\end{equation}
can also be covered by the generalized result, and the associated conclusion will further imply the microlocal partition of energy for standard Klein-Gordon equation
\begin{equation}\tag{KG}\label{eq_Klein-Gordon}
\left\{
\begin{aligned}
&(\partial_t^2 - \Delta + 1)w = 0, \\
&w|_{t=0} = w_0 \in H^1(\mathbb{R}^d), \\
&\partial_t w|_{t=0} = w_1 \in L^2(\mathbb{R}^d).
\end{aligned}
\right.
\end{equation}

One may have noticed that \eqref{eq_HW_partition_of_energy} and \eqref{eq_Schr_partition_of_energy} are proved only for $\delta\in ]0,\frac{1}{2}]$ or $]0,\frac{1}{2}[$. The case $\delta \geqslant \frac{1}{2}$ seems useless since with such $\delta$ the cut-off $\chi$ gives no information in the concentration of energy. The critical case $\delta=0$, however, leads to some interesting results in the limit of truncated energy. The related results will be presented in detail in the next part.

\subsection{Main results}\label{subsect_main_result}

We consider the fractional-type dispersive equation:
\begin{equation}\tag{E}\label{eq_frac_disper}
\left\{
\begin{aligned}
&\left(\frac{\partial_t}{i} - P(D_x)\right)u = 0, \\
&u|_{t=0} = u_0,
\end{aligned}
\right.
\end{equation}
where $P$ is radial and smooth except at $\xi=0$. For simplicity, $P$ will be identified as a function of $\rho=|\xi|$ in what follows. We further assume that $P$ is a fractional-type symbol. Namely, the following hypothesis hold for some $p_0,p_1\neq 0$.
\begin{equation}\tag{$\text{H}_{p_0,p_1}$}\label{hyp_fractional-type_symbol}
\begin{aligned}
&(1)\ P^{(1)} \text{ is strictly positive and monotone on } ]0,+\infty[; \\
&(2.0)\ \exists P_0\geqslant 0,\ \rho\rightarrow 0+,\ \  |P^{(1)}(\rho)-P_0| \sim \rho^{p_0},\ |P^{(2)}(\rho)| \sim \rho^{p_0-1}; \\
&(2.1)\ \exists P_1\geqslant 0,\ \rho\rightarrow +\infty,\ \  |P^{(1)}(\rho)-P_1| \sim \rho^{p_1},\ |P^{(2)}(\rho)| \sim \rho^{p_1-1}; \\
&(3.0)\ \forall j\in\mathbb{N}^*, j\geqslant 3,\ \forall \rho\in]0,1[,\ |P^{(j)}(\rho)| \lesssim \rho^{p_0+1-j};\\
&(3.1)\ \forall j\in\mathbb{N}^*, j\geqslant 3,\ \forall \rho\in]1,\infty[,\ |P^{(j)}(\rho)| \lesssim \rho^{p_1+1-j}.
\end{aligned}
\end{equation}

We introduce the symbol
\begin{equation}\label{eq_truncation_symbol}
a(t,x,\xi) = a_{\chi,\delta}(t,x,\xi) = \chi\left(\frac{x+tP'(\xi)}{|t|^{\frac{1}{2}+\delta}}\right)\mathbbm{1}_{|x|>|t||P'(\xi)|},
\end{equation}
where $\delta\in\mathbb{R}$, $\chi\in C_c^\infty(\mathbb{R}^d)$ is real with $\chi(0)=1$. The corresponding truncated energy is defined as
\begin{equation}\label{eq_truncated_energy}
\begin{aligned}
E(u_0,t) = E_{\chi,\delta}(u_0,t) =& \| \Op{a_{\chi,\delta}(t)}\left(u(t)\right) \|_{L^2}^2 \\
=& \| \Op{a_{\chi,\delta}(t)}\left(e^{itP(D_x)}u_0\right) \|_{L^2}^2,
\end{aligned}
\end{equation}
where
\begin{equation}
\Op{a_{\chi,\delta}(t)}\left(e^{itP(D_x)}u_0\right)(x) := \frac{1}{(2\pi)^d} \int e^{ix\xi+itP(\xi)} a_{\chi,\delta}(t,x,\xi)\hat{u}_0(\xi) d\xi.
\end{equation}

In Proposition \ref{prop_boundedness_of_truncated_operator}, \ref{prop_boundedness_of_truncated_operator_with_truncated_frequency}, and \ref{prop_alt_boundedness_of_truncation_operator}, we shall prove that, under some extra conditions on $P$, operator $\Op{a(t)}$, together with its variations to be introduced later, is bounded on $L^2$, uniformly in $|t| \gg 1$. $E(u_0,t)$ is therefore a well-defined truncated energy, at least when time $|t|$ is sufficiently large.

We first state the fundamental result which is available for fractional equations, namely equation \eqref{eq_frac_disper} with $P'(\xi) = |\xi|^{p}\frac{\xi}{|\xi|}$, such as Schrödinger equation ($p=1$), linearized gravity water-wave equation ($p=-\frac{1}{2}$), and linearized capillary water-wave equation ($p=\frac{1}{2}$).

\begin{theorem}\label{thm_main}
Let $\chi\in C_c^\infty(\mathbb{R}^d)$ be a real function such that $\chi(0)=1$. We further assume that $P$ satisfies hypothesis \eqref{hyp_fractional-type_symbol} with $P_0=P_1=0$.

(\romannumeral1) If $\delta <0$,
\begin{equation}\label{eq_limit_of_energy_subcritical}
\lim_{t\rightarrow\pm\infty} E_{\chi,\delta}(u_0,t) = 0.
\end{equation}

(\romannumeral2) If $\delta =0$,
\begin{equation}\label{eq_limit_of_energy_critical}
\lim_{t\rightarrow\pm\infty} E_{\chi,\delta}(u_0,t) = \frac{1}{(2\pi)^d} \int G_\chi(\rho,\omega) |\hat{u}_0(\rho\omega)|^2 \rho^{d-1} d\rho d\omega,
\end{equation}
where $(\rho,\omega)$ is polar coordinate. The function $G_\chi(\rho,\omega)$ is defined, when $P$ is convex, by
\begin{equation}\label{eq_definition_of_G_convex}
G_\chi(\rho,\omega):=
\frac{1}{(2\pi)^d} \left| \int_{0}^{\infty}\int_{y\cdot\omega=0} e^{i\frac{r^2+|y|^2}{2}} \chi\left( \sqrt{P''(\rho)}r\omega + \sqrt{\rho^{-1}P'(\rho)}y  \right) dydr\right|^{2},
\end{equation}
and, when $P$ is concave, by
\begin{equation}\label{eq_definition_of_G_concave}
G_\chi(\rho,\omega):=
\frac{1}{(2\pi)^d} \left| \int_{0}^{\infty}\int_{y\cdot\omega=0} e^{i\frac{-r^2+|y|^2}{2}} \chi\left( \sqrt{-P''(\rho)}r\omega + \sqrt{\rho^{-1}P'(\rho)}y \right) dydr\right|^{2}.
\end{equation}

(\romannumeral3) If $0<\delta<\frac{1}{2}$, we further assume that $\chi$ is radial. Then
\begin{equation}\label{eq_limit_of_energy_supercritical}
\lim_{t\rightarrow\pm\infty}  E_{\chi,\delta}(u_0,t) = \frac{1}{4} \|u_0\|_{L^2}^2.
\end{equation}
\end{theorem}
\hspace*{\fill}

\begin{remark}
	In \eqref{eq_limit_of_energy_supercritical}, we manage to calculate the limit of $E_{\chi,\delta}(u_0,t)$ and $E_{\chi,\delta}(u_0,-t)$ as $t\rightarrow+\infty$, instead of their sum as in \eqref{eq_partition_of_energy_wave} and \eqref{eq_Schr_partition_of_energy}. Notice that the heuristics discussed after \eqref{eq_intro_lim_of_truncation_domain} in the case of half-wave equations do not predict this fact. This shows the limitation of this formal reasoning when sharp cut-offs are involved in the symbols.
\end{remark}

\begin{remark}
For Schrödinger equation, the special structure of $P'(\xi)=\xi$ allows us to reduce the regularity required for $\chi$. In Appendix \ref{sect_Schrodinger}, we will show that limits \eqref{eq_limit_of_energy_subcritical} and \eqref{eq_limit_of_energy_critical} hold for all $\chi\in L^1$.
\end{remark}

We will see later that our proof of Theorem \ref{thm_main} does not hold for $P$ with nonzero $P_0,P_1$, such as half-Klein-Gordon equation \eqref{eq_half_KG}, where $P_1=1$. In order to deal with this difficulty, one way is to add some cut-off in frequency $\xi$. To be precise, we introduce the modified truncated symbol
\begin{equation}\label{eq_mod_truncation_symbol}
a^{\text{mod}}(t,x,\xi) = a^{\text{mod}}_{\chi,\delta}(t,x,\xi) = \chi\left(\frac{x+tP'(\xi)}{|t|^{\frac{1}{2}+\delta}}\right)\mathbbm{1}_{|x|>|t||P'(\xi)|} (1-\chi_l)\left( \frac{\xi}{|t|^{-\epsilon_0}} \right) \chi_h\left( \frac{\xi}{|t|^{\epsilon_1}} \right),
\end{equation}
where $\epsilon_0,\epsilon_1$ satisfies
\begin{equation}\label{eq_def_of_epsilons}
0<\epsilon_0 \leqslant \frac{1}{p_0+1},\ \ 
0<\epsilon_1 \leqslant \left\{
\begin{aligned}
&+\infty,&\text{if } -1\leqslant p_1 <0,\\
&\frac{1}{-(p_1+1)},&\text{if } p_1<-1,
\end{aligned}\right.
\end{equation}
and $\chi_l,\chi_h \in C_c^\infty$ are radial and equal to $1$ near zero. This symbol will be concerned only when $p_1<0<p_0$, the reason of which will be explained later. The corresponding truncated energy is denoted by
$$ E^\text{mod}(u_0,t) = E^\text{mod}_{\chi,\delta}(u_0,t) := \|\Op{a^{\text{mod}}_{\chi,\delta}(t)}u(t)\|_{L^2}^2. $$

\begin{theorem}\label{thm_mod_main}
Let $\chi\in C_c^\infty(\mathbb{R}^d)$ be a real function such that $\chi(0)=1$. We further assume that $P$ satisfies hypothesis \eqref{hyp_fractional-type_symbol} with $p_1<0<p_0$, $P_0,P_1>0$, and $\epsilon_0,\epsilon_1$ satisfy condition \eqref{eq_def_of_epsilons}.

(\romannumeral1) If $\delta <0$,
\begin{equation}\label{eq_mod_limit_of_energy_subcritical}
\lim_{t\rightarrow\pm\infty} E^\text{mod}_{\chi,\delta}(u_0,t) = 0.
\end{equation}

(\romannumeral2) If $\delta =0$,
\begin{equation}\label{eq_mod_limit_of_energy_critical}
\lim_{t\rightarrow\pm\infty} E^\text{mod}_{\chi,\delta}(u_0,t) = \frac{1}{(2\pi)^d} \int G_\chi(\rho,\omega) |\hat{u}_0(\rho\omega)|^2 \rho^{d-1} d\rho d\omega,
\end{equation}
where $(\rho,\omega)$ is polar coordinate. The function $G_\chi(\rho,\omega)$ is the same one defined by \eqref{eq_definition_of_G_convex} and \eqref{eq_definition_of_G_concave}.

(\romannumeral3) If $0<\delta<\frac{1}{2}$, we further assume that $\chi$ is radial. Then
\begin{equation}\label{eq_mod_limit_of_energy_supercritical}
\lim_{t\rightarrow\pm\infty} E^\text{mod}_{\chi,\delta}(u_0,t) = \frac{1}{4} \|u_0\|_{L^2}^2.
\end{equation}
\end{theorem}

\begin{remark}
In this theorem, we only consider the case $p_1<0<p_0$, which is enough to cover all $P_0,P_1 \neq 0$. Actually, when $p_0<0$ (resp. $p_1>0$), the hypothesis \eqref{hyp_fractional-type_symbol} with $P_0>0$ (resp. $P_1>0$) is equivalent to \eqref{hyp_fractional-type_symbol} with $P_0=0$ (resp. $P_1=0$), which has been studied in Theorem \ref{thm_main}.
\end{remark}

Another way to deal with nonzero $P_0,P_1$ is to add extra factor in the cut-off $\chi$, namely, considering an alternative truncated symbol
\begin{equation}\label{eq_alt_truncation_symbol}
a^\text{alt}(t,x,\xi) = a^\text{alt}_{\chi,\delta,\Lambda}(t,x,\xi) = \chi\left(\frac{x+tP'(\xi)}{|t|^{\frac{1}{2}+\delta}\Lambda(|t|^{\frac{1}{2}}\xi)}\right)\mathbbm{1}_{|x|>|t||P'(\xi)|},
\end{equation}
where $\Lambda \in C^\infty(\mathbb{R}^d\backslash\{0\})$ is strictly positive, radial and satisfies the following condition for some $\sigma_0,\sigma_1\in\mathbb{R}$:
\begin{equation}\tag{$\text{C}_{\sigma_0,\sigma_1}$}\label{hyp_fractional-type_symbol_extra_factor}
\begin{aligned}
&(1.0)\ \rho\rightarrow 0+,\ \Lambda(\rho)\sim \rho^{\sigma_0}; \\
&(1.1)\ \rho\rightarrow +\infty,\ \Lambda(\rho)\sim \rho^{\sigma_1}; \\
&(2.0)\ \forall \rho\in]0,1[,\ \Lambda^{(j)}(\rho)\lesssim \rho^{\sigma_0-j}; \\
&(2.1)\ \forall \rho\in]1,\infty[,\ \Lambda^{(j)}(\rho)\lesssim \rho^{\sigma_1-j}; \\
&(3)\lim_{\rho\rightarrow +\infty} \frac{\Lambda(\rho)}{\rho^{\sigma_1}} = \lambda_1>0.
\end{aligned}
\end{equation}
The associated truncated energy is denoted by
$$ E^\text{alt}(u_0,t) = E^\text{alt}_{\chi,\delta,\Lambda}(u_0,t) := \|\Op{a^{\text{alt}}_{\chi,\delta,\Lambda}(t)}u(t)\|_{L^2}^2, $$
and the result becomes

\begin{theorem}\label{thm_alt_main}
Let $\chi\in C_c^\infty(\mathbb{R}^d)$ be a real function such that $\chi(0)=1$. We assume that $P$ satisfies hypothesis \eqref{hyp_fractional-type_symbol} and $\Lambda$ satisfies condition \eqref{hyp_fractional-type_symbol_extra_factor} with $\sigma_0\geqslant p_0,\sigma_1\leqslant p_1$.

(\romannumeral1) If $\delta+\frac{\sigma_1}{2}<0$,
\begin{equation}\label{eq_alt_limit_of_energy_subcritical}
\lim_{t\rightarrow\pm\infty} E^\text{alt}_{\chi,\delta,\Lambda}(u_0,t) = 0.
\end{equation}

(\romannumeral2) If $\delta+\frac{\sigma_1}{2}=0$,
\begin{equation}\label{eq_alt_limit_of_energy_critical}
\lim_{t\rightarrow\pm\infty} E^\text{alt}_{\chi,\delta,\Lambda}(u_0,t) = \frac{1}{(2\pi)^d} \int G^\text{alt}_\chi(\rho,\omega) |\hat{u}_0(\rho\omega)|^2 \rho^{d-1} d\rho d\omega,
\end{equation}
where $(\rho,\omega)$ is polar coordinate. The function $G^\text{alt}_\chi(\rho,\omega)$ is defined, when $P$ is convex, by
\begin{equation}\label{eq_alt_definition_of_G_convex}
G^\text{alt}_\chi(\rho,\omega):=
\frac{1}{(2\pi)^d} \left| \int_{0}^{\infty}\int_{y\cdot\omega=0} e^{i\frac{r^2+|y|^2}{2}} \chi\left( \frac{\sqrt{P''(\rho)}r\omega + \sqrt{\rho^{-1}P'(\rho)}y}{\lambda_1\rho^{\sigma_1}}  \right) dydr\right|^{2},
\end{equation}
and, when $P$ is concave, by
\begin{equation}\label{eq_alt_definition_of_G_concave}
G^\text{alt}_\chi(\rho,\omega):=
\frac{1}{(2\pi)^d} \left| \int_{0}^{\infty}\int_{y\cdot\omega=0} e^{i\frac{-r^2+|y|^2}{2}} \chi\left( \frac{\sqrt{-P''(\rho)}r\omega + \sqrt{\rho^{-1}P'(\rho)}y}{\lambda_1\rho^{\sigma_1}} \right) dydr\right|^{2}.
\end{equation}

(\romannumeral3) If $0<\delta+\frac{\sigma_1}{2}<\frac{1}{2}$, we further assume that $\chi$ is radial. Then
\begin{equation}\label{eq_alt_limit_of_energy_supercritical}
\lim_{t\rightarrow\pm\infty} E^\text{alt}_{\chi,\delta,\Lambda}(u_0,t) = \frac{1}{4} \|u_0\|_{L^2}^2.
\end{equation}
\end{theorem}
\begin{remark}
The proof of Theorem \ref{thm_main} fails when $|x/t|$ is close to $P_0$ or $P_1$. The extra factor $\Lambda$ together with condition $\sigma_0\geqslant p_0,\sigma_1\leqslant p_1$ allows us to eliminate this case, and a demonstration similar to Theorem \ref{thm_main} will work when $|x/t|$ is away from $P_0$, $P_1$.
\end{remark}

As a byproduct of Proposition \ref{prop_boundedness_of_truncated_operator} and \ref{prop_boundedness_of_truncated_operator_with_truncated_frequency}, where the uniform-in-$t$ boundedness on $L^2$ of $\Op{a(t)}$ and $\Op{a^\text{mod}(t)}$ will be proved, these operators are also uniformly bounded with $\chi$ identically equal to $1$, namely 
\begin{theorem}\label{thm_boundedness_of_simple_truncated_operator}
Let $p_0,p_1\neq 0$ and $P$ satisfy the hypothesis \eqref{hyp_fractional-type_symbol}. There exists a constant $C>0$ independent of $t$, such that, 

(\romannumeral1) when $P_0=P_1=0$, 
$$ \|\Op{\mathbbm{1}_{|x|>|tP'(\xi)|}}\|_{\mathcal{L}(L^2)} \leqslant C $$
holds for all $|t|>0$;

(\romannumeral2) when $P_0,P_1>0$,
$$ \|\Op{\mathbbm{1}_{|x|>|tP'(\xi)|} (1-\chi_l)\left( \frac{\xi}{|t|^{-\epsilon_0}} \right) \chi_h\left( \frac{\xi}{|t|^{\epsilon_1}} \right) }\|_{\mathcal{L}(L^2)} \leqslant C $$
holds for all $|t|>t_0 \gg 1$, where $\epsilon_0,\epsilon_1$ are arbitrary parameters satisfying \eqref{eq_def_of_epsilons} and $\chi_l,\chi_h \in C_c^\infty$ are equal to $1$ near zero.
\end{theorem}
The boundedness of $\Op{\mathbbm{1}_E}$ for measurable sets $E\subset\mathbb{R}^{2d}$ is of great concern in micorlocal analysis, and the results above give a positive answer to some $E$ defined via convex functions. If one changes to Weyl quantization, this problem is known as localization of Wigner distribution. In fact, in this case, we have
$$ \langle \Op[w]{\mathbbm{1}_E}u,v \rangle_{L^2} = \frac{1}{(2\pi)^{\frac{d}{2}}} \int_{E}\mathcal{W}(u,v)(x,\xi)dxd\xi,  $$
where
$$ \mathcal{W}(u,v)(x,\xi) = \frac{1}{(2\pi)^{\frac{d}{2}}} \int_{\mathbb{R}^d} e^{iy\cdot\xi} u(x+\frac{y}{2}) \overline{v(x-\frac{y}{2})} dy $$
is the Wigner distribution of $(u,v)$. It has been found that the operator properties of $\Op[w]{\mathbbm{1}_E}$ (boundedness, positivity, spectrum, etc) is related to the geometry of $E$. For example, when $E$ is an ellipsoid, in \cite{Flandrin1988MaximumSE} and \cite{lieb2010localization}, the authors gave some sharp estimates of the $L^2$-norm of $\Op[w]{\mathbbm{1}_E}$, which is related to the size of the ellipsoid. As another example, when $E$ is a polygon on $\mathbb{R}^2$ with $N$ sides, it was proved in \cite{lerner2023integrating} that the norm of $\Op[w]{\mathbbm{1}_E}$ can be controlled by $\sqrt{N/2}$ for $N\geqslant 3$. In the same paper, the author also proved that there exists open set $E$ such that $\Op[w]{\mathbbm{1}_E}$ is not even bounded on $L^2$. The readers may refer to \cite{lerner2023integrating} for more results on this topic. \\

By applying the results of Theorem \ref{thm_mod_main} and \ref{thm_boundedness_of_simple_truncated_operator} to half-Klein-Gordon equation \eqref{eq_half_KG}, we are able to obtain the following microlocal partition of energy for Klein-Gordon equation, which is an analogue of partition of energy for wave equation.

\begin{theorem}\label{thm_K-G_main}
Let $w$ be the unique solution to Klein-Gordon equation \eqref{eq_Klein-Gordon}, namely
\begin{equation*}
\left\{
\begin{aligned}
&(\partial_t^2 - \Delta + 1)w = 0, \\
&w|_{t=0} = w_0 \in H^1(\mathbb{R}^d), \\
&\partial_t w|_{t=0} = w_1 \in L^2(\mathbb{R}^d),
\end{aligned}
\right.
\end{equation*}
where $w_0,w_1$ are real, so as $w$. Then, the truncated energy 
\begin{equation}\label{eq_K-G_def_truncated_energy}
\begin{aligned}
E^{KG}_{\epsilon}(w_0,w_1,t) :=& \|\Op{a_\epsilon^{KG}(t)}\partial_t w(t)\|_{L^2}^2 + \|\Op{a_\epsilon^{KG}(t)}\nabla w(t)\|_{L^2}^2 \\
& + \|\Op{a_\epsilon^{KG}(t)} w(t)\|_{L^2}^2
\end{aligned}
\end{equation}
satisfies
\begin{equation}\label{eq_K-G_limit_of_energy}
\lim_{t\rightarrow\pm\infty}  E^{KG}_{\epsilon}(w_0,w_1,t) = \frac{1}{4} \left( \|w_0\|_{H^1}^2 + \|w_1\|_{L^2}^2 \right),
\end{equation}
where
\begin{equation}\label{eq_K-G_def_truncated_symbol}
a^{KG}(t,x,\xi) = a_\epsilon^{KG}(t,x,\xi) := \mathbbm{1}_{|x|>\left|t\frac{\xi}{\langle\xi\rangle}\right|} \chi\left(\frac{\xi}{|t|^\epsilon}\right),
\end{equation}
$0<\epsilon<1$, and $\chi\in C^\infty_c(\mathbb{R}^d)$ is a real and radial function equal to $1$ near zero.
\end{theorem}

\begin{remark}
In view of the similarity between wave equation and Klein-Gordon equation, one may ask what \eqref{eq_K-G_limit_of_energy} will become if we apply the same truncation $\mathbbm{1}_{|x|>|t|}$ as in classical result \eqref{eq_partition_of_energy_wave}. The answer is, for all $0\leqslant r_0 \leqslant r_1$,
\begin{equation}\label{eq_K-G_limit_of_energy_classical}
\lim_{t\rightarrow \pm\infty} \int_{r_0<\left|\frac{x}{t}\right|<r_1} \left( |\partial_t w|^2 + |\nabla w|^2 + |w|^2 \right) dx = \| \mathbbm{1}_{]\rho_0,\rho_1[}(|D_x|) w_0\|_{H^1}^2 + \| \mathbbm{1}_{]\rho_0,\rho_1[}(|D_x|) w_1\|_{L^2}^2,
\end{equation}
where $]\rho_0,\rho_1[ = P'^{-1}(]r_0,r_1[)$. Since $P'$ takes value in $[0,1[$, we have in particular,
$$ \lim_{t\rightarrow \pm\infty} \int_{|x|>|t|} \left( |\partial_t w|^2 + |\nabla w|^2 + |w|^2 \right) dx = 0. $$
A detailed discussion of \eqref{eq_K-G_limit_of_energy_classical} will be given in Section \ref{sect_K-G_classical}.
\end{remark}

\subsection{Non-nullity of the limit in critical case}

In Theorem \ref{thm_main}, \ref{thm_mod_main}, and \ref{thm_alt_main}, we calculate the limit of energy in three cases. In sub-critical case $\delta<0$ (or $\delta+\frac{\sigma_1}{2}<0$), the truncated energy tends to $0$, no matter which $\chi$ we choose. This phenomenon also exists in super-critical case $0<\delta<\frac{1}{2}$ (or $0<\delta+\frac{\sigma_1}{2}<\frac{1}{2}$), where the limit is always half of total energy $\|u_0\|_{L^2}$. In critical case $\delta=0$ (or $\delta+\frac{\sigma_1}{2}=0$), however, the limit does depend on our choice of $\chi$.  If we further assume $\chi$ to be radial, it is not difficult to check that the limits \eqref{eq_limit_of_energy_critical}, \eqref{eq_mod_limit_of_energy_critical}, and \eqref{eq_alt_limit_of_energy_critical} are bounded and non-negative.

In fact, when $\chi$ is radial, the function $G_\chi$, $G^\text{alt}_\chi$ can be written in the form
$$\frac{1}{4}\frac{1}{(2\pi)^d} \left| \int_{\mathbb{R}^d} e^{i\frac{\pm x_1^2+|x'|^2}{2}} \chi\left( \frac{\sqrt{\pm P''(\rho)}x_1 + \sqrt{P'(\rho)/\rho}x'}{\lambda\rho^\sigma}  \right) dx\right|^{2},$$
where $+$ and $-$ stand for convex and concave case, respectively, $x=(x_1,x')\in\mathbb{R}\times\mathbb{R}^{d-1}$, and $\lambda>0$, $\sigma\in\mathbb{R}$. By Plancherel Theorem, it equals to
$$\frac{1}{4}\frac{1}{(2\pi)^{2d}} \left| \int_{\mathbb{R}^d} e^{i\frac{\pm \xi_1^2+|\xi'|^2}{2(\lambda\rho^\sigma)^2}} \hat{\chi}\left( \frac{\xi_1}{\sqrt{\pm P''(\rho)}} + \sqrt{\frac{\rho}{P'(\rho)}}\xi' \right) \frac{1}{\sqrt{\pm P''(\rho)}} \left(\frac{\rho}{P'(\rho)}\right)^{\frac{d-1}{2}}  d\xi\right|^{2},$$
which, after a change of variable, reads
\begin{equation}\label{eq_remark_limit}
\frac{1}{4}\frac{1}{(2\pi)^{2d}} \left| \int_{\mathbb{R}^d} e^{i\frac{P''(\rho)\rho \xi_1^2 + P'(\rho)|\xi'|^2}{2\rho(\lambda\rho^\sigma)^2}} \hat{\chi}\left( \xi \right) d\xi\right|^{2}.
\end{equation}
Therefore, $G_\chi$ can be estimated by
$$ 0 \leqslant G_\chi(\rho,\omega) \leqslant \frac{1}{4}\frac{1}{(2\pi)^{2d}} \|\hat{\chi}\|_{L^1}^2. $$

A natural question is then whether limits \eqref{eq_limit_of_energy_critical}, \eqref{eq_mod_limit_of_energy_critical}, and \eqref{eq_alt_limit_of_energy_critical} are nonzero for nontrivial initial data $u_0$. The answer is positive for fractional equation, i.e. with $P'(\xi) = |\xi|^{p-1} \xi$, $p \neq 0$. More precisely,

\begin{proposition}\label{prop_nullity_of_critical_limit}
Under the assumption $P'(\xi) = |\xi|^{p-1} \xi$, \eqref{eq_remark_limit} can be written, up to some multiple with constants, as
$$ \tilde{G}(\rho) = \left| \int_{\mathbb{R}^d} e^{i \frac{1}{2\lambda^2}\rho^{p-1-2\sigma}(p\xi_1^2+|\xi'|^2)} \hat{\chi}\left( \xi \right) d\xi \right|^2. $$
If $p \neq 2\sigma+1$ and $\chi\in\mathcal{S}(\mathbb{R}^d)$ with $\chi(0)\neq 0$, $\tilde{G}(\rho)$ is nonzero except on a set of null Lebesgue measure.
\end{proposition}

\begin{proof}
Since $\chi$ is a Schwartz function, the complex function
$$ F(z) := \frac{1}{2(2\pi)^{d}} \int_{\mathbb{R}^d} e^{i\frac{z}{2\lambda^2}(\frac{p}{2}\xi_1^2+\frac{1}{2}|\xi'|^2)} \hat{\chi}\left( \xi \right) d\xi $$
is analytic on upper half plane $\{z\in\mathbb{C}:\Imaginary z>0 \}$ and continuous on its closure. In \cite{Lusin1925}, the authors proved that either the real zeros of such function form a set of zero Lebesgue measure, or it is identically zero. The same result holds thus for $\tilde{G}(\rho) = |F(\rho^{p-1-2\sigma})|^2$. Due to the fact that $\chi(0) \neq 0$, $\tilde{G}$ is nonzero as $\rho^{p-1-2\sigma_1}$ is small enough, and $\tilde{G}$ is therefore nonzero almost everywhere.
\end{proof}

As a consequence, the limits \eqref{eq_limit_of_energy_critical}, \eqref{eq_mod_limit_of_energy_critical}, and \eqref{eq_alt_limit_of_energy_critical} are strictly positive for all nontrivial $u_0\in L^2$ under the assumption $p\neq 1$ (or $p\neq 2\sigma_1+1$). If $p= 1$ (or $p= 2\sigma_1+1$), the function $G_\chi(\rho)$ (or $G_\chi^\text{alt}$) will no more depend on $\rho$ and the limits \eqref{eq_limit_of_energy_critical}, \eqref{eq_mod_limit_of_energy_critical}, \eqref{eq_alt_limit_of_energy_critical} will take the form $ c_0(\chi)\|u_0\|_{L^2}^2 $, where
$$ c_0(\chi) = \left| \frac{1}{2(2\pi)^{d}} \int_{\mathbb{R}^d} e^{i(\frac{p}{2}\xi_1^2+\frac{1}{2}|\xi'|^2)} \hat{\chi}\left( \xi \right) d\xi\right|^{2}. $$
To obtain a nonzero limit, it suffices to choose $\chi$ such that the quantity above is nonzero. For example, one may take

-$\chi$ to be positive and supported in a sufficiently small ball centered at zero;

-$\chi$ to be Gaussian;

-$\chi$ of the form $\chi = \tilde{\chi}(\cdot/R)$, with $\tilde{\chi}\in C_c^\infty$, $\tilde{\chi}(0) \neq 0$, and $R\gg 1$.

\subsection{Plan of this paper}

The proof of Theorem \ref{thm_main}, \ref{thm_mod_main}, and \ref{thm_alt_main} will be divided into two parts: uniform boundedness of truncated operator and calculation of limit. In Section \ref{sect_L^2_bound}, we will prove that $\Op{a(t)}$ and $\Op{a^\text{mod}(t)}$ are uniformly bounded on $L^2$ in three steps. The first two steps are exactly the same, while the difference arises in the last step where one may see the difficulties caused by nonzero $P_0,P_1$. As a byproduct of this proof, Theorem \ref{thm_boundedness_of_simple_truncated_operator} can be shown easily. Section \ref{sect_L^2_bound_alt} is devoted to the uniform boundedness of $\Op{a^\text{alt}(t)}$, which is much simpler than that of $\Op{a(t)}$ and $\Op{a^\text{mod}(t)}$ thanks to the extra factor $\Lambda$. The uniform boundedness of truncated operators allows us to calculate the limits stated in Theorem \ref{thm_main}, \ref{thm_mod_main}, and \ref{thm_alt_main} only for some regular data $u_0$, which will be precised in Section \ref{sect_limit_of_energy}. In Section \ref{sect_Klein-Gordon}, we will prove the microlocal partition of energy for Klein-Gordon equation by studying the half-Klein-Gordon equation.

In Appendix \ref{sect_lemmas}, we collect technical inequalities which are frequently used in this paper, as well as some criteria of $L^2$-boundedness for pseudo-differential operators. Several stationary phase lemmas are presented in Appendix \ref{sect_stationary_phase} which is a key technique in calculating the limit of truncated energy. As mentioned before, our main result Theorem \ref{thm_main} can be refined for Schrödinger equation, whose rigorous statement and proof will be given in Appendix \ref{sect_Schrodinger}. Appendix \ref{sect_K-G_classical} is devoted to the discussion on the classical partition of energy for Klein-Gordon equation due to the study of asymptotic behavior of solution to half-Klein-Gordon equation. The last part, Appendix \ref{sect_proof_of_limit_of_integral}, contains some details omitted in Section \ref{sect_limit_of_energy}, especially for concave $P$.

\subsection{Notations and conventions}

To end this section, we clarify some notations and conventions used in this paper.

- We say $a$ is a symbol on $\mathbb{R}^d$, if $a$ is a function on $\{(x,\xi)\in \mathbb{R}^d\times\mathbb{R}^d\}$. The corresponding (pseudo-differential) operator is defined by
$$ \Op{a}f(x) := \frac{1}{(2\pi)^d} \int e^{ix\xi} a(x,\xi) \hat{f}(\xi) d\xi. $$
To make this definition meaningful, we will assume in this paper that $a$ is a measurable function with at most polynomial growth in $\xi$ and that $f$ belongs to the class of Schwartz functions.

- For any function $P:\mathbb{R}^d\rightarrow\mathbb{C}$, which can be regarded as a symbol independent of $x$, the corresponding operator will be denoted by $P(D_x)$.

- The kernel (or kernel function) of a linear operator $A: \mathcal{S}(\mathbb{R}^d) \mapsto \mathcal{S}'(\mathbb{R}^d)$ is defined as (if it exists) a tempered distribution $K$ on $\mathbb{R}^d\times\mathbb{R}^d$, such that
$$ Au(x) = \int K(x,y)u(y) dy. $$
For the simplicity of notation, in this paper, we will use the symbol of kernel function to represent the operator, i.e.
$$ Ku(x) = \int K(x,y)u(y) dy. $$

- A function $F:\mathbb{R}^d \mapsto \mathbb{C}$ is said to be radial, if there exists a function $f:[0,\infty[ \mapsto \mathbb{C}$, such that $F(x) = f(|x|)$ for all $x\in \mathbb{R}^d$. In this case, we will not distinguish function $F$ and $f$. That is, we will write instead $F(x) = F(|x|)$ or $f(x) = f(|x|)$ for $x\in\mathbb{R}^d$.

- We will use $c$, $C$, sometimes equipped with superscripts and subscripts, to represent all the small and large constants respectively.

- For non-zero quantities $\rho$, $r$, the notation $\rho \sim r$ means that there exists constants $c,C>0$, such that $c<\frac{\rho}{r}<C$.

%% file: main/L2_bound.tex
\section{\texorpdfstring{$L^{2}$}{L2}-boundedness of microlocal truncation operators}\label{sect_L^2_bound}

The goal of this section is the demonstration of following proposition, which eventually implies Theorem \ref{thm_boundedness_of_simple_truncated_operator}.

\begin{proposition}\label{prop_boundedness_of_truncated_operator}
Let $p_0,p_1\neq 0$, $\delta \in \mathbb{R}$, and $\chi\in C_c^\infty(\mathbb{R}^d)$. There exists a constant $C>0$ independent of $t$, such that, for all $|t|>0$,
$$ \|\Op{a_{\chi,\delta}(t)}\|_{\mathcal{L}(L^2)} \leqslant C, $$
where the symbol $a_{\chi,\delta}(t)$ is defined in \eqref{eq_truncation_symbol} and $P$ satisfies the hypothesis \eqref{hyp_fractional-type_symbol} with $P_0=P_1=0$.
\end{proposition}

In parallel, we shall also prove the following result:
\begin{proposition}\label{prop_boundedness_of_truncated_operator_with_truncated_frequency}
Let $p_0>0>p_1$, $\delta \in \mathbb{R}$, and $\chi\in C_c^\infty(\mathbb{R}^d)$. We assume that $P$ satisfies the hypothesis \eqref{hyp_fractional-type_symbol} with $P_0,P_1>0$. Then the modified truncated symbol
\begin{equation*}
a^{\text{mod}}_{\chi,\delta}(t,x,\xi) = a_{\chi,\delta}(t,x,\xi) (1-\chi_l)\left( \frac{\xi}{|t|^{-\epsilon_0}} \right) \chi_h\left( \frac{\xi}{|t|^{\epsilon_1}} \right),
\end{equation*}
which has already been defined in \eqref{eq_mod_truncation_symbol}, corresponds to a bounded operator on $L^2$, uniformly in $|t|>t_0 \gg 1$. Here $\chi_l,\chi_h \in C_c^\infty$ are radial and equal to $1$ near zero and $\epsilon_0,\epsilon_1$ satisfy the condition \eqref{eq_def_of_epsilons}.
\end{proposition}

One can see in the following proof that our demonstration cannot eliminate the truncation in $\xi$ in the definition \eqref{eq_mod_truncation_symbol} of $a^\text{mod}$. In fact, after some change of scaling, we will decompose the symbol $a$ (or $a^\text{mod}$) into three components, two of which are bounded for all $P_0,P_1 \geqslant 0$, while our treatment for the last component does not hold for nonzero $P_0,P_1$. The complementary cut-off in $\xi$ is used to solve this problem. Remark that it is still unknown whether such restriction is essential. \\

To begin with, one observes that, it is equivalent to study the cut-off inside the cone, namely 
$$ a^{\text{in}}(t,x,\xi) = \chi\left(\frac{x+tP'(\xi)}{|t|^{\frac{1}{2}+\delta}}\right)\mathbbm{1}_{|x|<|t||P'(\xi)|}, $$
since the operator with symbol
$$ \chi\left(\frac{x+tP'(\xi)}{|t|^{\frac{1}{2}+\delta}}\right) $$
is bounded uniformly in $t$ and $\delta\in\mathbb{R}$, due to Lemma \ref{lem_tech_boundedness_of_symbol_study_in_xi} together with Lemma \ref{lem_tech_change_of_x_and_xi}. In this section, we will distinguish $a$ and $a^\text{in}$ and denote both of them as $a$.

With a reflection in $\xi$, $t$ can be assumed to be positive. The application of Lemma \ref{lem_tech_change_of_scaling} allows us to replace $a$ by
$$ \tilde{a}(t,x,\xi) = a(t,\sqrt{t}x,\frac{\xi}{\sqrt{t}}) = \chi\left(\frac{\frac{x}{\sqrt{t}}+P'\left(\frac{\xi}{\sqrt{t}}\right)}{t^{\delta-\frac{1}{2}}}\right)\mathbbm{1}_{\left|\frac{x}{\sqrt{t}}\right|<\left|P'\left(\frac{\xi}{\sqrt{t}}\right)\right|}. $$
Now, we split $\tilde{a}$ into high and low frequency, namely $\tilde{a} = \tilde{a}^\flat + \tilde{a}^\sharp$, where
$$ \tilde{a}^\flat(t,x,\xi) = \tilde{a}(t,x,\xi) \tilde{\chi}(\xi), $$
and $\tilde{\chi}\in C_c^\infty(\mathbb{R}^d)$ is a radial function which equals $1$ near zero. In the following, we shall treat high and low frequency part at the same time. Before entering the next step, we introduce some notations which will be frequently used in this section. In all cases, we set 
$$ \mu=t^{\delta-\frac{1}{2}} \in ]0,+\infty[. $$
With $j=0$ for low frequency part and $j=1$ for high frequency part, we set

-when $P_j=0$,
$$ X(t,x):= \frac{|x|}{\sqrt{t}},\ \ \Xi(t,\xi):= P'\left(\frac{|\xi|}{\sqrt{t}}\right),\ \ \nu_j=+; $$

-when $P_j>0$, $P'>P_j$
$$ X(t,x):= \frac{|x|}{\sqrt{t}}-P_j,\ \ \Xi(t,\xi):= P'\left(\frac{|\xi|}{\sqrt{t}}\right)-P_j,\ \ \nu_j=+; $$

-when $P_j>0$, $P'<P_j$
$$ X(t,x):= P_j-\frac{|x|}{\sqrt{t}},\ \ \Xi(t,\xi):= P_j-P'\left(\frac{|\xi|}{\sqrt{t}}\right)\ \ \nu_j=-. $$

Remark that for all nonzero $\xi$, $\Xi$ is strictly positive. With these notations, our problem can be reduced to the uniform-in-$\mu,t$ boundedness of
\begin{equation*}
\begin{aligned}
b^\flat(t,\mu,x,\xi) &= \chi\left(\frac{(P_j+\nu_j X)\frac{x}{|x|}+(P_j+\nu_j\Xi)\frac{\xi}{|\xi|}}{\mu}\right)\mathbbm{1}_{0<\frac{X}{\Xi}<1} \tilde{\chi}\left(\frac{\xi}{\sqrt{t}}\right), \\
b^\sharp(t,\mu,x,\xi) &= \chi\left(\frac{(P_j+\nu_j X)\frac{x}{|x|}+(P_j+\nu_j\Xi)\frac{\xi}{|\xi|}}{\mu}\right)\mathbbm{1}_{0<\frac{X}{\Xi}<1} (1-\tilde{\chi})\left(\frac{\xi}{\sqrt{t}}\right).
\end{aligned}
\end{equation*}
We emphasize that our definition of $X$ does not ensure its strict positivity, but one may always eliminate the part $X<0$, due to the uniform boundedness of operator with symbol
$$ \chi\left(\frac{(P_j+\nu_j X)\frac{x}{|x|}+(P_j+\nu_j\Xi)\frac{\xi}{|\xi|}}{\mu}\right), $$
which is also a consequence of Lemma \ref{lem_tech_boundedness_of_symbol_study_in_xi} and Lemma \ref{lem_tech_change_of_x_and_xi}.

Now, we decompose $b^\iota$ ($\iota=\flat,\sharp$) as the sum of $b_0^\iota,\tilde{b}^\iota,b_1^\iota$ with cut-off $0<\frac{X}{\Xi}\ll 1$ and $\frac{X}{\Xi} \sim 1$ and $0<1-\frac{X}{\Xi}\ll 1$, respectively. To be precise,
\begin{equation*}
\begin{aligned}
b^\iota(t,\mu,x,&\xi) = b_0^\iota(t,\mu,x,\xi) + \tilde{b}^\iota(t,\mu,x,\xi) + b_1^\iota(t,\mu,x,\xi), \\
&b_0^\iota(t,\mu,x,\xi) = b^\iota(t,\mu,x,\xi) \chi_0\left(\frac{X}{\Xi}\right), \\
&\tilde{b}^\iota(t,\mu,x,\xi) = b^\iota(t,\mu,x,\xi) \Psi\left(\frac{X}{\Xi}\right), \\
&b_1^\iota(t,\mu,x,\xi) = b^\iota(t,\mu,x,\xi) \chi_1\left(1-\frac{X}{\Xi}\right), \\
\end{aligned}
\end{equation*}
where $\chi_0,\chi_1$ and $\Psi$ are radial, smooth and compactly supported. $\chi_0,\chi_1$ are supported in a small neighborhood of zero and equal to $1$ near zero, while $\Psi$ is compactly supported in $]0,1[$. By regarding $\mu$ as a $t$-independent parameter, we can reduce Proposition \ref{prop_boundedness_of_truncated_operator} and \ref{prop_boundedness_of_truncated_operator_with_truncated_frequency} to the following proposition:
\begin{proposition}\label{prop_bdd_main}
There exists $t,\mu$-independent constants $C>0$, such that,

(\romannumeral1) if $P_0,P_1 \geqslant 0$, for all $t,\mu>0$,
\begin{align}
&\|\Op{b_1^\iota(t,\mu)}\|_{\mathcal{L}(L^2)} \leqslant C, \label{eq_bdd_boundedness_of_b_1} \\
&\|\Op{\tilde{b}^\iota(t,\mu)}\|_{\mathcal{L}(L^2)} \leqslant C, \label{eq_bdd_boundedness_of_tilde{b}}
\end{align}

(\romannumeral2) if $P_0=P_1=0$, for all $t,\mu>0$.
\begin{align}\label{eq_bdd_boundedness_of_b_0}
\|\Op{b_0^\iota(t,\mu)}\|_{\mathcal{L}(L^2)} \leqslant C; 
\end{align}

(\romannumeral3) if $P_0,P_1>0$ and $p_1<0<p_0$, for all $\mu>0$, $t>1$,
\begin{align}\label{eq_bdd_mod_boundedness_of_b_0}
\left\|\Op{b_0^\iota(t,\mu)(1-\chi_l)\left( \frac{\xi}{t^{\frac{1}{2}-\epsilon_0}} \right) \chi_h\left( \frac{\xi}{t^{\frac{1}{2}+\epsilon_1}} \right)}\right\|_{\mathcal{L}(L^2)} \leqslant C. 
\end{align}
\end{proposition}

Before giving the proof, we indicate below the consequence of this proposition, which implies Proposition \ref{prop_boundedness_of_truncated_operator}, \ref{prop_boundedness_of_truncated_operator_with_truncated_frequency}, and will be used in the end of this section to conclude Theorem \ref{thm_boundedness_of_simple_truncated_operator}.

\begin{corollary}\label{cor_boundedness_of_truncated_operator_with_mu}
Let $\chi$, $\chi_l$, $\chi_h$ be defined as before.

(\romannumeral1) If $P$ verifies \eqref{hyp_fractional-type_symbol} with $p_0,p_1\neq 0$ and $P_0=P_1=0$, the operator
$$ \Op{ \chi\left(\frac{x+tP'(\xi)}{|t|\mu}\right) \mathbbm{1}_{|x|>|tP'(\xi)|} } $$
is bounded on $L^2$ uniformly in $t\neq 0$ and $\mu>0$.

(\romannumeral2) If $P$ verifies \eqref{hyp_fractional-type_symbol} with $p_0,p_1\neq 0$ and $\epsilon_0, \epsilon_1$ satisfy conditions \eqref{eq_def_of_epsilons}, the operator
$$ \Op{ \chi\left(\frac{x+tP'(\xi)}{|t|\mu}\right) \mathbbm{1}_{|x|>|tP'(\xi)|} (1-\chi_l)\left( \frac{\xi}{t^{\frac{1}{2}-\epsilon_0}} \right) \chi_h\left( \frac{\xi}{t^{\frac{1}{2}+\epsilon_1}} \right) } $$
is bounded on $L^2$ uniformly in $t>1$ and $\mu>0$.
\end{corollary}

\subsection{Study of symbol \texorpdfstring{$b_1$}{b1} and \texorpdfstring{$\tilde{b}$}{tildeb}}

In this part, we shall prove \eqref{eq_bdd_boundedness_of_b_1} and \eqref{eq_bdd_boundedness_of_tilde{b}}. One observes that both $b_1^\iota$ and $\tilde{b}^\iota$ are supported for $X\sim\Xi$, which allows us to reduce our problem via dyadic decomposition.

\begin{proof}[Proof of \eqref{eq_bdd_boundedness_of_tilde{b}}]
We start with a homogeneous dyadic decomposition
$$ 1 = \sum_{k\in\mathbb{Z}} \varphi\left(\frac{\eta}{2^k}\right), $$
where $\varphi\in C_c^\infty(\mathbb{R}^d)$ is radial and supported away from zero. In this way, we may decompose $\tilde{b}^\sharp$ as $\sum_{k\geqslant 0} \tilde{b}_k$, and $\tilde{b}^\flat$ as $\sum_{k< 0} \tilde{b}_k$ where
\begin{equation}\label{eq_bdd_definition_of_tilde_b_k}
\tilde{b}_k(t,\mu,x,\xi) = \chi\left(\frac{(P_j+\nu_j X)\frac{x}{|x|}+(P_j+\nu_j\Xi)\frac{\xi}{|\xi|}}{\mu}\right)\Psi\left(\frac{X}{\Xi}\right) \psi\left( 2^{-kp_j}X \right) \varphi\left(\frac{\xi}{2^k\sqrt{t}}\right), 
\end{equation}
where $\psi\in C_c^\infty(\mathbb{R}^d)$ is also radial and supported away from zero. The extra factor $\psi$ comes from the truncation $X \sim \Xi \sim (t^{-\frac{1}{2}}|\xi|)^{p_j} \sim 2^{kp_j}$. This factors implies that the $\tilde{b}_k$'s are almost orthogonal so that it suffices to prove the uniform (in $k,t,\mu$) boundedness of $\Op{\tilde{b}_k}$. 

Remark that, due to the compact support of $\chi$ and the fact that $0<c<\frac{X}{\Xi}<1-c$ for some small $c>0$, we have
\begin{equation}\label{eq_bdd_support_of_tilde_b_k}
\begin{aligned}
2^{kp_j} \sim \Xi \lesssim |X-\Xi| &= \left| |(P_j+\nu_j X)\frac{x}{|x|}| - |(P_j+\nu_j\Xi)\frac{\xi}{|\xi|}| \right| \\
&\leqslant \left| (P_j+\nu_j X)\frac{x}{|x|} - (P_j+\nu_j\Xi)\frac{\xi}{|\xi|} \right| \lesssim \mu.
\end{aligned}
\end{equation}

When $t 2^{k(p_j+1)} \geqslant 1$, we shall apply Calderon-Vaillancourt Theorem (see Lemma \ref{lem_tech_C-V}). For each derivative in $x$, if it acts on $\chi$, one gains $t^{-\frac{1}{2}}\mu^{-1} \lesssim t^{-\frac{1}{2}}2^{-kp_j}$ by \eqref{eq_bdd_support_of_tilde_b_k}. If $\partial_x$ acts on $\Psi$ or $\psi$, one obtains factors of size $t^{-\frac{1}{2}}2^{-kp_j}$. As for the derivatives in $\xi$, similarly, it leads to factors of size $P''(\xi/\sqrt{t})t^{-\frac{1}{2}}\mu^{-1}$, $\Xi^{-1}P''(\xi/\sqrt{t})t^{-\frac{1}{2}}$ or $t^{-\frac{1}{2}}2^{-k}$, which are all controlled by $t^{-\frac{1}{2}}2^{-k}$, as $|\xi| \sim 2^k\sqrt{t}$ and $2^{kp_j}\lesssim \mu$. Since
$$ t^{-\frac{1}{2}}2^{-kp_j} \times t^{-\frac{1}{2}}2^{-k} = t^{-1} 2^{-k(p_j+1)} \leqslant 1, $$
we may conclude by a change of scaling (Lemma \ref{lem_tech_change_of_scaling}).

When $t 2^{k(p+1)} \leqslant 1$, we shall use Lemma \ref{lem_tech_boundedness_of_symbol_polar}. We first check the assumption \eqref{eq_tech_boundedness_of_symbol_polar_assumption_angular} of this lemma with $\mu_k\in ]0,+\infty[$ defined below
\begin{equation*}
\mu_k = \left\{ \begin{array}{cl}
\mu 2^{-kp_j}, &\ \ \text{if } P_j+\nu_j X\sim 2^{kp_j}, \\
\mu, &\ \ \text{if } P_j+\nu_j X\sim 1.
\end{array}\right.
\end{equation*}
Note that we have either $P_j+\nu_j X \sim P_j+\nu_j \Xi \sim 1$ or $P_j+\nu_j X \sim P_j+\nu_j \Xi \sim 2^{kp_j}$. In fact, by definition, $P_j+\nu_j X$ and $P_j+\nu_j \Xi$ are both strictly positive. Thus, it is sufficient to consider $|k| \gg 1$. When $P_j$ is nonzero and $kp_j<0$, we have $X\sim \Xi \sim 2^{kp_j} \ll 1$ and then $P_j+\nu_j X \sim P_j+\nu_j\Xi \sim 1$. While $P_j$ is nonzero and $kp_j>0$, we have similarly $X\sim \Xi \sim 2^{kp_j} \gg 1$ and $P_j+\nu_j X \sim P_j+\nu_j\Xi \sim 2^{kp_j}$. Otherwise, $P_j$ equals to zero, which implies trivially $P_j+\nu_j X = X\sim 2^{kp_j}$ and $P_j+\nu_j\Xi = \Xi \sim 2^{kp_j}$. Due to observation $X\sim\Xi\sim 2^{kp_j}$, it is easy to obtain that, for all $\alpha,\beta\in\mathbb{N}^{d-1}$ and $N\in\mathbb{N}$,
$$ |\partial_{\omega}^\alpha \partial_{\theta}^\beta \tilde{b}_k(t,\mu,r\omega,\rho\theta)| \leqslant C_{\alpha,\beta,N} g_k(r,\rho) \mu_k^{-|\alpha|-|\beta|} \langle\frac{d(\omega,-\theta)}{\mu_k}\rangle^{-N}, $$
where $x=r\omega, \xi=\rho\theta$ are polar coordinates and
$$ g_k(r,\rho)  = \mathbbm{1}_{\rho\sim 2^k\sqrt{t}} \mathbbm{1}_{r\sim t^{\frac{1}{2}}2^{kp_j}}. $$
The operator of kernel $g_k$ is controlled by
$$ \|g_k\|_{L^2(drd\rho)} \lesssim (\sqrt{t}2^k\times \sqrt{t}2^{kp_j})^{\frac{1}{2}} = (t2^{k(p_j+1)})^\frac{1}{2} \leqslant 1, $$
which is no more than the assumption \eqref{eq_tech_boundedness_of_symbol_polar_assumption_kernel} of Lemma \ref{lem_tech_boundedness_of_symbol_polar}. As a result, we may conclude \eqref{eq_bdd_boundedness_of_tilde{b}} by \eqref{eq_tech_boundedness_of_symbol_polar_conclusion}.
\end{proof}

The idea of proof of \eqref{eq_bdd_boundedness_of_b_1} is similar. The only difficulty is that $b_1$ has a singularity near $X=\Xi$. We may treat the part away from $X=\Xi$ as above and study the area near $X=\Xi$ by convexity (or concavity) of $P$.

\begin{proof}[Proof of \eqref{eq_bdd_boundedness_of_b_1}]
As before, we begin with the homogeneous dyadic decomposition in $\frac{\xi}{\sqrt{t}}$, namely $b_1^\sharp=\sum_{k\geqslant0} b_{1,k}$ and $b_1^\flat=\sum_{k<0} b_{1,k}$, with
$$b_{1,k} = \chi\left(\frac{(P_j+\nu_j X)\frac{x}{|x|}+(P_j+\nu_j\Xi)\frac{\xi}{|\xi|}}{\mu}\right) \mathbbm{1}_{\frac{X}{\Xi}<1}\chi_1\left(1-\frac{X}{\Xi}\right) \psi\left( 2^{-kp_j}X \right) \varphi\left(\frac{\xi}{2^k\sqrt{t}}\right).$$
It suffices to prove that $\Op{b_{1,k}}$ is bounded on $L^2$, uniformly in $k,t,\mu$. In comparison with $\tilde{b}_k$ defined by \eqref{eq_bdd_definition_of_tilde_b_k}, the main difficulty is that the non-smooth term cannot be deleted. In the case $t 2^{k(p_j+1)} \leqslant 1$, we may repeat exactly the same argument as the study of $\tilde{b}_k$ since this argument does not require any regularity in $|x|,|\xi|$.

When $t 2^{k(p_j+1)} > 1$, we will separate the singularity near $\frac{X}{\Xi}=1$. Consider the following decomposition :
\begin{equation*}
\begin{aligned}
b_{1,k} &= b_{1,k}' + b_{1,k}'',\\
b_{1,k}' &= b_{1,k} \tilde{\chi}_1(\sqrt{t}2^{k\frac{1-p_j}{2}}(\Xi-X)),
\end{aligned}
\end{equation*}
where $\tilde{\chi}_1 \in C_c^\infty(\mathbb{R}^d)$ is radial and equal to $1$ near zero.

The proof of boundedness of $b_{1,k}'$ is similar to that of the case $t 2^{k(p_j+1)} \leqslant 1$. By setting $\mu_k\in ]0,+\infty[$ as before, namely
\begin{equation*}
\mu_k = \left\{ \begin{array}{cl}
\mu 2^{-kp_j}, &\ \ \text{if } P_j+\nu_j X\sim 2^{kp_j}, \\
\mu, &\ \ \text{if } P_j+\nu_j X\sim 1,
\end{array}\right.
\end{equation*}
we may have, for all $\alpha,\beta\in\mathbb{N}^{d-1}$ and $N\in\mathbb{N}$,
$$ |\partial_{\omega}^\alpha \partial_{\theta}^\beta b_{1,k}'(t,\mu,r\omega,\rho\theta)| \leqslant C_{\alpha,\beta,N} h_k(r,\rho) \mu_k^{-|\alpha|-|\beta|} \langle\frac{d(\omega,-\theta)}{\mu_k}\rangle^{-N}, $$
where
\begin{equation}\label{eq_bdd_def_of_h_k}
h_k(r,\rho)  = \sum_{n\sim \sqrt{t}2^{k\frac{p_j+1}{2}}} \mathbbm{1}_{I_n}(r) \mathbbm{1}_{J_n}(\rho),
\end{equation}
with
\begin{equation*}
\begin{aligned}
J_n :=& [ 2^{k\frac{1-p_j}{2}}n, 2^{k\frac{1-p_j}{2}}(n+1) ], \\
I_n :=&
\left\{\begin{aligned}
&\left[ \sqrt{t}P'\left( t^{-\frac{1}{2}}2^{k\frac{1-p_j}{2}}n \right) - c2^{k\frac{p_j-1}{2}}, \sqrt{t}P'\left( t^{-\frac{1}{2}}2^{k\frac{1-p_j}{2}}(n+1) \right) + c2^{k\frac{p_j-1}{2}} \right], \\
&\hspace{22em}\text{if }P''>0,\\
&\left[ \sqrt{t}P'\left( t^{-\frac{1}{2}}2^{k\frac{1-p_j}{2}}(n+1) \right) - c2^{k\frac{p_j-1}{2}}, \sqrt{t}P'\left( t^{-\frac{1}{2}}2^{k\frac{1-p_j}{2}}n \right) + c2^{k\frac{p_j-1}{2}} \right], \\
&\hspace{22em}\text{if }P''<0.
\end{aligned}\right.
\end{aligned}
\end{equation*}
Note that by writing in polar system $r=|x|$ $\rho=|\xi|$, $b'_{1,k}$ is supported for $\rho\sim 2^k\sqrt{t}$ and $\sqrt{t}2^{k\frac{1-p_j}{2}}|\Xi-X| \ll 1$ due to cut-off $\tilde{\chi}_1$. We first make a decomposition in $\rho$, namely 
$$ \rho \in [C^{-1}2^k\sqrt{t}, C2^k\sqrt{t}] \subset \underset{n\sim \sqrt{t}2^{k\frac{p_j+1}{2}}}{\cup} J_n, $$ 
and then the support $\sqrt{t}2^{k\frac{1-p_j}{2}}|\Xi-X| \ll 1$ ensures that $r$ lies in $I_n$ defined above, once $\rho$ belongs to $J_n$. This gives the control $h_k$ defined in \eqref{eq_bdd_def_of_h_k}.

In order to apply Lemma \ref{lem_tech_boundedness_of_symbol_polar}, it suffices to check that the operator with kernel $h_k$ is uniformly bounded on $L^2(\mathbb{R}_+)$, which can be reduced to $|I_n||J_n| \lesssim 1$ and that $\{I_n\}$ forms a uniformly finite cover. The first assertion is obvious since $n\sim \sqrt{t}2^{k\frac{p_j+1}{2}}$ implies that
\begin{equation*}
\begin{aligned}
|J_n||I_n| &\lesssim 2^{k\frac{1-p_j}{2}} \left( \left| \sqrt{t}P'\left( t^{-\frac{1}{2}}2^{k\frac{1-p_j}{2}}(n+1) \right) - \sqrt{t}P'\left( t^{-\frac{1}{2}}2^{k\frac{1-p_j}{2}}n \right) \right| + 2c2^{k\frac{p_j-1}{2}} \right)  \\
&\lesssim 2^{k\frac{1-p_j}{2}} \left(\sqrt{t} 2^{k(p_j-1)} t^{-\frac{1}{2}}2^{k\frac{1-p_j}{2}} + 2c2^{k\frac{p_j-1}{2}} \right) \lesssim 1.
\end{aligned}
\end{equation*}
As for the second one, we observe that $I_n\cap I_{n+l} \neq \varnothing$ if and only if
$$ \left| \sqrt{t}P'\left( t^{-\frac{1}{2}}2^{k\frac{1-p_j}{2}}n \right) - \sqrt{t}P'\left( t^{-\frac{1}{2}}2^{k\frac{1-p_j}{2}}(n+l) \right) \right| \leqslant 2c2^{k\frac{p_j-1}{2}}. $$
Without loss of generality, we may assume $l\geqslant 0$. Actually, the left hand side has the following equivalence:
\begin{align*}
&\left| \sqrt{t}P'\left( t^{-\frac{1}{2}}2^{k\frac{1-p_j}{2}}n \right) - \sqrt{t}P'\left( t^{-\frac{1}{2}}2^{k\frac{1-p_j}{2}}(n+l) \right) \right| \\
=& \left| 2^{k\frac{1-p}{2}}l P''\left( t^{-\frac{1}{2}}2^{k\frac{1-p_j}{2}}(n+sl) \right) \right|,\ \ \text{for some }s\in[0,1], \\
\sim& 2^{k\frac{1-p_j}{2}}l \times \left( t^{-\frac{1}{2}}2^{k\frac{1-p_j}{2}}(n+sl) \right)^{p_j-1},\ \ \text{since }|P''(\rho)| \sim \rho^{p_j-1} \\
\sim& 2^{k\frac{1-p_j}{2}}l \times 2^{k(p_j-1)} = 2^{k\frac{p_j-1}{2}}l.
\end{align*}
To prove the last equivalence, we may use the fact that $n,n+l \sim \sqrt{t}2^{k\frac{p_j+1}{2}}$, which implies that
$$ c\sqrt{t}2^{k\frac{p_j+1}{2}} \leqslant n \leqslant n+sl \leqslant n+l \leqslant C \sqrt{t}2^{k\frac{p_j+1}{2}}. $$
In conclusion, we have that $I_n\cap I_{n+l} \neq \varnothing$ holds for finitely many $l$. As a result,
$$ \|h_k\|_{\mathcal{L}(L^2)} \lesssim \sup_{n\sim \sqrt{t} 2^{k\frac{p_j+1}{2}}} \|\mathbbm{1}_{J_n} \otimes \mathbbm{1}_{I_n}\|_{\mathcal{L}(L^2)} \leqslant \sup_{n\sim \sqrt{t} 2^{k\frac{p_j+1}{2}}} \|\mathbbm{1}_{J_n}\mathbbm{1}_{I_n}\|_{L^2(\mathbb{R}^2_+)} \lesssim 1. $$

It remains to study the smooth symbol $b_{1,k}''$, which reads
\begin{equation*}
\begin{aligned}
b_{1,k}'' = &\chi\left(\frac{(P_j+\nu_j X)\frac{x}{|x|}+(P_j+\nu_j\Xi)\frac{\xi}{|\xi|}}{\mu}\right)  \psi\left( 2^{-kp}X \right) \varphi\left(\frac{\xi}{2^k\sqrt{t}}\right) \\
&\hspace{2em}\times \mathbbm{1}_{\frac{X}{\Xi}<1} \chi_1\left(1-\frac{X}{\Xi}\right)(1-\tilde{\chi}_1)(\sqrt{t}2^{k\frac{1-p_j}{2}}(\Xi-X)).
\end{aligned}
\end{equation*}
Remark that this symbol is smooth, since the singularity $\frac{X}{\Xi}=1$ is removed by $(1-\tilde{\chi}_1)$ factor. Under the condition $t 2^{k(p_j+1)} > 1$, it verifies the condition of Calderon-Vaillancourt Theorem (see Lemma \ref{lem_tech_C-V}). In fact, each derivative in $x$ leads to a factor of size $t^{-\frac{1}{2}}\mu^{-1}$ (from $\chi$), $t^{-\frac{1}{2}}2^{-kp_j}$ (from $\psi$ and $\chi_1$), or $2^{k\frac{1-p_j}{2}}$ (from $(1-\tilde{\chi}_1)$). The condition $t 2^{k(p_j+1)} > 1$ implies that $t^{-\frac{1}{2}}2^{-kp_j} \leqslant 2^{k\frac{1-p_j}{2}}$, while the compact support of $\chi$ and support of $(1-\tilde{\chi}_1)$ ensures that 
$$ t^{-\frac{1}{2}}2^{k\frac{p_j-1}{2}} \lesssim |X-\Xi| \lesssim \mu, $$
i.e. $t^{-\frac{1}{2}}\mu^{-1} \lesssim 2^{\frac{1-p_j}{2}k}$. The same argument for $\partial_\xi$ gives that each derivative in $\xi$ leads to a factor of size $2^{-k\frac{1-p_j}{2}}$. The desired result thus follows from a change of scaling (Lemma \ref{lem_tech_change_of_scaling}).
\end{proof}

\subsection{Study of symbol \texorpdfstring{$b_0$}{b0} with \texorpdfstring{$P_0=P_1=0$}{P0,P1=0}}

In the case $P_0=P_1=0$, due to the lack of almost orthogonality as $\tilde{b}_k$'s and $b_{1,k}$'s, the remaining symbol $b_0^\iota$ will be treated via Cotlar-Stein Lemma (Lemma \ref{lem_tech_C-S}). As before, we start with homogeneous dyadic decomposition in $\xi$, namely $b_0^\sharp = \mathbbm{1}_{X>0} \sum_{k\geqslant 0} c_k$ and $b_0^\flat = \mathbbm{1}_{X>0} \sum_{k< 0} c_k$, with
\begin{equation}\label{eq_bdd_definition_of_c_k}
c_k = \chi\left(\frac{(P_j+\nu_j X)\frac{x}{|x|}+(P_j+\nu_j \Xi)\frac{\xi}{|\xi|}}{\mu}\right) \chi_0\left(\frac{X}{\Xi}\right) \varphi\left(\frac{\xi}{2^k\sqrt{t}}\right).
\end{equation}
It suffices to prove the (uniform in $t$) boundedness of $\sum_{k\in\mathbb{Z}} c_k$ as the multiplication with $\mathbbm{1}_{X>0}$ is trivially bounded on $L^2$. 

We first check that $\Op{c_k}$'s are bounded uniformly in $k,t,\mu$. More precisely, all the $c_k$'s satisfy the following estimate:
\begin{lemma}\label{lem_bdd_bdd_c_k}
There exists $C>0$ independent of $k,t$, such that for all $t>0$ and $k\in\mathbb{Z}$,
\begin{equation}\label{eq_bdd_bdd_c_k}
\|\Op{c_k}\|_{\mathcal{L}(L^2)} \leqslant C \min\left(\max\left(1,(t2^{k(p_j+1)})^{-N_d}\right), (t2^{k(p_j+1)})^{\frac{d}{2}} \right) \leqslant C,
\end{equation}
where $N_d \in\mathbb{N}$ depends only on dimension $d$.
\end{lemma}
\begin{proof}
We observe that $c_k$ is supported for $X \ll \Xi \sim 2^{kp_j}$, which implies from one hand, as in \eqref{eq_bdd_support_of_tilde_b_k},
$$ 2^{kp_j} \sim |X-\Xi| \leqslant \left|(P_j+\nu_j X)\frac{x}{|x|}+(P_j+\nu_j \Xi)\frac{\xi}{|\xi|}\right| \lesssim \mu, $$
and, from another hand,
$$ |x| \ll \sqrt{t}2^{kp_j},\ |\xi|\sim \sqrt{t}2^{k}. $$
As a consequence of the second result, $\|\Op{c_k}\|_{\mathcal{L}(L^2)}$ can be trivially bounded by
$$ \|\Op{c_k}\|_{\mathcal{L}(L^2)} \lesssim \|c_k\|_{L^2(\mathbb{R}^{2d})} \lesssim (t2^{k(p_j+1)})^{\frac{d}{2}}. $$
It remains to check that
$$ \|\Op{c_k}\|_{\mathcal{L}(L^2)} \leqslant C \max\left(1,(t2^{k(p_j+1)})^{-N_d}\right), $$
which can be proved via Calderon-Vaillancourt Theorem (see Lemma \ref{lem_tech_C-V}). In fact, from each derivative in $x$, we may obtain extra factors of size $t^{-\frac{1}{2}}\mu^{-1}$ (action on $\chi$) or $t^{-\frac{1}{2}}2^{-kp_j}$ (action on $\chi_0$). As we have seen that $t^{-\frac{1}{2}}\mu^{-1} \lesssim t^{-\frac{1}{2}}2^{-kp_j}$, each derivative in $x$ leads to a factor of size $(t^{\frac{1}{2}}2^{kp_j})^{-1}$. Similarly, the action of $\partial_\xi$ on $\chi,\chi_0,\varphi$ gives factors of size $t^{-\frac{1}{2}}\mu^{-1}2^{k(p_j-1)}$, $t^{-\frac{1}{2}}2^{-k}$, and $t^{-\frac{1}{2}}2^{-k}$, respectively. We may also check that $ t^{-\frac{1}{2}}\mu^{-1} 2^{k(p_j-1)} \lesssim 2^{-k}t^{-\frac{1}{2}}$. To sum up, $c_k$ is smooth and satisfies 
\begin{equation}\label{eq_bdd_c_k_C-V}
\left| \partial_x^\alpha \partial_\xi^\beta c_k(x,\xi) \right| \leqslant C_{\alpha,\beta} \left(\frac{1}{\sqrt{t}2^{kp_j}}\right)^{|\alpha|} \left(\frac{1}{\sqrt{t}2^k}\right)^{|\beta|},\ \ \forall \alpha,\beta\in\mathbb{N}^d. 
\end{equation}
By Lemma \ref{lem_tech_change_of_scaling}, it is equivalent to consider rescaled symbol
$$ \tilde{c}_k(x,\xi) = c_k \left( 2^{k\frac{p_j-1}{2}} x, 2^{-k\frac{p_j-1}{2}} \xi \right), $$
which, as a result of \eqref{eq_bdd_c_k_C-V}, satisfies for all $\gamma\in\mathbb{N}^{2d}$,
$$ \|\partial_{x,\xi}^\gamma \tilde{c}_k\|_{L^\infty(\mathbb{R}^{2d})} \lesssim (t2^{k(p_j+1)})^{-|\gamma|}. $$
By applying Calderon-Vaillancourt Theorem (Lemma \ref{lem_tech_C-V}) to $\tilde{c}_k$, we have, due to estimate \eqref{eq_tech_C-V}, that
$$ \|\Op{c_k}\|_{\mathcal{L}(L^2)} = \|\Op{\tilde{c}_k}\|_{\mathcal{L}(L^2)} \leqslant C \max\left(1,(t2^{k(p_j+1)})^{-N_d}\right). $$
\end{proof}

In order to conclude \eqref{eq_bdd_boundedness_of_b_0} by Cotlar-Stein Lemma (Lemma \ref{lem_tech_C-S}), it is sufficient to check condition \eqref{eq_tech_C-S_def_of_A} and \eqref{eq_tech_C-S_def_of_B}, namely
\begin{lemma}
There exists $t,\mu$-independent constants $C$, such that for all $t,\mu>0$,
\begin{align}
&\sup_{k\in \mathbb{Z_+}} \sum_{l\in \mathbb{Z_+}} \|\Op{c_k}\Op{c_l}^*\|_{\mathcal{L}(L^2)}^\frac{1}{2} \leqslant C,\ \ \sup_{k\in \mathbb{Z_-}} \sum_{l\in \mathbb{Z_-}} \|\Op{c_k}\Op{c_l}^*\|_{\mathcal{L}(L^2)}^\frac{1}{2} \leqslant C, \label{eq_bdd_Cotlar-Stein_part1} \\
&\sup_{k\in \mathbb{Z_+}} \sum_{l\in \mathbb{Z_+}} \|\Op{c_k}^*\Op{c_l}\|_{\mathcal{L}(L^2)}^\frac{1}{2} \leqslant C,\ \ \sup_{k\in \mathbb{Z_-}} \sum_{l\in \mathbb{Z_-}} \|\Op{c_k}^*\Op{c_l}\|_{\mathcal{L}(L^2)}^\frac{1}{2} \leqslant C, \label{eq_bdd_Cotlar-Stein_part2}
\end{align}
where $\mathbb{Z}_- = \mathbb{Z} \cap ]-\infty,0[$ corresponds to low frequency part and $\mathbb{Z}_+ = \mathbb{Z} \cap [0,+\infty[$ corresponds to high frequency part.
\end{lemma}

\begin{proof}[Proof of \eqref{eq_bdd_Cotlar-Stein_part1}]
By symbolic calculus, $\Op{c_k}\Op{c_l}^*$ is an operator of symbol
$$ c_k \sharp c_l^*(x,\xi) = \frac{1}{(2\pi)^d} \int e^{-iy\eta} c_k(x,\xi+\eta) \overline{c_l(x+y,\xi+\eta)} d\eta. $$
By definition \eqref{eq_bdd_definition_of_c_k}, $c_l(x,\xi)$ is supported for $|\xi| \sim 2^l\sqrt{t}$. Thus, $c_k \sharp c_l^*$ is nonzero only if $|l-k|<N_0$ for some large $N_0\in\mathbb{N}^*$. As a consequence, \eqref{eq_bdd_Cotlar-Stein_part1} can be reduced to the uniform boundedness of $\Op{c_l}$, which has already been proved in Lemma \ref{lem_bdd_bdd_c_k}.
\end{proof}

\begin{proof}[Proof of \eqref{eq_bdd_Cotlar-Stein_part2}]
We apply again the symbolic calculus to obtain the following expression of symbol of $\Op{c_k}^*\Op{c_l}$,
\begin{equation*}
c_k^* \sharp c_l (x,\xi) = \frac{1}{(2\pi)^d} \int e^{i(x-y)\eta} \overline{c_k(y,\xi+\eta)} c_l(y,\xi) d\eta dy.
\end{equation*}

Recall that $c_l$ is supported for $|x|\ll \sqrt{t}2^{lp_j}$ and $|\xi|\sim t^{\frac{1}{2}}2^l$ with the estimate \eqref{eq_bdd_c_k_C-V}. We shall check that for all $l,k\in\mathbb{N}$, 
\begin{equation}\label{eq_bdd_Cotlar-Stein_main_estimate}
\|\Op{c_k}^*\Op{c_l}\|_{\mathcal{L}(L^2)} \lesssim 2^{-\frac{d}{2}|k-l|},
\end{equation}
which is enough to conclude \eqref{eq_bdd_Cotlar-Stein_part2}. Due to \eqref{eq_bdd_bdd_c_k}, we may ignore the case $|k-l|\leqslant N_0$, for some fixed large $N_0\in\mathbb{N}$. Remark that it is possible to prove \eqref{eq_bdd_Cotlar-Stein_main_estimate} only for $l\geqslant k$, since for terms with $l<k$, we have
\begin{align*}
\|\Op{c_k}^*\Op{c_l}\|_{\mathcal{L}(L^2)} =& \|\left(\Op{c_k}^*\Op{c_l}\right)^*\|_{\mathcal{L}(L^2)}\\
=& \|\Op{c_l}^*\Op{c_k}\|_{\mathcal{L}(L^2)} \lesssim 2^{-\frac{d}{2}|l-k|}.
\end{align*}

One observes that the bound of operator with symbol $c_k^* \sharp c_l$ can be controlled by
\begin{equation*}
\begin{aligned}
\| \Op{c_k^* \sharp c_l} \|_{\mathcal{L}(L^2)} \lesssim \| c_k^* \sharp c_l \|_{L^2(dxd\xi)} \lesssim \| \int e^{-iy\eta} \overline{c_k(y,\xi+\eta)} c_l(y,\xi) dy \|_{L^2(d\eta d\xi)}.
\end{aligned}
\end{equation*}
The integrand of the last integral is supported for
\begin{equation}\label{eq_bdd_Cotlar-Stein_support_of_composed_sym}
|\xi+\eta| \sim 2^k\sqrt{t}, |\xi| \sim 2^{l}\sqrt{t},\ \text{and } |y| \lesssim \min(2^{kp_j}\sqrt{t}, 2^{lp_j}\sqrt{t}).
\end{equation}
Moreover, we may apply integration by parts in $y$ to obtain some extra bounds in the estimate. To be precise, for all $N_1\in\mathbb{N}$
\begin{align*}
& \int e^{-iy\eta} \overline{c_k(y,\xi+\eta)} c_l(y,\xi) dy \\
=& \int \left(\frac{-\Delta_y}{|\eta|^2}\right)^{N_1} e^{-iy\eta} \overline{c_k(y,\xi+\eta)} c_l(y,\xi) dy \\
=& \int e^{-iy\eta} (-\Delta_y)^{N_1} \left(\overline{c_k(y,\xi+\eta)} c_l(y,\xi)\right) |\eta|^{-2N_1} dy \\
=& \sum_{|\alpha|+|\beta|=2N_1} C_{\alpha,\beta} \int e^{-iy\eta}  \overline{\partial_y^{\alpha}c_k(y,\xi+\eta)} \partial_y^{\beta}c_l(y,\xi) |\eta|^{-2N_1} dy.
\end{align*}
Since we have reduced our problem to the case $l \geqslant k + N_0$, the integral above is supported for $|\eta| \sim 2^{l}\sqrt{t}$. Together with \eqref{eq_bdd_c_k_C-V} and \eqref{eq_bdd_Cotlar-Stein_support_of_composed_sym}, we have
\begin{align*}
& \left| \int e^{-iy\eta} \overline{c_k(y,\xi+\eta)} c_l(y,\xi) dy \right| \\
\leqslant& \sum_{|\alpha|+|\beta|=2N_1} C_{\alpha,\beta} \int  \left|\partial_y^{\alpha}c_k(y,\xi+\eta)\right| \left|\partial_y^{\beta}c_l(y,\xi)\right| |\eta|^{-2N_1} dy \\
\lesssim& \sum_{|\alpha|+|\beta|=2N_1} \mathbbm{1}_{|\xi+\eta|\sim 2^k\sqrt{t}}\mathbbm{1}_{|\xi|\sim2^l\sqrt{t}} \left(2^{kp_j}\sqrt{t}\right)^{-|\alpha|} \left(2^{lp_j}\sqrt{t}\right)^{-|\beta|} (2^{l}\sqrt{t})^{-2N_1} \int \mathbbm{1}_{|y| \lesssim \min(2^{kp_j}\sqrt{t}, 2^{lp_j}\sqrt{t})} dy  \\
\lesssim& \mathbbm{1}_{|\xi+\eta|\sim 2^k\sqrt{t}}\mathbbm{1}_{|\xi|\sim2^l\sqrt{t}} (2^{l}\sqrt{t})^{-2N_1}\min\left(2^{kp_j}\sqrt{t},2^{lp_j}\sqrt{t}\right)^{d-2N_1}.
\end{align*}
The estimate above holds for all $N_1\in\mathbb{N}$, thus for all $N_1\in[0,\infty[$. In particular, we choose $N_1=\frac{d}{2}$, which gives that
\begin{align*}
\| \Op{c_k^* \sharp c_l} \|_{\mathcal{L}(L^2)} \lesssim& \| \int e^{-iy\eta} \overline{c_k(y,\xi+\eta)} c_l(y,\xi) dy \|_{L^2(d\eta d\xi)} \\
\lesssim& (2^{l}\sqrt{t})^{-d} \| \mathbbm{1}_{|\xi+\eta|\sim 2^k\sqrt{t}}\mathbbm{1}_{|\xi|\sim2^l\sqrt{t}} \|_{L^2(d\eta d\xi)} \\
\lesssim& (2^{l}\sqrt{t})^{-d} \times \left( 2^k\sqrt{t} \times 2^l\sqrt{t} \right)^{\frac{d}{2}}= 2^{\frac{d}{2}(k-l)} = 2^{-\frac{d}{2}|k-l|}.
\end{align*}
As a conclusion, we have managed to prove that
$$ \sup_{k\in\mathbb{Z}_\pm} \sum_{l\in\mathbb{Z}_\pm} \|\Op{c_k}^*\Op{c_l}\|_{\mathcal{L}(L^2)}^\frac{1}{2} \lesssim \sup_{k\in\mathbb{Z}_\pm} \sum_{l\in\mathbb{Z}_\pm} 2^{-\frac{d}{4}|k-l|} <\infty, $$
which completes the proof.
\end{proof}

\subsection{Study of symbol \texorpdfstring{$b_0$}{b0} with \texorpdfstring{$P_0,P_1>0$}{P0,P1>0}}

Till now, we have finished the proof of Proposition \ref{prop_boundedness_of_truncated_operator}. To complete the proof of Proposition \ref{prop_boundedness_of_truncated_operator_with_truncated_frequency}, it remains to check \eqref{eq_bdd_mod_boundedness_of_b_0}. Remark that the argument above relies on the fact that $|X| \ll 2^{kp_j}$ implies $x$ is supported in a region of area $(\sqrt{t}2^{kp_j})^d$, which is not true in the case where $P_0,P_1$ are non zero. To overcome this problem we need the extra truncation in $\xi$. 

\begin{proof}[Proof of \eqref{eq_bdd_mod_boundedness_of_b_0}]
As above, we may ignore the non-smooth factor $\mathbbm{1}_{0<\frac{X}{\Xi}<1}$. It remains smooth symbols
\begin{align*}
\tilde{c}^\sharp(t,x,\xi) &= \chi\left(\frac{(P_1+\nu_1 X)\frac{x}{|x|}+(P_1\nu_1\Xi)\frac{\xi}{|\xi|}}{\mu}\right) \chi_0\left(\frac{X}{\Xi}\right) (1-\tilde{\chi})\left(\frac{\xi}{\sqrt{t}}\right) \chi_h\left( \frac{\xi}{t^{\frac{1}{2}+\epsilon_1}} \right), \\
\tilde{c}^\flat(t,x,\xi) &= \chi\left(\frac{(P_0+\nu_0 X)\frac{x}{|x|}+(P_0\nu_0\Xi)\frac{\xi}{|\xi|}}{\mu}\right) \chi_0\left(\frac{X}{\Xi}\right) \tilde{\chi}\left(\frac{\xi}{\sqrt{t}}\right) (1-\chi_l)\left( \frac{\xi}{t^{\frac{1}{2}-\epsilon_0}} \right).
\end{align*}
We shall first check that $\tilde{c}^\sharp$ belongs uniformly to the Hörmander class $S^0_{1,\kappa}$ for some $\kappa\in ]0,1[$, namely the collection of smooth symbols $c(x,\xi)$ such that for all $\alpha,\beta\in\mathbb{N}^d$,
$$ |\partial_x^\alpha \partial_\xi^\beta c(x,\xi) | \lesssim \langle\xi\rangle^{-|\beta|+\kappa|\alpha|}. $$
It is well-known that the operators of symbol in this class is bounded on $L^2$, a proof of which can be found in \cite{Hormander2007}. We begin with the observation that high frequency symbol $\tilde{c}^\sharp$ is supported for $t^\frac{1}{2} \lesssim |\xi| \lesssim t^{\frac{1}{2}+\epsilon_1}$. Before calculating the bounds of derivatives in $x$ and $\xi$, recall that our goal is to show \eqref{eq_bdd_mod_boundedness_of_b_0} under the condition $t>1$.

For each derivative in $\xi$, we obtain from $\chi$ a factor of size
$$ t^{-\frac{1}{2}}\mu^{-1} P''(|\xi|/\sqrt{t}) \sim \mu^{-1} \left(\frac{|\xi|}{\sqrt{t}}\right)^{p_1} |\xi|^{-1} \lesssim \langle\xi\rangle^{-1}. $$
The last inequality is due the support of $\tilde{c}^\sharp$. More precisely,
$$ \left(\frac{|\xi|}{\sqrt{t}}\right)^{p_1} \sim \Xi \lesssim |X-\Xi| \lesssim \left|(P_1+\nu_1 X)\frac{x}{|x|}+(P_1+\nu_1\Xi)\frac{\xi}{|\xi|}\right| \lesssim \mu. $$
From the factor $\chi_0$, one gains
$$ \frac{X}{\Xi} \frac{P''(\xi/\sqrt{t})}{\Xi\sqrt{t}} \sim \frac{X}{\Xi} \frac{1}{|\xi|} \lesssim \langle\xi\rangle^{-1}. $$
Trivially, we will also obtain $\langle\xi\rangle^{-1}$ from derivative on $(1-\tilde{\chi})$ and $\chi_h$. 

As for derivatives in $x$, if $\partial_x$ acts on $\chi$, one gains $t^{-\frac{1}{2}}\mu^{-1} \lesssim (\sqrt{t}\Xi)^{-1}$. When it acts on $\chi_0$, the resulting factor is of size
$$ \frac{1}{\sqrt{t}\Xi} \sim \frac{1}{t} \left(\frac{|\xi|}{\sqrt{t}}\right)^{-p_1-1}|\xi|. $$
When $0>p_1 \geqslant -1$, we have
$$ \frac{1}{t} \left(\frac{|\xi|}{\sqrt{t}}\right)^{-p_1-1}|\xi| = |\xi|^{-p_1} t^{\frac{p_1-1}{2}} \lesssim \langle\xi\rangle t^{-\frac{1}{2}} \lesssim \langle\xi\rangle^\kappa t^{(\frac{1}{2}+\epsilon_1)(1-\kappa)-\frac{1}{2}}. $$
Thus, $\tilde{c}^\sharp \in S^0_{1,\kappa}$ for any $\kappa\in ]0,1[$ such that
$$ (\frac{1}{2}+\epsilon_1)(1-\kappa)-\frac{1}{2} \leqslant 0, $$
which is possible by choosing $\kappa$ close to $1$. When $p_1<-1$, the estimate above becomes
$$ \frac{1}{t} \left(\frac{|\xi|}{\sqrt{t}}\right)^{-p_1-1}|\xi| = \frac{|\xi|^{1-\kappa}}{t} \left(\frac{|\xi|}{\sqrt{t}}\right)^{-p_1-1}|\xi|^{\kappa} \lesssim t^{-1+(\frac{1}{2}+\epsilon_1)(1-\kappa)-(p_1+1)\epsilon_1} \langle\xi\rangle^\kappa. $$
To conclude $\tilde{c}^\sharp \in S^0_{1,\kappa}$, it suffices to choose $\kappa\in]0,1[$ such that
$$ -1+(\frac{1}{2}+\epsilon_1)(1-\kappa)-(p_1+1)\epsilon_1 \leqslant 0,  $$
which is equivalent to 
$$ \epsilon_1 \leqslant \frac{1}{-(p_1+1)}\left[ 1-(\frac{1}{2}+\epsilon_1)(1-\kappa) \right]. $$
This can be realized by choosing $\kappa$ close to $1$, due to the definition \eqref{eq_def_of_epsilons} of $\epsilon_1$.

To prove the uniform boundedness of $\tilde{c}^\flat$, which is supported for $t^{\frac{1}{2}-\epsilon_0} \lesssim |\xi| \lesssim t^{\frac{1}{2}}$, we shall apply Calderon-Vaillancourt Theorem (Lemma \ref{lem_tech_C-V}). As above one may check easily that each $\partial_\xi$ gives 
$$ |\xi|^{-1} \lesssim t^{-\frac{1}{2}+\epsilon_0}; $$
while each $\partial_x$ gives
$$ \frac{1}{\sqrt{t}} \left(\frac{|\xi|}{\sqrt{t}}\right)^{-p_0} \lesssim t^{-\frac{1}{2}+p_0\epsilon_0}. $$
As a consequence, the desired result follows from Lemma \ref{lem_tech_change_of_scaling} and Calderon-Vaillancourt Theorem (Lemma \ref{lem_tech_C-V}) once we have
$$ t^{-\frac{1}{2}+\epsilon_0} \times t^{-\frac{1}{2}+p_0\epsilon_0} \lesssim 1,\ \ \forall t>1, $$
equivalently, $(1+p_0)\epsilon_0 \leqslant 1$, which is exactly the definition \eqref{eq_def_of_epsilons} of $\epsilon_0$.
\end{proof}

\subsection{Proof of Theorem \ref{thm_boundedness_of_simple_truncated_operator}}

In all the proof above, we regard $\mu=t^{\delta-\frac{1}{2}} \in ]0,+\infty[$ as a time-independent parameter. This allows us to take the limit $\mu \rightarrow +\infty$ with all the uniform estimates remaining true. Rigorously, due to Corollary \ref{cor_boundedness_of_truncated_operator_with_mu}, for all $f,g\in \mathcal{S}(\mathbb{R}^d)$ and $t\in\mathbb{R},\mu>0$,
\begin{equation}\label{eq_bdd_uniform_in_mu_bound}
\left| \left\langle f , \Op{\chi\left( \frac{x+tP'(\xi)}{|t|\mu} \right) \mathbbm{1}_{|x|>|tP'(\xi)|} \Omega(t,\xi)} g \right\rangle  \right| \leqslant C\|f\|_{L^2}\|g\|_{L^2},
\end{equation}
where $\Omega=1$ when $P_0=P_1=0$, and
$$ \Omega(t,\xi) = (1-\chi_l)\left( \frac{\xi}{|t|^{-\epsilon_0}} \right) \chi_h\left( \frac{\xi}{|t|^{\epsilon_1}} \right) $$
when $P_0,P_1>0$.

The left hand side of \eqref{eq_bdd_uniform_in_mu_bound} is equal to
$$ \left| \frac{1}{(2\pi)^d} \int f(x) e^{-ix\xi} \chi\left( \frac{x+tP'(\xi)}{|t|\mu} \right) \mathbbm{1}_{|x|>|tP'(\xi)|} \Omega(t,\xi) \overline{\hat{g}(\xi)} d\xi dx \right|, $$
which, when $\mu\rightarrow +\infty$, due to Dominated Convergence Theorem, tends to
$$ \left| \frac{1}{(2\pi)^d} \int f(x) e^{-ix\xi} \mathbbm{1}_{|x|>|tP'(\xi)|} \Omega(t,\xi) \overline{\hat{g}(\xi)} d\xi dx \right|, $$
where we take $\chi(0)=1$ without loss of generality. In conclusion, for all $f,g\in \mathcal{S}(\mathbb{R}^d)$,
$$ \left| \left\langle f , \Op{\mathbbm{1}_{|x|>|tP'(\xi)|} \Omega(t,\xi)} g \right\rangle  \right| \leqslant C\|f\|_{L^2}\|g\|_{L^2}. $$
Theorem \ref{thm_boundedness_of_simple_truncated_operator} follows from the density of $\mathcal{S}(\mathbb{R}^d)$ in $L^2$.

%% file: main/L2_bound_alter.tex
\section{\texorpdfstring{$L^{2}$}{L2}-boundedness of microlocal truncation operators - An alternative symbol}\label{sect_L^2_bound_alt}

In this section, we will treat those $P$ with nonzero $P_0,P_1$ in an alternative way. Instead of adding extra truncation in $\xi$, we shall add some extra factor in the main truncation $\chi$. To be precise,

\begin{proposition}\label{prop_alt_boundedness_of_truncation_operator}
Let $P$, $\Lambda$ satisfy conditions \eqref{hyp_fractional-type_symbol} and \eqref{hyp_fractional-type_symbol_extra_factor} of Section \ref{subsect_main_result} respectively, with $p_1<0<p_0$, $\sigma_0 \geqslant p_0$, and $\sigma_1 \leqslant p_1$. We further assume that $\delta+\frac{\sigma_j}{2} < \frac{1}{2}$, $j=0,1$. Then there exist time-independent constants $C>0$, $t_0 \gg 1$, such that , for all $|t|>t_0$
$$ \|\Op{a^\text{alt}_{\chi,\delta,\Lambda}(t)}\|_{\mathcal{L}(L^2)} \leqslant C, $$
where $a^\text{alt}_{\chi,\delta,\Lambda}$ is defined in \eqref{eq_alt_truncation_symbol}
\end{proposition}

Without loss of generality, we may assume $t>0$. Meanwhile, the change of scaling (Lemma \ref{lem_tech_change_of_scaling}) allows us to reduce to the symbol
\begin{equation}\label{eq_alt_truncated_symbol_mod}
b(t,x,\xi) = \chi\left( \frac{\frac{x}{\sqrt{t}}+P'\left(\frac{\xi}{\sqrt{t}}\right)}{t^{\delta-\frac{1}{2}} \Lambda(\xi)} \right) H\left( \frac{\frac{|x|}{\sqrt{t}} - P'\left(\frac{|\xi|}{\sqrt{t}}\right)}{t^{\delta-\frac{1}{2}}\Lambda(\xi)} \right),
\end{equation}
where $H \in C^\infty_b(\mathbb{R}\backslash\{0\}) \cap L^\infty(\mathbb{R})$. To recover the desired result in Proposition \ref{prop_alt_boundedness_of_truncation_operator}, it suffices to take $H = \mathbbm{1}_{]0,+\infty[}$.

As in previous section, we set

-when $P_j=0$,
$$ X:= \frac{|x|}{\sqrt{t}},\ \ \Xi:= P'\left(\frac{|\xi|}{\sqrt{t}}\right),\ \ \nu_j=+; $$

-when $P_j>0$, $P'>P_j$
$$ X:= \frac{|x|}{\sqrt{t}}-P_j,\ \ \Xi:= P'\left(\frac{|\xi|}{\sqrt{t}}\right)-P_j,\ \ \nu_j=+; $$

-when $P_j>0$, $P'<P_j$
$$ X:= P_j-\frac{|x|}{\sqrt{t}},\ \ \Xi:= P_j-P'\left(\frac{|\xi|}{\sqrt{t}}\right),\ \ \nu_j=-. $$

With these notations, via homogeneous dyadic decomposition, we may rewrite symbol $b$ as
\begin{equation*}
\begin{aligned}
b(t,x,\xi) &= \sum_{k\in\mathbb{Z}} b_k(t,x,\xi), \\
b_k(t,x,\xi) &= \chi\left( \frac{(P_1+\nu_1 X)\frac{x}{|x|}+(P_1+\nu_1 \Xi)\frac{\xi}{|\xi|}}{t^{\delta-\frac{1}{2}} \Lambda(\xi)} \right) H \left( \frac{X-\Xi}{t^{\delta-\frac{1}{2}} \Lambda(\xi)} \right)\varphi\left(\frac{\xi}{\sqrt{t}2^{k}}\right),\ \ \forall k\geqslant 0, \\
b_k(t,x,\xi) &= \chi\left( \frac{(P_0+\nu_0 X)\frac{x}{|x|}+(P_0+\nu_0 \Xi)\frac{\xi}{|\xi|}}{t^{\delta-\frac{1}{2}} \Lambda(\xi)} \right) H \left( \frac{X-\Xi}{t^{\delta-\frac{1}{2}} \Lambda(\xi)} \right)\varphi\left(\frac{\xi}{\sqrt{t}2^{k}}\right),\ \ \forall k<0,
\end{aligned}
\end{equation*}
where $\varphi\in C_c^\infty(\mathbb{R}^d)$ is radial and supported in an annulus centered at zero. 

One observes that, due to factors $\chi$ and $\varphi$, $b_k$ is supported for 
$$ |X-\Xi| \lesssim t^{\delta-\frac{1}{2}} \Lambda(\xi) \sim t^{\delta-\frac{1}{2}} (\sqrt{t}2^{k})^{\sigma_j},$$
where $j=0$ when $k<0$, $j=1$ when $k\geqslant 0$. By using the fact that $\Xi \sim \left|\frac{\xi}{\sqrt{t}}\right|^{p_j} \sim 2^{kp_j}$, we obtain
$$ \left| \frac{X}{\Xi} -1 \right| \lesssim t^{\delta+\frac{\sigma_j}{2}-\frac{1}{2}} 2^{k(\sigma_j-p_j)} \leqslant t^{\delta+\frac{\sigma_j}{2}-\frac{1}{2}}. $$
The last inequality is the consequence of our assumptions $\sigma_0 \geqslant p_0$ and $\sigma_1 \leqslant p_1$. Since $\delta+\frac{\sigma_j}{2}-\frac{1}{2}<0$, if we further assume that $t \geqslant t_0 \gg 1$, the inequality above implies that $X\sim\Xi$, which allows us to add a complementary factor $\psi(2^{-kp_j}X)$ to the definition of $b_k$, where $\psi\in C_c^\infty(\mathbb{R})$ is radial and supported in an annulus centered at zero. In this way, we may reduce Proposition \ref{prop_alt_boundedness_of_truncation_operator} to
\begin{proposition}\label{prop_alt_boundedness_of_b_k}
Under the same assumptions as in Proposition \ref{prop_alt_boundedness_of_truncation_operator}, the operator $\Op{b_k}$ is bounded, uniformly in $t$ and $k$.
\end{proposition}

In what follows, we keep using the subscript $j$, where $j=0$ for $k<0$ and $j=1$ for $k\geqslant 0$. By definition,
$$ b_k(t,x,\xi) = \chi\left( \frac{(P_j+\nu_j X)\frac{x}{|x|}+(P_j+\nu_j\Xi)\frac{\xi}{|\xi|}}{t^{\delta-\frac{1}{2}} \Lambda(\xi)} \right) H \left( \frac{X-\Xi}{t^{\delta-\frac{1}{2}} \Lambda(\xi)} \right)\varphi\left(\frac{\xi}{\sqrt{t}2^{k}}\right) \psi\left(\frac{X}{2^{kp_j}}\right). $$
When $P_j =0$, $b_k$ is supported for
\begin{align*}
\left| \frac{x}{|x|}+\frac{\xi}{|\xi|} \right| \leqslant& \frac{1}{X} \left| X\frac{x}{|x|}+\Xi\frac{\xi}{|\xi|} \right| + \frac{1}{X}\left| (X-\Xi)\frac{\xi}{|\xi|} \right| \\ \lesssim& \frac{1}{X}t^{\delta-\frac{1}{2}}\Lambda(\xi) \sim 2^{-kp_j}t^{\delta-\frac{1}{2}}(\sqrt{t}2^{k})^{\sigma_j} = t^{\delta+\frac{\sigma_j}{2}-\frac{1}{2}} 2^{k(\sigma_j-p_j)}.
\end{align*}
When $P_j \neq 0$, we observe that $2^{k\sigma_j} \leqslant 2^{kp_j} \leqslant 1$, due to the choice $p_1<0<p_0$. By choosing $\Supp\varphi$ small, which allows us to take $\Supp\psi$ small, we have $P_j+\nu_j X \sim 1$. Thus, $b_k$ is supported for
\begin{align*}
\left| \frac{x}{|x|}+\frac{\xi}{|\xi|} \right| \leqslant& \frac{1}{P_j+\nu_j X} \left| (P_j+\nu_j X)\frac{x}{|x|}+(P_j+\nu_j\Xi)\frac{\xi}{|\xi|} \right| + \frac{1}{P_j+\nu_j X}\left| (X-\Xi)\frac{\xi}{|\xi|} \right| \\ 
\lesssim& \frac{1}{P_j+\nu_j X}t^{\delta-\frac{1}{2}}\Lambda(\xi) \sim t^{\delta-\frac{1}{2}}(\sqrt{t}2^{k})^{\sigma_j} = t^{\delta+\frac{\sigma_j}{2}-\frac{1}{2}} 2^{k\sigma_j}.
\end{align*}

If $t 2^{k(p_j+1)} \leqslant 1$, by setting $\mu:= t^{\delta+\frac{\sigma_j}{2}-\frac{1}{2}}2^{k(\sigma_j-p_j)} \in ]0,1[$ in the case $P_j=0$, and $\mu:= t^{\delta+\frac{\sigma_j}{2}-\frac{1}{2}}2^{k\sigma_j} \in ]0,1[$ in the case $P_j \neq 0$, one may check that, for all $\alpha,\beta\in\mathbb{N}^{d-1}$ and $N\in\mathbb{N}$,
$$ \left| \partial_{\omega}^{\alpha} \partial_{\theta}^{\beta} b_k(r\omega,\rho\theta) \right| \lesssim \mu^{-|\alpha|-|\beta|} \mathbbm{1}_{r\lesssim\sqrt{t}2^{kp}}\mathbbm{1}_{\rho\sim\sqrt{t}2^k} \langle \frac{d(\omega,-\theta)}{\mu} \rangle^{-N}. $$
By applying Lemma \ref{lem_tech_boundedness_of_symbol_polar}, we have
$$ \|\Op{b_k}\|_{\mathcal{L}(L^2)} \lesssim \| \mathbbm{1}_{r\lesssim\sqrt{t}2^{kp_j}}\mathbbm{1}_{\rho\sim\sqrt{t}2^k} \|_{L^2(\mathbb{R}_+^2)} = C \sqrt{t 2^{k(p_j+1)}} \leqslant C. $$

If $t 2^{k(p_j+1)} > 1$, we decompose $b_k$ as the sum of $b_k',b_k''$, which are defined by
\begin{equation*}
\begin{aligned}
b_k(t,x,\xi) &= b_k'(t,x,\xi) + b_k''(t,x,\xi), \\
b_k'(t,x,\xi) &= b_k \chi_0\left( \sqrt{t}2^{k\frac{1-p_j}{2}} (X-\Xi) \right),
\end{aligned}
\end{equation*}
where $\chi_0 \in C_c^\infty(\mathbb{R})$ is radial, supported in a neighborhood of zero, and equal to $1$ near zero.

Clearly, $b_k''$ is smooth on $\mathbb{R}^{2d}$. Thus, to prove the boundedness of $\Op{b_k''}$, we may apply Calderon-Vaillancourt Theorem (Lemma \ref{lem_tech_C-V}). By definition, $b_k''$ reads
\begin{equation*}
\begin{aligned}
&\chi\left( \frac{(P_j+\nu_j X)\frac{x}{|x|}+(P_j+\nu_j\Xi)\frac{\xi}{|\xi|}}{t^{\delta-\frac{1}{2}} \Lambda(\xi)} \right) H \left( \frac{X-\Xi}{t^{\delta-\frac{1}{2}} \Lambda(\xi)} \right) \\
&\hspace{6em}\times\varphi\left(\frac{\xi}{\sqrt{t}2^{k}}\right) \psi\left(\frac{X}{2^{kp_j}}\right) (1-\chi_0)\left( \sqrt{t}2^{k\frac{1-p_j}{2}} (X-\Xi) \right),
\end{aligned}
\end{equation*}
which is supported for
$$ t^{-\frac{1}{2}}2^{k\frac{p_j-1}{2}} \lesssim |X-\Xi| \lesssim t^{\delta-\frac{1}{2}}\Lambda(\xi) \sim t^{\delta+\frac{\sigma_j}{2}-\frac{1}{2}}2^{k\sigma_j}. $$
As a consequence, $t^{-(\delta+\frac{\sigma_j}{2})} \lesssim 2^{k(\sigma_j+\frac{1-p_j}{2})}$.

For each $\partial_x$, when it acts on $\chi$, one obtains in its bound an extra factor of size
$$ \frac{1}{\sqrt{t}}t^{\frac{1}{2}-\delta}\Lambda(\xi)^{-1} \sim t^{-(\delta+\frac{\sigma_j}{2})}2^{-kp_j} \lesssim 2^{k\frac{1-p_j}{2}} 2^{k(\sigma_j-p_j)} \leqslant 2^{k\frac{1-p_j}{2}}. $$
When $\partial_x$ acts on $\psi$ and $(1-\chi_0)$ factors, we gain $t^{-\frac{1}{2}}2^{-kp_j} < 2^{k\frac{p_j+1}{2}-kp_j} = 2^{k\frac{1-p_j}{2}}$ and $2^{k\frac{1-p_j}{2}}$, respectively. For each $\partial_\xi$, similarly, one gains $2^{k\frac{p_j-1}{2}}$ in its bound. Namely, for all $\alpha,\beta\in\mathbb{N}^{d}$,
$$ \left| \partial_x^\alpha \partial_\xi^\beta b_k''(x,\xi) \right| \lesssim 2^{k\frac{1-p_j}{2}(|\alpha|-|\beta|)}. $$
The uniform boundedness of $\Op{b_k''}$ follows from Lemma \ref{lem_tech_change_of_scaling} and Calderon-Vaillancourt Theorem (Lemma \ref{lem_tech_C-V}).

It remains to study the symbol $b_k'$. To overcome the singularity near $X=\Xi$, we will apply Lemma \ref{lem_tech_boundedness_of_symbol_polar} with the same setting of $\mu$ as in previous paragraph, namely $\mu:= t^{\delta+\frac{\sigma_j}{2}-\frac{1}{2}}2^{k(\sigma_j-p_j)} \in ]0,1[$ in the case $P_j=0$, and $\mu:= t^{\delta+\frac{\sigma_j}{2}-\frac{1}{2}}2^{k\sigma_j} \in ]0,1[$ in the case $P_j \neq 0$. It is easy to check that, for all $\alpha,\beta\in\mathbb{N}^{d-1}$ and $N\in\mathbb{N}$,
$$ \left| \partial_{\omega}^{\alpha} \partial_{\theta}^{\beta} b_k(r\omega,\rho\theta) \right| \lesssim \mu^{-|\alpha|-|\beta|} h_k(r,\rho) \langle \frac{d(\omega,-\theta)}{\mu} \rangle^{-N}, $$
where 
$$ h_k(r,\rho) = \sum_{n\sim \sqrt{t} 2^{k\frac{p_j+1}{2}}} \mathbbm{1}_{I_n}(r) \mathbbm{1}_{J_n}(\rho), $$
which is exactly the same one defined in \eqref{eq_bdd_def_of_h_k}. We have seen that the operator with symbol $h_k$ is bounded on $L^2(\mathbb{R}_+)$ uniformly in $k$. By Lemma \ref{lem_tech_boundedness_of_symbol_polar}, we may conclude the uniform-in-$k$ boundedness of $b_k'$ and the proof of Proposition \ref{prop_alt_boundedness_of_b_k}, hence Proposition \ref{prop_alt_boundedness_of_truncation_operator} is completed.

%% file: main/Limit_of_energy.tex
\section{Limit of truncated energy}\label{sect_limit_of_energy}

In this section, we will complete the proof of Theorem \ref{thm_main}, \ref{thm_mod_main}, and \ref{thm_alt_main} by calculating the limit of truncated energy for some regular initial data $u_0$. These three results will follow from the following proposition.

\begin{proposition}\label{prop_limit_of_energy}
Let $a_{\chi,\delta,\Lambda}(t)$ be the symbol defined by
\begin{equation}\label{eq_lim_truncation_symbol}
a= a_{\chi,\delta,\Lambda}(t,x,\xi) = \chi\left(\frac{x+tP'(\xi)}{|t|^{\frac{1}{2}+\delta}\Lambda(|t|^\frac{1}{2}\xi)}\right)\mathbbm{1}_{|x|>|t||P'(\xi)|},
\end{equation}
where $P\in C^\infty(\mathbb{R}^d\backslash\{0\})$ is assumed to be a real radial function satisfying that $P''(\rho)\neq 0$ for all $\rho\in]0,\infty[$. Furthermore, we assume that $\chi\in C_c^\infty(\mathbb{R}^d)$ is real and radial with $\chi(0)=1$ and that $\Lambda$ verifies condition \eqref{hyp_fractional-type_symbol_extra_factor} without any restriction in $\sigma_0,\sigma_1$.

With these settings, if there exists $t_0\gg 1$, such that $\Op{a(t)}$ is bounded on $L^2$ uniformly in $|t|>t_0$, for all $u_0\in L^2$, the limits \eqref{eq_alt_limit_of_energy_subcritical}, \eqref{eq_alt_limit_of_energy_critical}, and \eqref{eq_alt_limit_of_energy_supercritical} hold true.
\end{proposition}

\begin{corollary}\label{cor_limit_of_energy}
We consider the same symbol $a$ with an extra truncation in $\xi$, i.e.
$$ \tilde{a} = \tilde{a}_{\chi,\delta,\Lambda}(t,x,\xi) = \chi\left(\frac{x+tP'(\xi)}{|t|^{\frac{1}{2}+\delta}\Lambda(|t|^\frac{1}{2}\xi)}\right)\mathbbm{1}_{|x|>|t||P'(\xi)|} (1-\chi_l)\left( \frac{\xi}{|t|^{-\epsilon_0}} \right) \chi_h\left( \frac{\xi}{|t|^{\epsilon_1}} \right) $$
where $\chi_l,\chi_h \in C_c^\infty(\mathbb{R}^d)$ are equal to $1$ in a neighborhood of zero, and $\epsilon_0,\epsilon_1>0$. 

If $P$, $\Lambda$ satisfy the same conditions as in Proposition \ref{prop_limit_of_energy} and $\Op{\tilde{a}_{\delta,\chi,\Lambda}(t)}$ is bounded uniformly in $|t|>t_0 \gg 1$, then, for all $u_0\in L^2$, the limits \eqref{eq_alt_limit_of_energy_subcritical}, \eqref{eq_alt_limit_of_energy_critical}, and \eqref{eq_alt_limit_of_energy_supercritical} hold true with $E^\text{alt}_{\delta,\chi,\Lambda}(t)$ replaced by $\|\Op{\tilde{a}_{\delta,\chi,\Lambda}(t)}u(t)\|_{L^2}^2$.
\end{corollary}

In order to complete the proof of Theorem \ref{thm_main} and \ref{thm_alt_main}, we may combine Proposition \ref{prop_boundedness_of_truncated_operator}, \ref{prop_alt_boundedness_of_truncation_operator} with Proposition \ref{prop_limit_of_energy}, where we need to take $\Lambda \equiv 1$ in the proof of Theorem \ref{thm_main}. In the same way, Theorem \ref{thm_mod_main} follows from Proposition \ref{prop_boundedness_of_truncated_operator_with_truncated_frequency} and Corollary \ref{cor_boundedness_of_truncated_operator_with_mu} with $\Lambda \equiv 1$.

Before calculating the limit of truncated energy, we remark that, in the hypothesis of Proposition \ref{prop_limit_of_energy}, $\Op{a(t)}$ (or $\Op{\tilde{a}(t)}$) is assumed to be bounded uniformly in $|t|>t_0 \gg 1$, which allows us to replace general $u_0\in L^2$ by those belonging to some dense subset of $L^2$. In what follows, we may assume that $\hat{u}_0\in C_c^\infty(\mathbb{R}\backslash\{0\})$. As a consequence, by taking $t_0 \gg 1$, 
$$ (1-\chi_l)\left( \frac{\xi}{|t|^{-\epsilon_0}} \right) \chi_h\left( \frac{\xi}{|t|^{\epsilon_1}} \right) \hat{u}_0(\xi) = \hat{u}_0(\xi), $$
which proves Corollary \ref{cor_boundedness_of_truncated_operator_with_mu} from Proposition \ref{prop_limit_of_energy}.

\subsection{Super-critical case \texorpdfstring{$0<\delta+\frac{\sigma_1}{2}<\frac{1}{2}$}{0<delta+sigma/2<1/2}}

In this part, we will study the case $\delta+\frac{\sigma_1}{2}\in]0,\frac{1}{2}[$ (associated to limit \eqref{eq_limit_of_energy_supercritical}, \eqref{eq_mod_limit_of_energy_supercritical}, or \eqref{eq_alt_limit_of_energy_supercritical}) by following the same method introduced in \cite{delort2022microlocal}.

By definition, the truncated energy introduced in \eqref{eq_truncated_energy} is
\begin{equation*}
\begin{aligned}
E_{\chi,\delta,\Lambda}(\epsilon t) =& \|\Op{a_{\chi,\delta,\Lambda}(\epsilon t)}u(\epsilon t)\|_{L^2}^2 \\
=& \frac{1}{(2\pi)^{2d}} \int e^{ix\cdot(\xi-\xi')} e^{i\epsilon t (P(\xi)-P(\xi'))} \chi\left(\frac{x+\epsilon tP'(\xi)}{t^{\frac{1}{2}+\delta}\Lambda(t^\frac{1}{2}\xi)}\right) \\
&\hspace{4em}\times\chi\left(\frac{x+\epsilon tP'(\xi')}{t^{\frac{1}{2}+\delta}\Lambda(t^\frac{1}{2}\xi)}\right)\mathbbm{1}_{\frac{|x|}{t}>|P'(\xi)|,|P'(\xi')|} \hat{u}_0(\xi) \overline{\hat{u}_0(\xi')} dx d\xi d\xi',
\end{aligned}
\end{equation*}
where $t \gg 1$ and $\epsilon=\pm$. In polar system $x=r\omega$, $\xi=\rho\theta$, $\xi'=\rho'\theta'$, the integral above can be written as
\begin{equation*}
\begin{aligned}
& \frac{1}{(2\pi)^{2d}} \int e^{i(r\rho\omega\cdot\theta-r\rho'\omega\cdot\theta')} e^{i\epsilon t (P(\rho)-P(\rho'))} \chi\left(\frac{r\omega+\epsilon tP'(\rho)\theta}{t^{\frac{1}{2}+\delta}\Lambda(t^\frac{1}{2}\rho)}\right) \chi\left(\frac{r\omega+\epsilon tP'(\rho')\theta'}{t^{\frac{1}{2}+\delta}\Lambda(t^\frac{1}{2}\rho')}\right) \\
&\hspace{4em}\times\mathbbm{1}_{\frac{r}{t}>P'(\rho),P'(\rho')} \hat{u}_0(\rho\theta) \overline{\hat{u}_0(\rho'\theta')} (r\rho\rho')^{d-1} d\theta d\theta' d\omega dr d\rho d\rho'.
\end{aligned}
\end{equation*}
We firstly focus on the integral in $\theta$, with integral in $\theta'$ treated in exactly the same way,
\begin{equation}\label{eq_lim_crit_int_in_theta}
\int e^{ir\rho\omega\cdot\theta} \chi\left(\frac{r\omega+\epsilon tP'(\rho)\theta}{t^{\frac{1}{2}+\delta}\Lambda(t^\frac{1}{2}\rho)}\right) \hat{u}_0(\rho\theta) d\theta.
\end{equation}

In this part, we always set
$$ \mu=t^{\delta+\frac{\sigma_1}{2}-\frac{1}{2}} \in ]0,1[, $$
which is strictly positive and small, since we may choose $t>t_0 \gg 1$. Due to Lemma \ref{lem_stat_int_on_angular_limit_part}, \eqref{eq_lim_crit_int_in_theta} can be written as the sum of 
$$ (2\pi)^\frac{d-1}{2} e^{i\epsilon\frac{\pi}{4}(d-1)} e^{-i\epsilon r\rho} (r\rho)^{-\frac{d-1}{2}}\chi\left(\frac{r - t P'(\rho)}{t^{\frac{1}{2}+\delta}\Lambda(t^\frac{1}{2}\rho)}\right) \hat{u}_0(-\epsilon\rho\omega) \kappa\left(\frac{r}{t}\right) $$
and a remainder
$$ e^{-i\epsilon r\rho} \mu^{d-1} S_{-\frac{d+1}{2}}(\omega,\mu,\rho,\frac{r}{t}-P'(\rho),t; r\rho\mu^2) \kappa\left(\frac{r}{t}\right), $$
where $\kappa \in C_c^\infty(]0,\infty[)$ equals to $1$ in a neighborhood of $1$, and $S_{m}(\omega,\mu,\rho,r',t; \zeta)$ is supported for $\zeta>c>0$, $\rho\sim 1$ and $|r'| \lesssim \mu$ and satisfies for all $\alpha\in\mathbb{N}^{d-1}$, $j,k,l,\gamma\in\mathbb{N}$,
$$ | \partial_{\omega}^{\alpha} \partial_{\mu}^j \partial_{\rho}^{k} \partial_{r'}^{l} \partial_{\zeta}^{\gamma} S_{m} |
\leqslant C \mu^{-(|\alpha|+j+l)} \langle\zeta\rangle^{m-\gamma}. $$
Remark that it is harmless to add an extra factor $\kappa$, since the integrand of \eqref{eq_lim_crit_int_in_theta} is supported for $r\sim t$, which is a consequence of the cut-off $\chi$ together with $t \gg 1$, $\delta+\frac{\sigma_1}{2}<\frac{1}{2}$, and $\rho\sim 1$. We may repeat this argument for the integral in $\theta'$ and the truncated energy $E_{\chi,\delta,\Lambda}(\epsilon t)$ can be decomposed into a principal part
\begin{equation}\label{eq_lim_sup_crit_energy_principal}
\begin{aligned}
& \frac{1}{(2\pi)^{d+1}} \int e^{-i\epsilon r(\rho-\rho')} e^{i\epsilon t (P(\rho)-P(\rho'))} \chi\left(\frac{r- tP'(\rho)}{t^{\frac{1}{2}+\delta}\Lambda(t^\frac{1}{2}\rho)}\right) \chi\left(\frac{r- tP'(\rho')}{t^{\frac{1}{2}+\delta}\Lambda(t^\frac{1}{2}\rho')}\right) \\
&\hspace{4em}\times\mathbbm{1}_{\frac{r}{t}>P'(\rho),P'(\rho')} \hat{u}_0(-\epsilon\rho\omega) \overline{\hat{u}_0(-\epsilon\rho'\omega)} (\rho\rho')^{\frac{d-1}{2}} \kappa^2\left(\frac{r}{t}\right) d\omega dr d\rho d\rho',
\end{aligned}
\end{equation}
and remainders
\begin{equation*}
\begin{aligned}
& \frac{1}{(2\pi)^{d+1}} \int e^{-i\epsilon r(\rho-\rho')} e^{i\epsilon t (P(\rho)-P(\rho'))} \mu^{2(d-1)} \kappa^2\left(\frac{r}{t}\right) S_{m}(\omega,\mu,\rho,\frac{r}{t}-P'(\rho),t; r\rho\mu^2)  \\
&\hspace{4em}\times S_{m'}(\omega,\mu,\rho',\frac{r}{t}-P'(\rho'),t; r\rho'\mu^2) \mathbbm{1}_{\frac{r}{t}>P'(\rho),P'(\rho')} (r\rho\rho')^{d-1} d\omega dr d\rho d\rho',
\end{aligned}
\end{equation*}
where $(m,m')$ takes values among $(-\frac{d-1}{2},-\frac{d+1}{2}),(-\frac{d+1}{2},-\frac{d-1}{2})$, $(-\frac{d+1}{2},-\frac{d+1}{2})$. Note that due to the condition $\delta+\frac{\sigma_1}{2}>\frac{1}{2}$, we have
$$ r\rho\mu^2 = r\rho t^{2(\delta+\frac{\sigma_1}{2})-1} \sim t^{2(\delta+\frac{\sigma_1}{2})} >c>0. $$
The sum of these remainders can be simplified as
\begin{equation}\label{eq_lim_sup_crit_energy_remainder}
\begin{aligned}
& \int e^{-i\epsilon r(\rho-\rho')} e^{i\epsilon t (P(\rho)-P(\rho'))} \frac{1}{\mu^2 r} \Sigma(\omega,\mu,\rho,\rho',r,t;\frac{r}{t}-P'(\rho),\frac{r}{t}-P'(\rho'))  \\
&\hspace{20em}\mathbbm{1}_{\frac{r}{t}>P'(\rho),P'(\rho')} d\omega dr d\rho d\rho',
\end{aligned}
\end{equation}
where $\Sigma(\omega,\mu,\rho,\rho',r,t; s,s')$ is supported for $r \sim t$, $\rho,\rho'\sim 1$ and $|s|,|s'| \lesssim \mu$ and satisfies for all $\alpha\in\mathbb{N}^{d-1}$, $j,k,k',l,\gamma,\gamma'\in\mathbb{N}$,
$$ | \partial_{\omega}^{\alpha} \partial_{\mu}^j \partial_{\rho}^{k} \partial_{\rho'}^{k'} \partial_{r}^{l} \partial_{s}^{\gamma} \partial_{s'}^{\gamma'} \Sigma |
\lesssim \mu^{-(|\alpha|+j+\gamma+\gamma')} t^{-l}. $$

Before stepping further, we introduce the following integral:
\begin{equation}\label{eq_lim_def_of_standard_integral}
\begin{aligned}
I(t,\epsilon_1,\epsilon_1',\epsilon_2,\epsilon_2';F) :=& \int e^{i[r(\epsilon_1\rho+\epsilon_1'\rho') - t(\epsilon_2 P(\rho)+ \epsilon_2' P(\rho'))]} \mathbbm{1}_{\frac{r}{t}>P'(\rho),P'(\rho')}\\
&\hspace{1em}\times F(\rho,\rho',r,t;r-\epsilon_1\epsilon_2 tP'(\rho), r-\epsilon_1'\epsilon_2' tP'(\rho'))  dr d\rho d\rho'.
\end{aligned}
\end{equation}
The limit of such integral has been studied in \cite{delort2022microlocal} for strictly convex $P$, while the concave case can be studied with almost the same argument. To be precise,

\begin{proposition}\label{prop_lim_limit_of_integral}
Let $F(\rho,\rho',r,t;\zeta,\zeta')$ be a smooth function on $\mathbb{R}_+^{4}\times\mathbb{R}^2$ and $\delta'\in ]\frac{1}{2},1[$. Assume that $F$ is supported for
$$ \rho,\rho'\sim 1,\ r \sim t,\ |\zeta|,|\zeta'|\lesssim t^{\delta'}, $$
and for all $j,j',k,\gamma,\gamma'\in\mathbb{N}$,
$$ | \partial_{\rho}^{j} \partial_{\rho'}^{j'} \partial_{r}^{k} \partial_{\zeta}^{\gamma} \partial_{\zeta'}^{\gamma'} F(\rho,\rho',r,t;\zeta,\zeta') |
\lesssim t^{-\delta'(k+\gamma+\gamma')}. $$
We assume further that the following point-wise limit exists
$$ \lim_{t\rightarrow+\infty} F(\rho,\rho',r\sqrt{t}+tP'(\rho'),t;\zeta\sqrt{t},\zeta'\sqrt{t}) = F_0(\rho,\rho'). $$
Under all the assumptions above, we have
$$ \lim_{t\rightarrow+\infty} I(t,-\epsilon,\epsilon,-\epsilon,\epsilon;F) = \frac{\pi}{2} \int_0^\infty F_0(\rho,\rho) d\rho, $$
for all $\epsilon=\pm$ and $P \in C^\infty$ with $P''>0$ or $P''<0$.
\end{proposition}
The proof for strictly convex $P$ follows from that of Proposition 3.1.3 in \cite{delort2022microlocal}. As for concave case, we will give a brief proof in Appendix \ref{sect_proof_of_limit_of_integral}. Remark that we compute the limit of $I(t,-\epsilon,\epsilon,-\epsilon,\epsilon;F)$ for both signs $\epsilon=\pm 1$, while in \cite{delort2022microlocal} only the limit of the sum of these two terms was determined. The proof of our stronger result is not essentially different from the one in \cite{delort2022microlocal} and we shall explain the modification one has to make to the argument in Appendix \ref{sect_proof_of_limit_of_integral}. 

With the notations above, the truncated energy $E_{\chi,\delta,\Lambda}(\epsilon t)$ given by the sum of \eqref{eq_lim_sup_crit_energy_principal} and \eqref{eq_lim_sup_crit_energy_remainder} equals to
$$ E_{\chi,\delta,\Lambda}(\epsilon t) = I(t,-\epsilon,\epsilon,-\epsilon,\epsilon;F) + I(t,-\epsilon,\epsilon,-\epsilon,\epsilon;F_R),$$
where
\begin{align*}
F(\rho,\rho',r,t;\zeta,\zeta') =& \frac{1}{(2\pi)^{d+1}} \kappa^2\left(\frac{r}{t}\right) \int \chi\left(\frac{\zeta}{t^{\frac{1}{2}+\delta}\Lambda(t^\frac{1}{2}\rho)}\right) \chi\left(\frac{\zeta'}{t^{\frac{1}{2}+\delta}\Lambda(t^\frac{1}{2}\rho')}\right) \\
&\hspace{8em}\times\hat{u}_0(-\epsilon\rho\omega) \overline{\hat{u}_0(-\epsilon\rho'\omega)} (\rho\rho')^{\frac{d-1}{2}} d\omega, \\
F_R(\rho,\rho',r,t;\zeta,\zeta') =& t^{-2(\delta+\frac{\sigma_1}{2})} \int \frac{t}{r} \Sigma(\omega,t^{\delta+\frac{\sigma_1}{2}-\frac{1}{2}},\rho,\rho',r,t;\frac{\zeta}{t},\frac{\zeta'}{t}) d\omega.
\end{align*}
It is easy to check that $F,F_R$ satisfy the conditions of Proposition \ref{prop_lim_limit_of_integral} with 
$$\delta'=\delta+\frac{\sigma_1}{2}+\frac{1}{2}.$$ 
Note that due to condition $\delta+\frac{\sigma_1}{2} \in ]0,\frac{1}{2}[$, we have $\delta'\in ]\frac{1}{2},1[$, which is required by Proposition \ref{prop_lim_limit_of_integral}. The corresponding limit is
$$F_0(\rho,\rho')=\frac{1}{(2\pi)^{d+1}} \int \hat{u}_0(-\epsilon\rho\omega) \overline{\hat{u}_0(-\epsilon\rho'\omega)} (\rho\rho')^{\frac{d-1}{2}} d\omega $$
and $0$, respectively. Therefore, we may conclude \eqref{eq_limit_of_energy_supercritical}, \eqref{eq_mod_limit_of_energy_supercritical}, and \eqref{eq_alt_limit_of_energy_supercritical} by Proposition \ref{prop_lim_limit_of_integral}. In fact, with $\epsilon=\pm$, one has
\begin{align*}
\lim_{t\rightarrow\pm\infty} E_{\chi,\delta}(t) = \lim_{t\rightarrow+\infty} E_{\chi,\delta}(\epsilon t) =& \lim_{t\rightarrow+\infty} I(t,-\epsilon,\epsilon,-\epsilon,\epsilon;F) + \lim_{t\rightarrow+\infty} I(t,-\epsilon,\epsilon,-\epsilon,\epsilon;F_R) \\
=& \frac{\pi}{2} \int_0^\infty F_0(\rho,\rho) d\rho + 0 \\
=& \frac{1}{4} \frac{1}{(2\pi)^d} \int \hat{u}_0(-\epsilon\rho\omega) \overline{\hat{u}_0(-\epsilon\rho\omega)} \rho^{d-1} d\omega d\rho = \frac{1}{4} \|u_0\|_{L^2}^2.
\end{align*}

\subsection{Sub-critical and critical case \texorpdfstring{$\delta+\frac{\sigma_1}{2} \leqslant 0$}{delta+sigma/2<=0}}

In the rest of this section, we will study, under the condition $\delta+\frac{\sigma_1}{2}\leqslant 0$, the truncated energy $E_{\chi,\delta,\Lambda}(u_0,\epsilon t)$, with $\epsilon = \pm$, $t>t_0 \gg 1$. Here, we only write the proof of the case $P''>0$, while the case $P''<0$ can be calculated in exactly the same way.

By definition \eqref{eq_truncated_energy}, the truncated energy $E_{\chi,\delta,\Lambda}(u_0,\epsilon t)$ equals to
\begin{equation*}
\begin{aligned}
\frac{1}{(2\pi)^{2d}} \int e^{ix\cdot(\xi-\xi')} e^{i\epsilon t \left( P(\xi) - P(\xi') \right)} &\chi\left( \frac{x+\epsilon tP'(\xi)}{t^{\frac{1}{2}+\delta}\Lambda(t^\frac{1}{2}\xi)} \right) \chi\left( \frac{x+\epsilon tP'(\xi')}{t^{\frac{1}{2}+\delta}\Lambda(t^\frac{1}{2}\xi')} \right) \\
&\times\mathbbm{1}_{|\frac{x}{t}|>|P'(\xi)|,|P'(\xi')|} \hat{u}_0(\xi) \overline{\hat{u}_0(\xi')} d\xi d\xi' dx,
\end{aligned}
\end{equation*}
which can be rewritten in polar system $x=r\omega,\ \xi=\rho\theta,\ \xi'=\rho'\theta'$ as
\begin{equation*}
\begin{aligned}
\frac{1}{(2\pi)^{2d}} \int e^{ir\omega\cdot(\rho\theta-\rho'\theta')} &e^{i\epsilon t \left( P(\rho)-P(\rho') \right)} \chi\left( \frac{r\omega+\epsilon tP'(\rho)\theta}{t^{\frac{1}{2}+\delta}\Lambda(t^\frac{1}{2}\rho)} \right) \chi\left( \frac{r\omega+\epsilon tP'(\rho')\theta'}{t^{\frac{1}{2}+\delta}\Lambda(t^\frac{1}{2}\rho')} \right) \\
&\times\mathbbm{1}_{\frac{r}{t}>P'(\rho),P'(\rho')} \hat{u}_0(\rho\theta) \overline{\hat{u}_0(\rho'\theta')} (r\rho\rho')^{d-1} dr d\theta d\theta' d\omega d\rho d\rho'.
\end{aligned}
\end{equation*}
We decompose $E_{\chi,\delta,\Lambda}(u_0,\epsilon t)$ as the sum of $E_\pm(\epsilon t)$, where $E_+$, $E_-$ are defined as integral over $\rho>\rho'$, $\rho'>\rho$, respectively, and the dependence on $\chi,\delta,\Lambda,u_0$ is omitted for the simplicity of notation. Since $E_- = \overline{E_+}$, it is enough to focus on the study of $E_+(\epsilon t)$. We first check that

\begin{lemma}
The integral
\begin{equation*}
\begin{aligned}
E_+(\epsilon t) = \frac{1}{(2\pi)^{2d}} \int e^{ir\omega\cdot(\rho\theta-\rho'\theta')} &e^{i\epsilon t \left( P(\rho)-P(\rho') \right)} \chi\left( \frac{r\omega+\epsilon tP'(\rho)\theta}{t^{\frac{1}{2}+\delta}\Lambda(t^\frac{1}{2}\rho)} \right) \chi\left( \frac{r\omega+\epsilon tP'(\rho')\theta'}{t^{\frac{1}{2}+\delta}\Lambda(t^\frac{1}{2}\rho')} \right) \\
&\times\mathbbm{1}_{\frac{r}{t}>P'(\rho)} \mathbbm{1}_{\rho>\rho'} \hat{u}_0(\rho\theta) \overline{\hat{u}_0(\rho'\theta')} (r\rho\rho')^{d-1} dr d\theta d\theta' d\omega d\rho d\rho'
\end{aligned}
\end{equation*}
equals, up to some $O(t^{-\frac{1}{2}})$ terms, to
\begin{equation}\label{eq_lim_sub&crit_simplified_energy}
\begin{aligned}
\frac{1}{(2\pi)^{2d}} \int_{\theta,\theta'\in\sqrt{t}(\epsilon\omega+\mathbb{S}^{d-1})} &e^{i \left[ \sqrt{t}P'(\rho)\rho\omega\cdot(\theta-\theta') + r\rho\omega\cdot(\theta-\theta') + wP'(\rho)\theta'\cdot\omega \right]} e^{-\epsilon i[rw + \frac{1}{2}P''(\rho)w^2]}  \\
&\times \chi\left( \frac{r\omega+\epsilon P'(\rho)\theta}{t^\delta\Lambda(t^\frac{1}{2}\rho)} \right) \chi\left( \frac{(r+P''(\rho)w)\omega+\epsilon P'(\rho)\theta'}{t^\delta\Lambda(t^\frac{1}{2}\rho)} \right) \\
&\times  \mathbbm{1}_{r>0}\mathbbm{1}_{\rho>\frac{w}{\sqrt{t}}>0} \left|\hat{u}_0(-\epsilon\rho\omega)\right|^2 \left(P'(\rho)\right)^{d-1}\rho^{2(d-1)} d\theta d\theta' dr dw d\omega d\rho.
\end{aligned}
\end{equation}
\end{lemma}
\begin{proof}
To begin with, via change of variable,
$$ r \rightarrow rt^{\frac{1}{2}} + tP'(\rho),\ \ \rho' \rightarrow \rho-\frac{w}{t^\frac{1}{2}}, $$
the integral $E_+(\epsilon t)$ can be rewritten as
\begin{equation*}
\begin{aligned}
\frac{t^{d-1}}{(2\pi)^{2d}} \int &e^{i \left[ tP'(\rho)\rho\omega\cdot(\theta-\theta') + \sqrt{t}r\rho\omega\cdot(\theta-\theta') + \sqrt{t}wP'(\rho)\theta'\cdot\omega + rw\theta'\cdot\omega \right]}   \\
&\times e^{\epsilon it (P(\rho)-P(\rho-\frac{w}{\sqrt{t}}))}\chi\left( \frac{r\omega+\sqrt{t}P'(\rho)(\omega+\epsilon\theta)}{t^{\delta}\Lambda(t^\frac{1}{2}\rho)} \right) \\
&\times\chi\left( \frac{r\omega+\sqrt{t}P'(\rho)(\omega+\epsilon\theta')-\epsilon\sqrt{t}(P'(\rho)-P'(\rho-t^{-\frac{1}{2}}w))\theta'}{t^{\delta}\Lambda(t^\frac{1}{2}\rho-w)} \right) \\
&\times  \mathbbm{1}_{r>0}\mathbbm{1}_{\rho>\frac{w}{\sqrt{t}}>0} \hat{u}_0(\rho\theta) \overline{\hat{u}_0((\rho-t^{-\frac{1}{2}}w)\theta')} \\
&\times\left(\frac{r}{\sqrt{t}}+P'(\rho)\right)^{d-1}\rho^{d-1}(\rho-t^{-\frac{1}{2}}w)^{d-1}  d\theta d\theta' dr dw d\omega d\rho.
\end{aligned}
\end{equation*}
Due to $\chi$, $\hat{u}_0$ factors, the integrand is supported for
\begin{equation}\label{eq_lim_sub&crit_support_of_energy}
\begin{aligned}
&0<r,w\lesssim t^{\delta+\frac{\sigma_1}{2}},\ \rho\sim 1, \\
&|\omega + \epsilon\theta|,|\omega + \epsilon\theta'| \lesssim t^{\delta+\frac{\sigma_1}{2}-\frac{1}{2}}.
\end{aligned}
\end{equation}
In fact, $\rho\sim 1$ follows directly from the fact that $\hat{u}_0$ is compactly supported away from zero. Since $\chi$ is compactly supported, the first $\chi$ factor implies that
\begin{align*}
r =& \left| |(r+\sqrt{t}P'(\rho))\omega| - |\sqrt{t}P'(\rho)\theta| \right| \\
\leqslant& \left| r\omega+\sqrt{t}P'(\rho)(\omega+\epsilon\theta) \right| \lesssim t^\delta \Lambda(t^{\frac{1}{2}}\rho) \sim t^{\delta+\frac{\sigma_1}{2}},
\end{align*}
which further implies that
\begin{align*}
\sqrt{t}|\omega + \epsilon\theta| \lesssim& |\sqrt{t}P'(\rho)(\omega + \epsilon\theta)| \\
\leqslant& \left| r\omega+\sqrt{t}P'(\rho)(\omega+\epsilon\theta) \right| + |r\omega| \lesssim t^\delta \Lambda(t^{\frac{1}{2}}\rho) \sim t^{\delta+\frac{\sigma_1}{2}}.
\end{align*}
By applying a similar argument to the second $\chi$ factor, we obtain
\begin{equation*}
\begin{aligned}
w &\sim \left|\sqrt{t}P'(\rho) - \sqrt{t}P'(\rho-t^{-\frac{1}{2}}w) \right|  \\
&\leqslant \left| (r+\sqrt{t}P'(\rho)) - \sqrt{t}P'(\rho-t^{-\frac{1}{2}}w)\right| +r\\
&\leqslant \left| (r+\sqrt{t}P'(\rho))\omega + \epsilon\sqrt{t}P'(\rho-t^{-\frac{1}{2}}w)\theta'\right| +r \lesssim t^{\delta+\frac{\sigma_1}{2}},
\end{aligned}
\end{equation*}
and that
\begin{equation*}
\begin{aligned}
\sqrt{t}|\omega+\epsilon\theta'| \sim& \left| \sqrt{t}P'(\rho)(\omega+\epsilon\theta') \right| \\
\leqslant& \left| r\omega+\sqrt{t}P'(\rho)(\omega+\epsilon\theta')-\epsilon\sqrt{t}(P'(\rho)-P'(\rho-t^{-\frac{1}{2}}w))\theta' \right| \\
&+ r + \left| \sqrt{t}(P'(\rho)-P'(\rho-t^{-\frac{1}{2}}w)) \right| \\
&\lesssim t^{\delta+\frac{\sigma_1}{2}} + r + w \lesssim t^{\delta+\frac{\sigma_1}{2}}.
\end{aligned}
\end{equation*}

As a result, the boundedness of integrand implies that
$$ E_{+}(\epsilon t) \lesssim t^{d-1} \times t^{2(\delta+\frac{\sigma_1}{2})} \times \left(\frac{1}{\sqrt{t}}t^{\delta+\frac{\sigma_1}{2}}\right)^{2(d-1)} = t^{2(\delta+\frac{\sigma_1}{2}) d}, $$
which tends to zero as $t\rightarrow +\infty$ when $\delta+\frac{\sigma_1}{2}<0$, and the limit \eqref{eq_limit_of_energy_subcritical}, \eqref{eq_mod_limit_of_energy_subcritical}, and \eqref{eq_alt_limit_of_energy_subcritical} follow. In the remaining of this section, we take $\delta+\frac{\sigma_1}{2}=0$.

The support of integrand also allows us to simplify $E_+(\epsilon t)$, up to some $O(t^{-\frac{1}{2}})$ terms, as
\begin{equation*}
\begin{aligned}
\frac{t^{d-1}}{(2\pi)^{2d}} \int &e^{i \left[ tP'(\rho)\rho\omega\cdot(\theta-\theta') + \sqrt{t}r\rho\omega\cdot(\theta-\theta') + \sqrt{t}wP'(\rho)\theta'\cdot\omega + rw\theta'\cdot\omega \right]}   \\
&\times e^{\epsilon i[\sqrt{t}P'(\rho)w - \frac{1}{2}P''(\rho)w^2]}\chi\left( \frac{r\omega+\sqrt{t}P'(\rho)(\omega+\epsilon\theta)}{t^\delta\Lambda(t^\frac{1}{2}\rho)} \right) \\
&\times\chi\left( \frac{r\omega+\sqrt{t}P'(\rho)(\omega+\epsilon\theta')-\epsilon P''(\rho)w\theta'}{t^\delta\Lambda(t^\frac{1}{2}\rho)} \right) \\
&\times  \mathbbm{1}_{r>0}\mathbbm{1}_{\rho>\frac{w}{\sqrt{t}}>0} \hat{u}_0(\rho\theta) \overline{\hat{u}_0(\rho\theta')} \left(P'(\rho)\right)^{d-1}\rho^{2(d-1)} d\theta d\theta' dr dw d\omega d\rho.
\end{aligned}
\end{equation*}
Here, we use the approximations
\begin{equation*}
\begin{aligned}
t\left(P(\rho)-P\left(\rho-\frac{w}{t^\frac{1}{2}}\right)\right) =& t\left(-P'(\rho)\left(-\frac{w}{t^\frac{1}{2}}\right) -\frac{P''(\rho)}{2}\frac{w^2}{t} + O\left(\frac{w^3}{t^\frac{3}{2}}\right)\right) \\
=& \sqrt{t}P'(\rho)w - \frac{1}{2} P''(\rho) w^2 + O\left(\frac{1}{\sqrt{t}}\right),
\end{aligned}
\end{equation*}
to simplify the phase and
\begin{equation*}
\begin{aligned}
\sqrt{t}\left(P'(\rho)-P'\left(\rho-\frac{w}{\sqrt{t}}\right)\right) =&
\sqrt{t}\left( -P''(\rho)\left(-\frac{w}{\sqrt{t}}\right) + O\left(\frac{w^2}{t}\right) \right) \\
=& P''(\rho)w + O\left( \frac{1}{\sqrt{t}}\right)
\end{aligned}
\end{equation*}
to simplify the argument of the second $\chi$.

By applying a change of variable in $\theta,\theta'$,
\begin{align*}
\theta &\mapsto t^{-\frac{1}{2}}\theta - \epsilon\omega, \\
\theta' &\mapsto t^{-\frac{1}{2}}\theta' - \epsilon\omega,
\end{align*}
we can rewrite $E_+(\epsilon t)$ as
\begin{equation*}
\begin{aligned}
\frac{1}{(2\pi)^{2d}} \int_{\theta,\theta'\in\sqrt{t}(\epsilon\omega+\mathbb{S}^{d-1})} &e^{i \left[ \sqrt{t}P'(\rho)\rho\omega\cdot(\theta-\theta') + r\rho\omega\cdot(\theta-\theta') + wP'(\rho)\theta'\cdot\omega + t^{-\frac{1}{2}}rw\theta'\cdot\omega \right]}   \\
&\times e^{-\epsilon i[rw + \frac{1}{2}P''(\rho)w^2]}\chi\left( \frac{r\omega+\epsilon P'(\rho)\theta}{t^\delta\Lambda(t^\frac{1}{2}\rho)} \right) \\
&\times\chi\left( \frac{(r+P''(\rho)w)\omega+\epsilon P'(\rho)\theta' -\epsilon t^{-\frac{1}{2}} P''(\rho)w \theta'}{t^\delta\Lambda(t^\frac{1}{2}\rho)} \right) \\
&\times  \mathbbm{1}_{r>0}\mathbbm{1}_{\rho>\frac{w}{\sqrt{t}}>0} \hat{u}_0(t^{-\frac{1}{2}}\rho\theta-\epsilon\rho\omega) \overline{\hat{u}_0(t^{-\frac{1}{2}}\rho\theta'-\epsilon\rho\omega)} \\
&\times\left(P'(\rho)\right)^{d-1}\rho^{2(d-1)} d\theta d\theta' dr dw d\omega d\rho + O(t^{-\frac{1}{2}}),
\end{aligned}
\end{equation*}
where the integrand is supported for $0<r,w\lesssim 1$, $\rho\sim 1$, and $|\theta|,|\theta'| \lesssim 1$ due to \eqref{eq_lim_sub&crit_support_of_energy} together with $\delta+\frac{\sigma_1}{2}=0$, which allows us to do another simplification and write $E_+(\epsilon t)$ as \eqref{eq_lim_sub&crit_simplified_energy}.
\end{proof}

Till now, we have managed to write $E_+(\epsilon t)$, up to some admissible terms, as \eqref{eq_lim_sub&crit_simplified_energy}, namely
\begin{equation*}
\begin{aligned}
\frac{1}{(2\pi)^{2d}} \int_{\theta,\theta'\in\sqrt{t}(\epsilon\omega+\mathbb{S}^{d-1})} &e^{i \left[ \sqrt{t}P'(\rho)\rho\omega\cdot(\theta-\theta') + r\rho\omega\cdot(\theta-\theta') + wP'(\rho)\theta'\cdot\omega \right]} e^{-\epsilon i[rw + \frac{1}{2}P''(\rho)w^2]}  \\
&\times \chi\left( \frac{r\omega+\epsilon P'(\rho)\theta}{t^\delta\Lambda(t^\frac{1}{2}\rho)} \right) \chi\left( \frac{(r+P''(\rho)w)\omega+\epsilon P'(\rho)\theta'}{t^\delta\Lambda(t^\frac{1}{2}\rho)} \right) \\
&\times  \mathbbm{1}_{r>0}\mathbbm{1}_{\rho>\frac{w}{\sqrt{t}}>0} \left|\hat{u}_0(-\epsilon\rho\omega)\right|^2 \left(P'(\rho)\right)^{d-1}\rho^{2(d-1)} d\theta d\theta' dr dw d\omega d\rho.
\end{aligned}
\end{equation*}
In the rest of this section, we will calculate the limit of this integral and conclude Proposition \ref{prop_limit_of_energy}. Since the integral in $\theta,\theta'$ is over a sphere centered at $\epsilon\omega$ with radius $\sqrt{t}$, we may write $\theta$, $\theta'$ in local coordinate
\begin{align*}
\theta &= h \epsilon\omega + y,\ h=\sqrt{t} - \sqrt{t-|y|^2},\\
\theta' &= h' \epsilon\omega + y',\ h'=\sqrt{t} - \sqrt{t-|y'|^2},
\end{align*}
where $h,h'\in\mathbb{R}$, $y,y'\in \omega^\perp:=\{z\in\mathbb{R}^d:z\cdot\omega=0\}$. The condition of support implies that $|h|,|h'|,|y|,|y'| \lesssim 1$. It is easy to check that, as $t\rightarrow+\infty$
\begin{align*}
&\sqrt{t}\omega\cdot(\theta-\theta') = \epsilon\sqrt{t}\left( \sqrt{t-|y'|^2} - \sqrt{t-|y|^2} \right) \rightarrow \epsilon\left( \frac{|y|^2}{2} - \frac{|y'|^2}{2} \right), \\
&\theta\cdot\omega= \epsilon \left( \sqrt{t} - \sqrt{t-|y|^2} \right) \rightarrow 0, \\
&\theta'\cdot\omega= \epsilon \left( \sqrt{t} - \sqrt{t-|y'|^2} \right) \rightarrow 0.
\end{align*}
Therefore, by Dominated Convergence Theorem, as $t$ tends to infinity, the limit of $E_+(\epsilon t)$ equals to
\begin{equation*}
\begin{aligned}
\frac{1}{(2\pi)^{2d}} \int_{y,y'\in\omega^\perp,r,w>0} &e^{\epsilon i \left[ P'(\rho)\rho\left(\frac{|y|^2}{2}-\frac{|y'|^2}{2}\right) - rw - \frac{1}{2}P''(\rho)w^2 \right]}  \\
&\times \chi\left( \frac{r\omega+\epsilon P'(\rho)y}{\lambda_1\rho^{\sigma_1}} \right) \chi\left( \frac{(r+P''(\rho)w)\omega+\epsilon P'(\rho)y'}{\lambda_1\rho^{\sigma_1}} \right) \\
&\times  \left|\hat{u}_0(-\epsilon\rho\omega)\right|^2 \left(P'(\rho)\right)^{d-1}\rho^{2(d-1)} dy dy' dr dw d\omega d\rho,
\end{aligned}
\end{equation*}
which, after a change of variable, is equal to
\begin{equation}\label{eq_lim_crit_limit_of_energy_sans_simplify}
\begin{aligned}
\frac{1}{(2\pi)^{2d}} \int_{y,y'\in\omega^\perp,r,w>0} &e^{\epsilon i \left[ \left(\frac{|y|^2}{2}-\frac{|y'|^2}{2}\right) - rw - \frac{1}{2}w^2 \right]} \chi\left( \frac{\sqrt{P''(\rho)}r\omega+\epsilon \sqrt{\rho^{-1}P'(\rho)} y}{\lambda_1\rho^{\sigma_1}} \right) \\
&\times  \chi\left( \frac{\sqrt{P''(\rho)}(r+w)\omega+\epsilon \sqrt{\rho^{-1}P'(\rho)} y'}{\lambda_1\rho^{\sigma_1}} \right) \\
&\times \left|\hat{u}_0(-\epsilon\rho\omega)\right|^2 \rho^{d-1} dy dy' dr dw d\omega d\rho.
\end{aligned}
\end{equation}
In order to give a compact form, we introduce the following functions,
\begin{equation*}
\begin{aligned}
&H(r,\omega) := \frac{1}{(2\pi)^\frac{d}{2}} \int_{y\cdot\omega=0} e^{\epsilon i \frac{r^2+|y|^2}{2}} \chi\left( \frac{\sqrt{P''(\rho)}r\omega+\epsilon \sqrt{\rho^{-1}P'(\rho)} y}{\lambda_1\rho^{\sigma_1}} \right) dy, \\
&F(r,\omega) := \int_{r}^\infty H(s,\omega) ds,
\end{aligned}
\end{equation*}
Remark that since $\chi\in \mathcal{S}(\mathbb{R}^d)$, $H$ decay rapidly at infinity, uniformly in $\omega$. Thus, $F$ is well-defined. With these functions, we can rewrite the integral \eqref{eq_lim_crit_limit_of_energy_sans_simplify} as 
\begin{equation*}
\begin{aligned}
&\frac{1}{(2\pi)^{d}} \int \int_{0}^{\infty} H(r,\omega) \int_{0}^{\infty}\overline{H(r+w,\omega)} dw dr \left|\hat{u}_0(-\epsilon\rho\omega)\right|^2 \rho^{d-1} d\omega d\rho \\
=& -\frac{1}{(2\pi)^{d}} \int \int_{0}^{\infty} \partial_r F(r,\omega) \overline{F(r,\omega)} dr \left|\hat{u}_0(-\epsilon\rho\omega)\right|^2 \rho^{d-1} d\omega d\rho.
\end{aligned}
\end{equation*}
As a consequence,
\begin{equation*}
\begin{aligned}
\lim_{t\rightarrow\infty} E_{\chi,\delta,\Lambda}(u_0,\epsilon t) =& \lim_{t\rightarrow\infty} 2\Real E_+(\epsilon t) \\
=& -\frac{2}{(2\pi)^{d}} \Real \int \int_{0}^{\infty} \partial_rF(r,\omega) \overline{F(r,\omega)} dr \left|\hat{u}_0(-\epsilon\rho\omega)\right|^2 \rho^{d-1} d\omega d\rho \\
=& -\frac{1}{(2\pi)^{d}} \int \int_{0}^{\infty} \frac{\partial}{\partial r}|F|^2(r,\omega) dr \left|\hat{u}_0(-\epsilon\rho\omega)\right|^2 \rho^{d-1} d\omega d\rho \\
=& \frac{1}{(2\pi)^{d}} \int |F(0,\omega)|^2 \left|\hat{u}_0(-\epsilon\rho\omega)\right|^2 \rho^{d-1} d\omega d\rho,
\end{aligned}
\end{equation*}
where 
$$ |F(0,\omega)|^2 = \frac{1}{(2\pi)^d} \left| \int_0^\infty \int_{y\cdot\omega=0} e^{\epsilon i \frac{r^2+|y|^2}{2}} \chi\left( \frac{\sqrt{P''(\rho)}r\omega+\epsilon \sqrt{\rho^{-1}P'(\rho)} y}{\lambda_1\rho^{\sigma_1}} \right) dydr \right|^2, $$
which is exactly $G^\text{alt}_\chi(\rho,\omega)$ defined in \eqref{eq_alt_definition_of_G_convex}, or \eqref{eq_definition_of_G_convex} with $\sigma_1=0$ and $\lambda_1=1$. The limits \eqref{eq_limit_of_energy_critical}, \eqref{eq_mod_limit_of_energy_critical}, and \eqref{eq_alt_limit_of_energy_critical} thus follow.

%% file: main/Examples.tex
\section{Study of Klein-Gordon equation}\label{sect_Klein-Gordon}

In this section, we shall prove Theorem \ref{thm_K-G_main} via a study of half-Klein-Gordon equation, i.e. \eqref{eq_frac_disper} with $P(\xi) = \langle\xi\rangle$. Let $w$ be the (real) solution to Klein-Gordon equation \eqref{eq_Klein-Gordon}. We have then
$$ 0 = (\partial_t^2 - \Delta + 1) w = -\left( \frac{\partial_t}{i} - P(D_x) \right)  \left( \frac{\partial_t}{i} + P(D_x) \right)w. $$
Thus, the complex-valued function
$$ u := \left( \frac{\partial_t}{i} + P(D_x) \right)w, $$
is the unique solution to half-Klein-Gordon equation with initial data
$$ u_0 := u|_{t=0} =\frac{w_1}{i} + P(D_x)w_0 \in L^2. $$
Due to the fact that $w$ is real-valued, we have the following relations
$$ \partial_t w = -\Imaginary u,\ \ w = \langle D_x \rangle^{-1} \Real u. $$

In this section, we denote the truncation $\Op{a^\text{KG}_\epsilon(t)}$, whose symbol is defined in \eqref{eq_K-G_def_truncated_symbol}, as $A(t)$, and define operators $A_{\pm}(t)$ as $\Op{a^\text{KG}_{\pm}(t)}$, where
\begin{equation}\label{eq_K-G_truncation_symbol_pm}
a^\text{KG}_{\pm}(t) = \chi\left( \frac{x \pm tP'(\xi)}{|t|^{\frac{1}{2}+\delta}} \right) \mathbbm{1}_{|x|>|tP'(\xi)|} \chi\left( \frac{\xi}{|t|^\epsilon} \right),
\end{equation}
$0<\delta<\frac{1}{2}$, $0<\epsilon<1$, and $\chi\in C_c^\infty(\mathbb{R}^d)$ are the same as in \eqref{eq_K-G_def_truncated_symbol}. Since we are interested in the behavior as $t \rightarrow +\infty$, it is harmless to assume $t \gg 1$. Before continuing the proof, we clarify that all the involved operators are bounded uniformly in $t \gg 1$.
\begin{lemma}\label{lem_K-G_bdd_of_truncation_op}
There exists time-independent constant $C>0$ and $t_0 \gg 1$, such that for all $t>t_0$,
$$ \|A(t)\|_{\mathcal{L}(L^2)},\ \|A_{\pm}(t)\|_{\mathcal{L}(L^2)} \leqslant C. $$
\end{lemma}
\begin{proof}
It is obvious that $P(\xi)=\langle\xi\rangle$ satisfies hypothesis \eqref{hyp_fractional-type_symbol} with $p_0=1$, $P_0=0$, and $p_1=-2$, $P_1=1$. In what follows, we shall focus on symbols $a^\text{KG}_{\pm}(t)$ (with $a^\text{KG}_\epsilon(t)$ treated in the same way) and decompose them in high and low frequency. Let $\tilde{\chi}\in C_c^\infty(\mathbb{R}^d)$ be a radial function which is equal to $1$ in the ball $B(0,1)$ and vanishes outside a larger ball $B(0,2)$. We may write
$$ a_{\pm,l}^\text{KG}(t):= a^\text{KG}_{\pm}(t)\tilde{\chi}(\xi),\ \ a_{\pm,h}^\text{KG}(t):= a^\text{KG}_{\pm}(t)(1-\tilde{\chi})(\xi). $$

For low frequency part $a_{\pm,l}^\text{KG}(t)$, it is easy to construct a symbol $P_l(\xi)$ such that $P_l$ verifies hypothesis \eqref{hyp_fractional-type_symbol} with $p_0=p_1=1$ and $P_0=P_1=0$, and that $P_l(\xi)=P(\xi)$ for all $|\xi| \leqslant 2$. In this way, we have
$$ a_{\pm,l}^\text{KG}(t,x,\xi) = \chi\left( \frac{x \pm tP'_l(\xi)}{|t|^{\frac{1}{2}+\delta}} \right) \mathbbm{1}_{|x|>|tP'(\xi)|} \tilde{\chi}(\xi), $$
where the factor involving $t^\epsilon$ disappears since the symbol is supported for $ |\xi| \leqslant 2 \ll t^\epsilon $, by choosing $t \gg 1$. Now, we may apply Proposition \ref{prop_boundedness_of_truncated_operator} to obtain the uniform-in-$t$ boundedess of operator with symbol
$$ \chi\left( \frac{x \pm tP'_l(\xi)}{|t|^{\frac{1}{2}+\delta}} \right) \mathbbm{1}_{|x|>|tP'(\xi)|} $$
and hence boundedness of the operator $\Op{a_{\pm,l}^\text{KG}(t)}$.

The boundedness of $\Op{a_{\pm,h}^\text{KG}(t)}$ is actually an immediate consequence of Proposition \ref{prop_boundedness_of_truncated_operator_with_truncated_frequency}. More precisely, by choosing arbitrary $\epsilon'$ such that $(\epsilon',\epsilon)$ satisfies conditions \eqref{eq_def_of_epsilons} associated to $P$ (i.e. with $p_0=1$ and $p_1=-2$), namely
$$ 0<\epsilon'\leqslant\frac{1}{p_0+1}=\frac{1}{2},\ \ 0<\epsilon\leqslant \frac{1}{-(p_1+1)}=1, $$
we are able to apply Proposition \ref{prop_boundedness_of_truncated_operator_with_truncated_frequency} to obtain the uniform boundedness of operator with symbol
$$ \chi\left( \frac{x \pm tP'(\xi)}{t^{\frac{1}{2}+\delta}} \right) \mathbbm{1}_{|x|>|tP'(\xi)|} (1-\chi)\left(\frac{\xi}{t^{-\epsilon'}}\right)  \chi\left( \frac{\xi}{t^\epsilon} \right). $$
If we add the high frequency truncation $(1-\tilde{\chi})(\xi)$, the symbol will be supported in $|\xi|\geqslant 1 \gg t^{-\epsilon'}$ since we have chosen $t \gg 1$. As a result, the truncation $(1-\chi)$ equals $1$ and the uniform boundedness of $\Op{a_{\pm,h}^\text{KG}(t)}$ follows.

We have shown that $A_{\pm}(t) = \Op{a_{\pm,l}^\text{KG}(t)} + \Op{a_{\pm,h}^\text{KG}(t)}$ is uniformly bounded on $L^2$. By repeating the same argument and replacing Proposition \ref{prop_boundedness_of_truncated_operator}, \ref{prop_boundedness_of_truncated_operator_with_truncated_frequency} by Theorem \ref{thm_boundedness_of_simple_truncated_operator}, we may also obtain the uniform boundedness of $A(t)$.
\end{proof}

Now we turn back to the proof of Theorem \ref{thm_K-G_main}. By definition, the truncated energy \eqref{eq_K-G_limit_of_energy} can be expressed as
$$ E^\text{KG}_\epsilon(\pm t) = \left\|A(t) \Imaginary u(\pm t)\right\|_{L^2}^2 + \left\|A(t) \frac{D_x}{\langle D_x\rangle} \Real u(\pm t)\right\|_{L^2}^2 + \left\|A(t) \frac{1}{\langle D_x\rangle} \Real u(\pm t)\right\|_{L^2}^2. $$
The three terms on the right hand side takes the form
\begin{equation}\label{eq_K-G_basic_form}
Q^\text{KG}_\pm(t,\epsilon_0,R) = \frac{1}{4}\left\|A(t) \left(Ru(\pm t) + \epsilon_0 \overline{Ru(\pm t)} \right) \right\|_{L^2}^2,
\end{equation}
where $\epsilon_0 \in \{+,-\}$ and $R$ is a bounded Fourier multiplier taking values among $1$, $D_x \langle D_x \rangle^{-1}$, and $\langle D_x \rangle^{-1}$. Due to the uniform boundedness of truncation operators $A(t)$, $A_{\pm}(t)$, \eqref{eq_K-G_basic_form} can be written as
\begin{align*}
&\frac{1}{4}\left\|A_\pm(t) Ru(\pm t) + \epsilon_0 A_\mp(t) \overline{Ru(\pm t)} \right\|_{L^2}^2 + O\left( \left\|(A(t)-A_\pm(t)) Ru(\pm t) \right\|_{L^2} \right) \\
&+ O\left( \left\|(A(t)-A_\mp(t)) \overline{Ru(\pm t)} \right\|_{L^2} \right) \\
=& \frac{1}{4}\left\|A_\pm(t) Ru(\pm t) \right\|_{L^2}^2 + \frac{1}{4}\left\| A_\mp(t) \overline{Ru(\pm t)} \right\|_{L^2}^2 + \epsilon_0 \Real \left\langle A_\pm(t) Ru(\pm t), A_\mp(t) \overline{Ru(\pm t)} \right\rangle_{L^2}\\ 
&+O\left( \left\|(A(t)-A_\pm(t)) Ru(\pm t) \right\|_{L^2} \right) + O\left( \left\|(A(t)-A_\mp(t)) \overline{Ru(\pm t)} \right\|_{L^2} \right) \\
=& \frac{1}{2}\left\|A_\pm(t) Ru(\pm t) \right\|_{L^2}^2 + \epsilon_0 \Real \left\langle A_\pm(t) Ru(\pm t), A_\mp(t) \overline{Ru(\pm t)} \right\rangle_{L^2}\\ 
&+O\left( \left\|(A(t)-A_\pm(t)) Ru(\pm t) \right\|_{L^2} \right),
\end{align*}
where we use the fact that for all complex-valued function $f\in L^2$
$$ A(t)\overline{f} = \overline{A(t)f},\ \ A_\pm(t)\overline{f} = \overline{A_\mp(t)f}. $$

In order to conclude the desired limit \eqref{eq_K-G_limit_of_energy}, it suffices to prove
\begin{proposition}\label{prop_K-G_lim_of_componants}
Let $v_0,v_{1,0}$ be two functions in $L^2$. With $A(t)$, $A_\pm(t)$ as above, we have following limits:
\begin{align}
&\lim_{t\rightarrow +\infty} \left\|A_\pm(t) e^{\pm itP(D_x)}v_0 \right\|_{L^2}^2 = \frac{1}{4}\|v_0\|_{L^2}^2, \label{eq_K-G_lim_of_componants_p1} \\
&\lim_{t\rightarrow +\infty} \left\langle A_\pm(t) e^{\pm itP(D_x)}v_0, A_\mp(t) e^{\mp itP(D_x)}v_{0,1} \right\rangle_{L^2} = 0, \label{eq_K-G_lim_of_componants_p2} \\
&\lim_{t\rightarrow +\infty} \left\| (A(t)-A_{\pm}(t))e^{\pm itP(D_x)}v_0 \right\|_{L^2}^2 = 0. \label{eq_K-G_lim_of_componants_p3}
\end{align}
\end{proposition}

Once Proposition \ref{prop_K-G_lim_of_componants} is proved, we may apply the three limits with $v_0 = Ru_0$ and $v_{0,1} = \overline{Ru_0}$ to obtain that
$$ \lim_{t\rightarrow +\infty}Q^\text{KG}_\pm(t,\epsilon_0,R) = \frac{1}{8}\|Ru_0\|_{L^2}^2, $$
since, due to the definition of $u$, we have, for all $t\in\mathbb{R}$,
$$ Ru(t) = e^{itP(D_x)}Ru_0,\ \ \overline{Ru(t)} = e^{-itP(D_x)}\overline{Ru_0}. $$
As a result, the limit \eqref{eq_K-G_limit_of_energy} follows from
\begin{align*}
&\lim_{t\rightarrow\pm\infty} E^{KG}_{\epsilon}(w_0,w_1,t) = \lim_{t\rightarrow+\infty} E^{KG}_{\epsilon}(w_0,w_1,\pm t) \\ 
=& \lim_{t\rightarrow+\infty} Q^\text{KG}_\pm(t,-1,1) + Q^\text{KG}_\pm(t,-1,D_x\langle D_x \rangle^{-1}) + Q^\text{KG}_\pm(t,1,\langle D_x \rangle^{-1}) \\
=& \frac{1}{8} \left( \left\| u_0 \right\|_{L^2}^2 + \left\| \frac{D_x}{\langle D_x \rangle} u_0 \right\|_{L^2}^2 + \left\| \frac{1}{\langle D_x \rangle} u_0 \right\|_{L^2}^2 \right) = \frac{1}{4}\|u_0\|_{L^2}^2 = \frac{1}{4} \left( \|w_0\|_{H^1}^2 + \|w_1\|_{L^2}^2 \right).
\end{align*}

In the rest of this section, we shall prove the limits \eqref{eq_K-G_lim_of_componants_p1}, \eqref{eq_K-G_lim_of_componants_p2}, and \eqref{eq_K-G_lim_of_componants_p3}. We have seen that operators $A(t)$, $A_\pm(t)$ are bounded uniformly in $t>t_0 \gg 1$. As a result, it suffices to calculate these limits for those $v_0,v_{0,1}$ belonging to some dense subspace of $L^2$. In what follows, we assume $\hat{v}_0, \hat{v}_{0,1}$ are smooth and supported in a annulus centered at zero.

The first limit \eqref{eq_K-G_lim_of_componants_p1} is no more than a consequence of \eqref{eq_mod_limit_of_energy_supercritical} with $P(\xi)=\langle\xi\rangle$, $\epsilon_1 = \epsilon$, $\chi_h=\chi$, and any $\epsilon_0 \in ]0,\frac{1}{2}]$. The exceptional truncation $\chi_l$ is not a problem, since it disappears when $|t|$ is large enough. For the remaining results \eqref{eq_K-G_lim_of_componants_p2} and \eqref{eq_K-G_lim_of_componants_p3}, we will apply a similar argument as in the proof of Proposition \ref{prop_limit_of_energy} for super-critical case.

\subsection{Limit of interaction term}

We first calculate the limit \eqref{eq_K-G_lim_of_componants_p2}. Clearly, via conjugation, it is enough to study the limit of
\begin{equation}\label{eq_K-G_interaction_term}
\left\langle A_+(t) e^{itP(D_x)}v_0, A_-(t) e^{- itP(D_x)}v_{0,1} \right\rangle_{L^2},
\end{equation}
since the other one can be recovered by relation
\begin{equation*}
\begin{aligned}
\overline{\left\langle A_-(t) e^{- itP(D_x)}v_0, A_+(t) e^{+ itP(D_x)}v_{0,1} \right\rangle_{L^2}} &= \left\langle \overline{A_-(t) e^{- itP(D_x)}v_0}, \overline{A_+(t) e^{itP(D_x)}v_{0,1}} \right\rangle_{L^2} \\
&= \left\langle A_+(t) e^{itP(D_x)}\overline{v_0}, A_-(t) e^{- itP(D_x)} \overline{v_{0,1}} \right\rangle_{L^2},
\end{aligned}
\end{equation*}
where the Fourier transform of $\overline{v_0}$, $\overline{v_{0,1}}$ still belongs to the class $C_c^\infty(\mathbb{R}^d\backslash\{0\})$.

By definition \eqref{eq_K-G_truncation_symbol_pm} of $A_{\pm}(t) = \Op{a^\text{KG}_{\pm}(t)}$, \eqref{eq_K-G_interaction_term} equals, in polar system, to
\begin{equation*}
\begin{aligned}
& \frac{1}{(2\pi)^{2d}} \int e^{i(r\rho\omega\cdot\theta-r\rho'\omega\cdot\theta')} e^{it (P(\rho)+P(\rho'))} \chi\left(\frac{r\omega+tP'(\rho)\theta}{t^{\frac{1}{2}+\delta}}\right) \chi\left(\frac{r\omega- tP'(\rho')\theta'}{t^{\frac{1}{2}+\delta}}\right) \\
&\hspace{4em}\times\mathbbm{1}_{\frac{r}{t}>P'(\rho),P'(\rho')} \hat{v}_0(\rho\theta) \overline{\hat{v}_{0,1}(\rho'\theta')} (r\rho\rho')^{d-1} d\theta d\theta' d\omega dr d\rho d\rho'.
\end{aligned}
\end{equation*}
Here we may omit the truncation in $\xi=\rho\theta$ and $\xi'=\rho'\theta'$ by taking $t \gg 1$. As in previous section, we focus on the integral in $\theta$ and $\theta'$, which are equal to, by Lemma \ref{lem_stat_int_on_angular_limit_part},
\begin{equation*}
\begin{aligned}
&\int e^{ir\rho\omega\cdot\theta} \chi\left(\frac{r\omega+tP'(\rho)\theta}{t^{\frac{1}{2}+\delta}}\right) \overline{\hat{v}_{0}(\rho\theta)}d\theta \\
&\hspace{10em}= e^{-i r\rho} \mu^{d-1} S_{-\frac{d-1}{2}}^{+}(\omega,\mu,\rho,\frac{r}{t}-P'(\rho),t; r\rho\mu^2) \kappa\left(\frac{r}{t}\right), \\
&\int e^{-ir\rho'\omega\cdot\theta'} \chi\left(\frac{r\omega-tP'(\rho')\theta'}{t^{\frac{1}{2}+\delta}}\right) \overline{\hat{v}_{0,1}(\rho'\theta')}d\theta' \\
&\hspace{10em}= e^{-i r\rho'} \mu^{d-1} S_{-\frac{d-1}{2}}^{-}(\omega,\mu,\rho',\frac{r}{t}-P'(\rho'),t; r\rho'\mu^2) \kappa\left(\frac{r}{t}\right),
\end{aligned}
\end{equation*}
respectively, where $\mu=t^{\delta-\frac{1}{2}}$ and $S_{m}^{\pm}(\omega,\mu,\rho,r',t; \zeta)$ is supported for $\zeta>c>0$, $\rho\sim 1$ and $|r'| \lesssim \mu$ and satisfies for all $\alpha\in\mathbb{N}^{d-1}$, $j,k,l,\gamma\in\mathbb{N}$,
$$ | \partial_{\omega}^{\alpha} \partial_{\mu}^j \partial_{\rho}^{k} \partial_{r'}^{l} \partial_{\zeta}^{\gamma} S_{m} |
\leqslant C \mu^{-(|\alpha|+j+l)} \langle\zeta\rangle^{m-\gamma}. $$
Here we add extra factor $\kappa\in C_c^\infty(]0,+\infty[)$, which equals to $1$ in a neighborhood of $1$, due to the support of integrand.

As a consequence, \eqref{eq_K-G_interaction_term} reads
\begin{equation*}
\begin{aligned}
& \frac{1}{(2\pi)^{2d}} \int e^{i\left[r(-\rho-\rho')-t (-P(\rho)-P(\rho'))\right]} S_{0}^{+}(\omega,\mu,\rho,\frac{r}{t}-P'(\rho),t; r\rho\mu^2) \\
&\hspace{4em}\times S_{0}^{-}(\omega,\mu,\rho',\frac{r}{t}-P'(\rho'),t; r\rho'\mu^2) \kappa^2\left(\frac{r}{t}\right)\mathbbm{1}_{\frac{r}{t}>P'(\rho),P'(\rho')} \\
&\hspace{4em}\times \hat{v}_0(\rho\theta) \overline{\hat{v}_{0,1}(\rho'\theta')} (\rho\rho')^{\frac{d-1}{2}} d\omega dr d\rho d\rho',
\end{aligned}
\end{equation*}
which can be rewritten as $I(t,-,-,-,-;F)$ defined in \eqref{eq_lim_def_of_standard_integral}, with
\begin{equation}\label{eq_K-G_interaction_def_of_F}
F(\rho,\rho',r,t;\zeta,\zeta') = S_{0}^{+}(\omega,t^{\delta-\frac{1}{2}},\rho,\frac{\zeta}{t},t; r\rho t^{2\delta-1}) S_{0}^{-}(\omega,t^{\delta-\frac{1}{2}},\rho',\frac{\zeta'}{t},t; r\rho' t^{2\delta-1})\kappa^2\left(\frac{r}{t}\right). 
\end{equation}
In \cite{delort2022microlocal}, the author has proved in Proposition 3.1.1 that 
\begin{proposition}\label{prop_K-G_interaction_limit_of_integral}
Let $F(\rho,\rho',r,t;\zeta,\zeta')$ be a smooth function on $\mathbb{R}_+^{4}\times\mathbb{R}^2$ and $\delta'\in ]\frac{1}{2},1[$. Assume that $F$ is supported for
$$ \rho,\rho'\sim 1,\ r \sim t,\ |\zeta|,|\zeta'|\lesssim t^{\delta'}, $$
and for all $j,j',k,\gamma,\gamma'\in\mathbb{N}$,
$$ | \partial_{\rho}^{j} \partial_{\rho'}^{j'} \partial_{r}^{k} \partial_{\zeta}^{\gamma} \partial_{\zeta'}^{\gamma'} F(\rho,\rho',r,t;\zeta,\zeta') |
\lesssim t^{-\delta'(k+\gamma+\gamma')}. $$
Under all the assumptions above, we have
$$ \lim_{t\rightarrow +\infty} I(t,\pm,\pm,\pm,\pm;F) = 0. $$
\end{proposition}

It is easy to check that the function $F$ defined in \eqref{eq_K-G_interaction_def_of_F} satisfies the conditions above with $\delta'=\delta+\frac{1}{2}$ and limit \eqref{eq_K-G_lim_of_componants_p2} follows.

\subsection{Limit of energy outside the truncation area}
 
It remains to prove \eqref{eq_K-G_lim_of_componants_p3}, which requires a study of the $L^2$-norm of $(A(t)-A_{\pm}(t))e^{\pm itP(D_x)}v_0$. As in previous part, by conjugation, it suffices to focus on
$$ (A(t)-A_{+}(t))e^{itP(D_x)}v_0(x) = \frac{1}{(2\pi)^d} \int e^{ix\xi} e^{itP(\xi)} (1-\chi)\left( \frac{x+tP'(\xi)}{t^{\frac{1}{2}+\delta}} \right) \mathbbm{1}_{|x|>|tP'(\xi)|} \hat{v}_0(\xi) d\xi, $$
where we omit again the truncation in $\xi$ by assuming $t \gg 1$. Actually, we have
$$ \overline{(A(t)-A_{-}(t))e^{-itP(D_x)}v_0} = (A(t)-A_{+}(t))e^{itP(D_x)}\overline{v_0}, $$
with $\hat{\overline{v}_0}$ belonging to the same subspace $C_c^\infty(\mathbb{R}^d\backslash\{0\})$.

We first check that the $L^2$-norm of $(A(t)-A_{+}(t))e^{itP(D_x)}v_0$ concentrates near $|x|=t$, i.e.
\begin{lemma}\label{lem_K-G_outside_truncation_outside_cone}
There exists a radial function $\kappa\in C_c^\infty(\mathbb{R}^d)$ supported in an annulus centered at zero, such that, when $t \rightarrow +\infty$,
$$ (A(t)-A_{+}(t))e^{itP(D_x)}v_0 = \kappa\left(\frac{x}{t}\right)(A(t)-A_{+}(t))e^{itP(D_x)}v_0 + O_{L^2}(t^{-N}),$$
for any $N\in\mathbb{N}$.
\end{lemma}
\begin{proof}
Let $\kappa_0\in C_c^\infty(\mathbb{R}^d)$ be a radial function supported near zero and equal to $1$ near zero. By definition, 
$$ (A(t)-A_{+}(t))e^{itP(D_x)}v_0(x) = \frac{1}{(2\pi)^d} \int e^{ix\xi} e^{itP(\xi)} (1-\chi)\left( \frac{x+tP'(\xi)}{t^{\frac{1}{2}+\delta}} \right) \mathbbm{1}_{|x|>|tP'(\xi)|} \hat{v}_0(\xi) d\xi. $$
Due to the support of integrand, namely $0<c<|\xi|<C$ and $|x|>tP'(|\xi|)$, the function above is supported for $|x|>tc_0$, where $0<c_0<P'(c)$. That is to say, when $\Supp\kappa_0$ is chosen small enough,
$$ \kappa_0\left(\frac{x}{t}\right)(A(t)-A_{+}(t))e^{itP(D_x)}v_0(x) = 0,\ \ \forall t,x. $$

Let $\kappa_1\in C_c^\infty(\mathbb{R}^d)$ be a radial function supported outside the unit ball. We further assume that $\kappa_1$ equals to $1$ away from zero. By observing that $|P'(\xi)|<1$ for all $\xi\in\mathbb{R}^d$, we have
$$ \kappa_1\left(\frac{x}{t}\right)(A(t)-A_{+}(t))e^{itP(D_x)}v_0(x) = \frac{\kappa_1\left(\frac{x}{t}\right)}{(2\pi)^d}  \int e^{i(x\xi+tP(\xi))} (1-\chi)\left( \frac{x+tP'(\xi)}{t^{\frac{1}{2}+\delta}} \right) \hat{v}_0(\xi) d\xi.$$
Remark that the non-smooth term $\mathbbm{1}_{|x|>t|P'(\xi)|}$ is identically $1$ as we add the cut-off $\kappa_1$. By integration by parts in $\xi$, we may rewrite the quantity above as
\begin{align*}
&\frac{\kappa_1\left(\frac{x}{t}\right)}{(2\pi)^d}  \int e^{i(x\xi+tP(\xi))} \frac{-\partial_\xi}{i} \cdot \left[ \frac{x+tP'(\xi)}{|x+tP'(\xi)|^2}(1-\chi)\left( \frac{x+tP'(\xi)}{t^{\frac{1}{2}+\delta}} \right) \hat{v}_0(\xi) \right] d\xi \\
=& \frac{\kappa_1\left(\frac{x}{t}\right)}{(2\pi)^d}  \int e^{i(x\xi+tP(\xi))} q_1(t,x,\xi) \hat{v}_0(\xi) d\xi + \frac{\kappa_1\left(\frac{x}{t}\right)}{(2\pi)^d}  \int e^{i(x\xi+tP(\xi))} q_0(t,x,\xi) \cdot \hat{v}_1(\xi) d\xi,
\end{align*}
where  $v_1(x) = xv_0(x)$, and
\begin{align*}
q_0(t,x,\xi) &= \frac{-\partial_\xi}{i} \left[ \frac{x+tP'(\xi)}{|x+tP'(\xi)|^2}(1-\chi)\left( \frac{x+tP'(\xi)}{t^{\frac{1}{2}+\delta}} \right) \right], \\
q_1(t,x,\xi) &= \frac{x+tP'(\xi)}{|x+tP'(\xi)|^2}(1-\chi)\left( \frac{x+tP'(\xi)}{t^{\frac{1}{2}+\delta}} \right)
\end{align*}
are smooth symbols satisfying for all $\alpha,\beta\in\mathbb{N}^{d}$,
$$ \left| \partial_{x}^\alpha \partial_\xi^\beta q_k(t,x,\xi) \right| \lesssim t^{-2\delta} t^{-(\frac{1}{2}+\delta)|\alpha|} t^{(\frac{1}{2}-\delta)|\beta|},\ \ k=0,1. $$
By Calderon-Vaillancourt Theorem (Lemma \ref{lem_tech_C-V}) and Lemma \ref{lem_tech_change_of_scaling}, the $\mathcal{L}(L^2)$-norm of operators $\Op{q_k}$'s is bounded by $t^{-2\delta}$, which implies that
$$ \left\| \kappa_1\left(\frac{x}{t}\right)(A(t)-A_{+}(t))e^{itP(D_x)}v_0 \right\|_{L^2} \lesssim t^{-2\delta}\left( \|v_0\|_{L^2} + \|v_1\|_{L^2} \right) \sim t^{-2\delta} \|\langle x \rangle v_0\|_{L^2}. $$
By repeating this procedure for $M$ times, we obtain
$$ \left\| \kappa_1\left(\frac{x}{t}\right)(A(t)-A_{+}(t))e^{itP(D_x)}v_0 \right\|_{L^2} \lesssim  t^{-2M\delta} \|\langle x \rangle^M v_0\|_{L^2} \lesssim t^{-2M\delta}. $$
The last estimate follows from the fact that $v_0$ is a Schwartz function. The proof is completed by choosing $\kappa = 1-\kappa_0-\kappa_1$.
\end{proof}

Thanks to Lemma \ref{lem_K-G_outside_truncation_outside_cone}, it remains to estimate the $L^2(dx)$-norm of 
$$ \kappa\left(\frac{x}{t}\right)(A(t)-A_{+}(t))e^{itP(D_x)}v_0(x), $$
or equivalently, in polar system, the $L^2(r^{d-1} drd\omega)$-norm of $I(t,r,\omega)$ defined by
\begin{equation}\label{eq_K-G_outside_truncation_near_cone_integral}
I(t,r,\omega) = \kappa\left(\frac{r}{t}\right) \int e^{ir\rho\omega\cdot\theta} e^{itP(\rho)} (1-\chi)\left( \frac{r\omega+tP'(\rho)\theta}{t^{\frac{1}{2}+\delta}} \right) \mathbbm{1}_{r>tP'(\rho)} \hat{v}_0(\rho\theta) \rho^{d-1} d\rho d\theta
\end{equation}

To begin with, we decompose \eqref{eq_K-G_outside_truncation_near_cone_integral} into three parts $I_+(t,r,\omega),I_-(t,r,\omega),I_0(t,r,\omega)$ by inserting
$$ \chi_0\left( \frac{\omega+\theta}{t^{\delta-\frac{1}{2}}} \right),\ \chi_0\left( \frac{\omega-\theta}{t^{\delta-\frac{1}{2}}} \right),\ 1-\chi_0\left( \frac{\omega+\theta}{t^{\delta-\frac{1}{2}}} \right)-\chi_0\left( \frac{\omega-\theta}{t^{\delta-\frac{1}{2}}} \right) $$
to the integral respectively, where $\chi_0\in C_c^\infty(\mathbb{R}^d)$ is radial, supported in a small ball centered at zero, and equal to $1$ near zero. The desired result \eqref{eq_K-G_lim_of_componants_p3} thus follows from the lemma below :

\begin{lemma}\label{lem_K-G_outside_truncation_near_cone}
For all $N\in\mathbb{N}$, there exists constants $C$, $C_N$ which is independent of $t,r,\omega$, such that
\begin{align}
&\| I_{+}(t,r,\omega) \|_{L^2(r^{d-1}drd\omega)}^2 \leqslant C_N t^{-N}, \label{eq_K-G_outside_truncation_near_cone_p1} \\
&\| I_{-}(t,r,\omega) \|_{L^2(r^{d-1}drd\omega)}^2 \leqslant C t^{-1}, \label{eq_K-G_outside_truncation_near_cone_p2} \\
&\| I_{0}(t,r,\omega) \|_{L^2(r^{d-1}drd\omega)}^2 \leqslant C_N t^{-N}. \label{eq_K-G_outside_truncation_near_cone_p3}
\end{align}
\end{lemma}

\begin{proof}[Proof of \eqref{eq_K-G_outside_truncation_near_cone_p1}]
The integral $I_+(t,r,\omega)$, by definition, reads
$$ \kappa\left(\frac{r}{t}\right) \int e^{i\left[r\rho\omega\cdot\theta+tP(\rho)\right]} (1-\chi)\left( \frac{r\omega+tP'(\rho)\theta}{t^{\frac{1}{2}+\delta}} \right) \chi_0\left( \frac{\omega+\theta}{t^{\delta-\frac{1}{2}}} \right) \mathbbm{1}_{r>tP'(\rho)} \hat{v}_0(\rho\theta) \rho^{d-1} d\rho d\theta. $$
The intergand of $I_+$ is supported for
$$ |r-tP'(\rho)| = |-r\theta+tP'(\rho)\theta| \geqslant |r\omega + tP'(\rho)\theta| - r|\theta+\omega| \geqslant ct^{\frac{1}{2}+\delta} - Cc_0 t \times t^{\delta-\frac{1}{2}} = (c-Cc_0)t^{\frac{1}{2}+\delta}, $$
where $0<c_0 \ll 1$ is the radius of $\Supp \chi_0$. This implies that
$$|r-tP'(\rho)| \geqslant c' t^{\frac{1}{2}+\delta}. $$
As a consequence, the integrand of $I_{+}$ is smooth. This allows us to apply integration by parts in $\rho$, since $\rho\sim 1$ and
$$ |r\omega\cdot\theta+tP'(\rho)| \geqslant |-r+tP'(\rho)| - r|\omega+\theta| \geqslant (c-2Cc_0)t^{\frac{1}{2}+\delta} \geqslant c'' t^{\frac{1}{2}+\delta}. $$
More precisely, by using
$$ e^{i\left[r\rho\omega\cdot\theta+tP(\rho)\right]} = \frac{-i}{r\omega\cdot\theta+tP'(\rho)} \partial_\rho  e^{i\left[r\rho\omega\cdot\theta+tP(\rho)\right]}, $$
one may gain $t^{-(\frac{1}{2}+\delta)}$ from $|r\omega\cdot\theta+tP'(\rho)|^{-1}$ and $t^{\frac{1}{2}-\delta}$ from each $\partial_\rho$. In conclusion, after $M$ times integration by parts in $\rho$ we have
$$ |I_+(t,r,\omega)| \lesssim t^{-2M\delta}. $$
Due to factor $\kappa$, $I_+$ is supported for $r\sim t$, which implies that
$$ \| I_{+}(t,r,\omega) \|_{L^2(r^{d-1}drd\omega)}^2 \lesssim t^{-4M\delta+d} \leqslant t^{-N}, $$
where $M$ is large enough so that $4M\delta-d \geqslant N$.
\end{proof}

\begin{proof}[Proof of \eqref{eq_K-G_outside_truncation_near_cone_p2}]

The integral $I_-$ will be treated as in Section \ref{sect_limit_of_energy}. We first observe that the integrand of $I_-(t,r,\omega)$, which reads
$$ \kappa\left(\frac{r}{t}\right) \int e^{i\left[r\rho\omega\cdot\theta+tP(\rho)\right]} (1-\chi)\left( \frac{r\omega+tP'(\rho)\theta}{t^{\frac{1}{2}+\delta}} \right) \chi_0\left( \frac{\omega-\theta}{t^{\delta-\frac{1}{2}}} \right) \mathbbm{1}_{r>tP'(\rho)} \hat{v}_0(\rho\theta) \rho^{d-1} d\rho d\theta, $$
is supported for, as $t \gg 1$,
$$ |r\omega+tP'(\rho)\theta| \geqslant |r\theta+tP'(\rho)\theta| - r|\omega-\theta| \geqslant r+tP'(\rho) - Cc_0 t^{\frac{1}{2}+\delta} \geqslant c't, $$
where $0<c_0 \ll 1$ is the radius of $\Supp\chi_0$. That is to say, the factor $(1-\chi)$ is identically $1$, and $I_-$ becomes
$$ \kappa\left(\frac{r}{t}\right) \int_{\rho_0}^{\min(\rho_1,(P')^{-1}(r/t))} e^{itP(\rho)} \int_{\mathbb{S}^d}  e^{ir\rho\omega\cdot\theta} \chi_0\left( \frac{\omega-\theta}{t^{\delta-\frac{1}{2}}} \right) \hat{v}_0(\rho\theta) d\theta  \rho^{d-1} d\rho,  $$
where $\hat{v}_0(\rho\theta)$ is supported for $\rho \in [\rho_0,\rho_1]$ and we use the convention $(P')^{-1}(s) = +\infty$ when $s\geqslant 1$. Now, we may apply Lemma \ref{lem_stat_stationary_phase} with $\lambda=r\rho$, $\mu = t^{\delta-\frac{1}{2}}$, and
$$ F(x,y,z,\mu;\rho) = \chi_0\left( \frac{y-z}{\mu} \right) \tilde{\chi}_0\left( \frac{x-y}{\mu} \right) \hat{v}_0(\rho y). $$
Here $\tilde{\chi}_0\in C_c^\infty(\mathbb{R}^d)$ is chosen to be equal to $1$ on the support of $\chi_0$. Note that $F$ has to be taken at $(x,y,z)=(\omega,\theta,\omega)$ in order to recover the above integral. This corresponds in the statement of Lemma \ref{lem_stat_stationary_phase} to $x=\theta$, $y=\theta'$, and $z=\omega$ (and no variable $\omega'$). The extra term $\rho$ is an extra parameter staying in a compact subset. Due to our choice of $\mu$, it is clear that $r\sim t$ and $\rho\sim 1$ imply $\lambda\mu^2 \geqslant c t^{2\delta} \gg 1$. As a result, the integral in $\theta$ equals to
$$ \int_{\mathbb{S}^d}  e^{ir\rho\omega\cdot\theta} \chi_0\left( \frac{\omega-\theta}{t^{\delta-\frac{1}{2}}} \right) \hat{v}_0(\rho\theta) d\theta = (2\pi)^{\frac{d-1}{2}} e^{ir\rho} \mu^{d-1} S_{-\frac{d-1}{2}}(\omega,\rho,\mu;r\rho\mu^2), $$
with $S_{m}(\omega,\rho,\mu;\zeta)$ smooth, supported for $\rho\sim 1$, $\zeta > 1$, and satisfying for all $\alpha,\in\mathbb{N}$, $j,k,\gamma\in\mathbb{N}$
$$ |\partial_{\omega}^{\alpha} \partial_{\rho}^{j} \partial_{\mu}^{k}\partial_{\zeta}^{\gamma} S^{\pm}_{m}(\omega,\rho,\mu;\zeta)| 
\leqslant C \mu^{-|\alpha|-k} \langle\zeta\rangle^{m-\gamma}. $$
This formulation allows us to rewrite $I_- r^{\frac{d-1}{2}}$, up to multiplication with constants, as
$$ I_-(t,r,\omega)r^{\frac{d-1}{2}} = \kappa\left(\frac{r}{t}\right) \int_{\rho_0}^{\min(\rho_1,(P')^{-1}(r/t))} e^{i\left[r\rho+tP(\rho)\right]} S_{0}(\omega,\rho,\mu;r\rho\mu^2) \rho^\frac{d-1}{2} d\rho. $$
Now, we may apply integration by parts in $\rho$. Due to the fact that $r+tP'(\rho) \sim t$ and $\min(\rho_1,(P')^{-1}(r/t)) \sim 1$, the boundary terms and remaining term are all bounded by $t^{-1}$, namely
$$ \left| I_-(t,r,\omega)r^{\frac{d-1}{2}} \right| \lesssim \mathbbm{1}_{r\sim t} t^{-1}, $$
which implies \eqref{eq_K-G_outside_truncation_near_cone_p2}.
\end{proof}

\begin{proof}[Proof of \eqref{eq_K-G_outside_truncation_near_cone_p3}]
Unlike the study of $I_{\pm}$, $I_0$ will be estimated via integration by parts in $\theta$ variable. It is clear that the integrand of
\begin{equation*}
\begin{aligned}
I_0(t,r,\omega) =& \kappa\left(\frac{r}{t}\right) \int e^{ir\rho\omega\cdot\theta} e^{itP(\rho)} (1-\chi)\left( \frac{r\omega+tP'(\rho)\theta}{t^{\frac{1}{2}+\delta}} \right)\mathbbm{1}_{r>tP'(\rho)} \hat{v}_0(\rho\theta) \rho^{d-1} \\
&\hspace{10em}\times\left[ 1-\chi_0\left( \frac{\omega+\theta}{t^{\delta-\frac{1}{2}}} \right)-\chi_0\left( \frac{\omega-\theta}{t^{\delta-\frac{1}{2}}} \right) \right]  d\rho d\theta
\end{aligned}
\end{equation*}
is supported away from $\{\theta\pm\omega=0\}$, which allows us to rewrite the integral above in local coordinate
$$ \theta = h\omega + \sqrt{1-h^2}y,\ \ h\in [-1,1],y\in\omega^\perp\cap\mathbb{S}^{d-1}, $$
where $\omega^\perp$ is defined as the hyperplane $\{ y\in\mathbb{R}^d : y\cdot\omega=0 \}$. In this way, we may rewrite $I_0$ as
\begin{equation}\label{eq_K-G_outside_truncation_near_cone_p3_int_loc_coord}
\begin{aligned}
&\kappa\left(\frac{r}{t}\right) \int e^{ir\rho h} F_0(t,r,\rho,\omega,h\omega + \sqrt{1-h^2}y) (1-h^2)^{\frac{d}{2}-1} \\
&\hspace{4em}\times\left[ 1-\chi_0\left( \frac{(h+1)\omega + \sqrt{1-h^2}y}{t^{\delta-\frac{1}{2}}} \right)-\chi_0\left( \frac{(h-1)\omega + \sqrt{1-h^2}y}{t^{\delta-\frac{1}{2}}} \right) \right] dh dy  d\rho,
\end{aligned}
\end{equation}
where 
$$ F_0(t,r,\rho,\omega,\theta) = e^{itP(\rho)} (1-\chi)\left( \frac{r\omega+tP'(\rho)\theta}{t^{\frac{1}{2}+\delta}} \right)\mathbbm{1}_{r>tP'(\rho)} \hat{v}_0(\rho\theta) \rho^{d-1}. $$
Remark that due to the cut-off away from $\pm\omega$, the integrand of \eqref{eq_K-G_outside_truncation_near_cone_p3_int_loc_coord} is supported for
$$ | (h\pm 1)\omega + \sqrt{1-h^2}y |^2 \geqslant ct^{2\delta-1}. $$
By developing the inequality above we obtain
$$ \sqrt{1-h^2} \geqslant c't^{\delta-\frac{1}{2}}.$$
Thus, $F_0$ is supported for $\rho \sim 1$ and satisfies for all $k\in\mathbb{N}$,
$$ \left| \partial_h^k \left( F_0(t,r,\rho,\omega,h\omega + \sqrt{1-h^2}y) \right) \right| \lesssim t^{(1-2\delta)k}. $$
The same estimates hold for $(1-h^2)^{\frac{d}{2}-1}$ and cut-off $\chi_0$'s in dimension $d \geqslant 2$, while in trivial case $d=1$, $I_0$ is identically zero.

Now, we may apply $M$ times integration by parts in $h$ for \eqref{eq_K-G_outside_truncation_near_cone_p3_int_loc_coord}. As $r\sim t$, $\rho\sim 1$, each $\partial_h$ in the amplitude gives $t^{1-2\delta}$, we have
$$ \left| I_0(t,r,\omega) \right| \lesssim \mathbbm{1}_{r\sim t} t^{-2M\delta}, $$
and \eqref{eq_K-G_outside_truncation_near_cone_p3} follows from
$$ \left\| I_0(t,r,\omega) \right\|_{L^2(r^{d-1}drd\omega)}^2 \lesssim t^{-4M\delta+d} \lesssim t^{-N}, $$
by choosing $4M\delta-d \geqslant N$.
\end{proof}

%% file: appendix/Appendix_technical_lemma.tex
\section{Technical lemmas}\label{sect_lemmas}

This appendix is a collection of technical lemmas which are used in previous sections.

\subsection{Some technical inequalities}

\begin{lemma}\label{lem_tech_int_on_sphere}
For any real number $m>d-1$, there exists constant $C = C(m,d)>0$, such that
$$ \sup_{\xi\in\mathbb{R}^{d}} \int_{\mathbb{S}^{d-1}} \langle R\omega-\xi \rangle ^{-m}d\omega \leqslant C R^{-(d-1)}.$$
\end{lemma}
\begin{proof}
By change of variable, it is equivalent to study the boundedness of 
$$ \int_{R\mathbb{S}^{d-1}} K(\omega-\xi) d\omega, $$
where $K(x) = (1+|x|^{m})^{-1}$. This integral can be regarded as a convolution of $K$ and the Borel measure $\mu_{R}$, which is defined by
$$ \forall \phi\in C_{c}(\mathbb{R}^d),\ \ \langle \mu_R, \phi \rangle:= \int_{R\mathbb{S}^{d-1}} \phi(\omega) d\omega. $$

We introduce the $\alpha$-dimensional density of a Borel measure $\mu$:
$$M^{(\alpha)}_{\mu}(x):= \sup_{r>0} \frac{\mu(B(x,r))}{r^{\alpha}}.$$
Remark that $M^{(d-1)}_{\mu_R}(\xi) = M^{(d-1)}_{\mu_1}(\xi/R)$ and that $M^{(d-1)}_{\mu_1}$ is a bounded function. Thus, $M^{(d-1)}_{\mu_R}(\xi)$ is bounded uniformly in $R>0$ and $\xi\in\mathbb{R}^d$.

Now, it suffices to show that there exists some constant $c = c(m,d)$, such that for any Borel measure $\mu$ on $\mathbb{R}^d$,
$$K*\mu(\xi) \leqslant c M^{(d-1)}_{\mu}(\xi).$$
By applying a translation, the problem can be reduced to the case $\xi=0$.
\begin{equation*}
\begin{aligned}
K*\mu(0) \
&= \int_{\mathbb{R}^d} K(-y) d\mu(y) \\
&= \int_{\mathbb{R}^d} K(y) d\mu(y) \\
&= \lim_{n\rightarrow\infty} \frac{1}{n} \sum_{j=1}^{n} \mu(B(0,r^{(n)}_j)), \\
\end{aligned}
\end{equation*}
where that last equality follows from dominated convergence theorem, with $r^{(n)}_j$ defined by $K^{-1}(1-\frac{j}{n}) = r^{(n)}_j \mathbb{S}^{d-1}$. We may calculate $r^{(n)}_j$ explicitly $r^{(n)}_j = \left(\frac{j/n}{1-j/n}\right)^{\frac{1}{m}}$. Therefore, by definition of $(d-1)$-dimensional density, we have
\begin{equation*}
\begin{aligned}
K*\mu(0) \
&\leqslant M^{(d-1)}_{\mu}(0) \lim_{n\rightarrow\infty} \frac{1}{n} \sum_{j=1}^{n} \left(r^{(n)}_j\right)^{d-1} \\
&= M^{(d-1)}_{\mu}(0) \lim_{n\rightarrow\infty} \frac{1}{n} \sum_{j=1}^{n} \left(\frac{j/n}{1-j/n}\right)^{\frac{d-1}{m}} \\
&= M^{(d-1)}_{\mu}(0) \int_{0}^{1} \left(\frac{x}{1-x}\right)^{\frac{d-1}{m}} dx.
\end{aligned}
\end{equation*}
The last quantity is finite since $m>d-1$.
\end{proof}

\begin{lemma}\label{lem_tech_fourier_symbol}
Let $S_{-1} \in C^{\infty}(\mathbb{R})$ with $|S_{-1}^{(\alpha)}(\xi)|\leqslant C_\alpha \langle\xi\rangle^{-1-\alpha} $, for any $\alpha\in\mathbb{N}$. Then, for any $N\in\mathbb{N}$, there exists some constant $C$, such that
\begin{equation}\label{eq_tech_lem_fourier_symbol}
\left| \int_{\mathbb{R}} e^{i\lambda\xi} S_{-1}(\xi) d\xi \right| \leqslant C \langle\lambda\rangle^{-N}( 1 + \log_{-}(|\lambda|) ),\ \ \forall\lambda\neq 0,
\end{equation}
where the integral on the left hand side should be understood as oscillatory integral and the function $\log_{-}$ is defined by:
\begin{equation*}
\log_{-}(t):= \left\{
\begin{array}{cl}
|\log(t)| & \text{if\ }t \in ]0,1[ \\
0 & \text{otherwise}
\end{array}\right.
\end{equation*}
\end{lemma}
\begin{proof}
For simplicity, we may assume $\lambda>0$ and that $S_{-1}$ is supported in $[c_0,\infty[$ for some $c_0>0$.

When $\lambda \geqslant c$ for some small constant $c>0$, we may apply integration by part on $\xi$, then the left hand side of \eqref{eq_tech_lem_fourier_symbol} can be controlled by
$$ \frac{C}{\lambda} \int_{\mathbb{R}} \langle\xi\rangle^{-2} d\xi \lesssim \langle\lambda\rangle^{-1}.$$
To obtain arbitrary polynomial decrease for large $\lambda$, we only need to apply integration by parts several times.

When $\lambda<c$, we introduce a cut-off $\chi\in C_{c}^{\infty}(\mathbb{R})$, such that $\chi=1$ in a large neighbourhood of $0$. For the part where $|\lambda\xi|$ is small:
$$\left| \int_{\mathbb{R}} e^{i\lambda\xi} S_{-1}(\xi) \chi(\lambda\xi) d\xi \right|
\leqslant C\int_{c_1}^{\frac{c_2}{\lambda}} \frac{d\xi}{\xi} \leqslant C'(1+ \log_{-}(\lambda)).$$
For the remaining part where $|\lambda\xi|$ is large, we need to estimate
$$\int_{\mathbb{R}} e^{i\lambda\xi} S_{-1}(\xi) (1-\chi)(\lambda\xi) d\xi.$$
By applying integration by parts in $\xi$, we will obtain two integrals. The first one is of the form
$$ \int_{\mathbb{R}} e^{i\lambda\xi} S_{-1}(\xi) \chi'(\lambda\xi) d\xi, $$
which can be treated as above. The second one take the form
$$ \frac{1}{\lambda}\int_{\mathbb{R}} e^{i\lambda\xi} S_{-2}(\xi) (1-\chi)(\lambda\xi) d\xi, $$
which is bounded by
$$ \frac{1}{\lambda} \int_{|\xi|>\frac{c'}{\lambda}} \langle\xi\rangle^{-2} d\xi \leqslant C. $$
\end{proof}

\begin{lemma}[Cotlar-Stein Lemma]\label{lem_tech_C-S}
Let $H$ be a Hilbert space and $\{T_j\}_{j\in\mathbb{N}}$ be a series of bounded operators on $H$. If
\begin{align}
&A = \sup_{j\in\mathbb{N}} \sum_{k\in\mathbb{N}} \|T_j T_k^*\|_{\mathcal{L}(H)}^\frac{1}{2} < +\infty \label{eq_tech_C-S_def_of_A}, \\
&B = \sup_{j\in\mathbb{N}} \sum_{k\in\mathbb{N}} \|T_j^* T_k\|_{\mathcal{L}(H)}^\frac{1}{2} < +\infty \label{eq_tech_C-S_def_of_B},
\end{align}
the operator $T = \sum_{j\in\mathbb{N}}T_j$ is well-defined via point-wise limit
$$ \forall u\in H,\ \ Tu := \lim_{J\rightarrow+\infty} \sum_{j=0}^{J} T_j u\ \ \text{in }H.$$
Moreover, $T$ is bounded on $H$ with estimate
\begin{equation}\label{eq_tech_C-S_bdd}
\|T\|_{\mathcal{L}(H)} \leqslant \sqrt{AB}.
\end{equation}
\end{lemma}
The inequality \eqref{eq_tech_C-S_bdd} was first given in \cite{cotlar1956inequality} for finite sum with $\|T_j T_k^*\|_{\mathcal{L}(H)}$ and $\|T_j^* T_k\|_{\mathcal{L}(H)}$ decreasing exponentially in $|j-k|$. The generalized version stated above and its proof can be found, for example, in Theorem 1, Chapter \MakeUppercase{\romannumeral7} of \cite{stein1993harmonic}.

\subsection{Criteria on \texorpdfstring{$L^2$}{L2}-boundedness of pseudo-differential operators}

\begin{lemma}\label{lem_tech_change_of_scaling}
Let $a$ be a symbol on $\mathbb{R}^d$ and $\lambda>0$. The rescaled symbol
$$ a_{\lambda}(x,\xi) := a(\lambda x, \frac{\xi}{\lambda}) $$
satisfies
$$ \| \operatorname{Op}(a_\lambda) \|_{\mathcal{L}(L^2)} = \| \operatorname{Op}(a) \|_{\mathcal{L}(L^2)}. $$
\end{lemma}

\hspace{\fill}

\begin{lemma}\label{lem_tech_change_of_x_and_xi}
Let $a$ be a symbol on $\mathbb{R}^d$ and symbol $\tilde{a}$ is defined as
$$ \tilde{a}(x,\xi) := a(\xi, x). $$
Then we have
$$ \|\Op{a}\|_{\mathcal{L}(L^2)} = \|\Op{\tilde{a}}\|_{\mathcal{L}(L^2)}. $$
\end{lemma}
\hspace{\fill}

\begin{lemma}\label{lem_tech_Young's_inequality}
Let $K$ be a kernel function of operator $\mathcal{K}$ defined as
$$ \mathcal{K}u(x) := \int K(x,x-y)u(y)dy. $$
If there exists $K_0\in L^1(\mathbb{R}^d)$ such that
$$ |K(x,z)| \leqslant K_0(z),\ \forall x,z\in\mathbb{R}^d, $$
we have
$$ \|\mathcal{K}\|_{\mathcal{L}(L^2)} \leqslant \|K_0\|_{L^1}. $$
\end{lemma}
\hspace{\fill}

\begin{lemma}\label{lem_tech_estimate_as_kernel}
Let $a$ and $b$ be symbols satisfying
$$ |a(x,\xi)| \leqslant b(x,\xi), $$
and $B\in\mathcal{L}(L^2)$ be an operator defined by
$$ Bu(x) = \int b(x,\xi)v(\xi) d\xi. $$
Then 
$$ \|\operatorname{Op}(a)\|_{\mathcal{L}(L^2)} \leqslant \|B\|_{\mathcal{L}(L^2)}. $$
\end{lemma}
\hspace{\fill}

\begin{lemma}\label{lem_tech_bound_operator_int}
Let $m:\mathbb{S}^{d-1}\times\mathbb{S}^{d-1}\times]0,\infty[\rightarrow\mathbb{C}$ be a smooth function satisfying for all $\alpha,\alpha'\in\mathbb{N}^{d-1}$ and $N\in\mathbb{N}$
$$ |\partial_{\omega}^{\alpha}\partial_{\omega'}^{\alpha'}m(\omega,\omega',\mu)|
\leqslant C_{\alpha,\alpha',N} A \mu^{-|\alpha|-|\alpha'|} \langle\frac{d(\omega,-\omega')}{\mu}\rangle^{-N}, $$
where $A$ is a quantity independent of $\omega$ and $\omega'$, $d$ is the distance on the sphere $\mathbb{S}^{d-1}$. For all $\lambda>0$, the operator $T_\lambda$ is defined by
$$ T_{\lambda}u(\omega') = \int_{\mathbb{S}^{d-1}} e^{i\lambda\omega\omega'} m(\omega,\omega',\mu) u(\omega) d\omega. $$

Then there exists a constant $C>0$, such that
$$ \| T_\lambda \|_{\mathcal{L}(L^2(\mathbb{S}^{d-1}))} \leqslant CA \lambda^{-\frac{d-1}{2}}. $$
\end{lemma}

The case of $\mu\in[1,\infty[$ is classical, the proof of which can be found in, for example, \cite{sogge2017fourier}, Theorem 2.1.1. As for the case $\mu\in]0,1]$, one may refer to \cite{delort2022microlocal}. A direct consequence is stated as following:

\begin{lemma}\label{lem_tech_boundedness_of_symbol_polar}
Let $a(x,\xi,\mu)$ be a smooth symbol on $\mathbb{R}^d\times\mathbb{R}^d$ depending on a parameter $\mu\in ]0,\infty[$. Let $b$ be a function  on $\mathbb{R}_+^2$, such that the associated operator $B$ defined below is bounded on $L^2(\mathbb{R}_+)$, namely
\begin{equation}\label{eq_tech_boundedness_of_symbol_polar_assumption_kernel}
\|B\|_{\mathcal{L}(L^2(\mathbb{R}_+))} < \infty,\ Bf(r) = \int_0^\infty b(r,\rho)f(\rho) d\rho.
\end{equation}

If, in polar system $x=r\omega$, $\xi=\rho\theta$, $a(x,\xi,\mu)$ satisfies for all $\alpha,\beta\in\mathbb{N}^{d-1}$ and $N\in\mathbb{N}$,
\begin{equation}\label{eq_tech_boundedness_of_symbol_polar_assumption_angular}
|\partial_{\omega}^\alpha \partial_{\theta}^\beta a(r\omega,\rho\theta,\mu)| \leqslant C_{\alpha,\beta,N} b(r,\rho) \mu^{-|\alpha|-|\beta|} \langle\frac{d(\omega,-\theta)}{\mu}\rangle^{-N},
\end{equation}
then we have
\begin{equation}\label{eq_tech_boundedness_of_symbol_polar_conclusion}
\|\operatorname{Op}(a)\|_{\mathcal{L}(L^2(\mathbb{R}^d))} \leqslant C \|B\|_{\mathcal{L}(L^2(\mathbb{R}_+))},
\end{equation}
where $C>0$ is a universal constant.
\end{lemma}
\begin{proof}
By definition of $\operatorname{Op}(a)$, we have
$$ \operatorname{Op}(a)u(r\omega) = \frac{1}{(2\pi)^d} \int_0^\infty\int_{\mathbb{S}^{d-1}} e^{ir\rho \omega\theta} a(r\omega,\rho\theta) \hat{u}(\rho\theta) \rho^{d-1} d\theta d\rho. $$

It is easy to check that $\|f(\rho,\theta)\|_{L^2(d\rho d\theta)} = (2\pi)^{\frac{d}{2}}\|u\|_{L^2(\mathbb{R}^d)}$, with $f(\rho,\theta) = \hat{u}(\rho\theta)\rho^{\frac{d-1}{2}}$. Thus,
\begin{equation*}
\begin{aligned}
&\|\operatorname{Op}(a)u\|_{L^2(\mathbb{R}^d)} = \|\operatorname{Op}(a)u(r\omega)r^{\frac{d-1}{2}}\|_{L^2(drd\omega)} \\
&\hspace{4em}= \left\|\frac{1}{(2\pi)^d} \int_0^\infty\int_{\mathbb{S}^{d-1}} e^{ir\rho \omega\theta} a(r\omega,\rho\theta) f(\rho,\theta) (r\rho)^{\frac{d-1}{2}} d\theta d\rho\right\|_{L^2(drd\omega)} \\
&\hspace{4em}\leqslant C_0 \left\| \int_0^\infty (r\rho)^{\frac{d-1}{2}} \|\int_{\mathbb{S}^{d-1}} e^{ir\rho \omega\theta} a(r\omega,\rho\theta) f(\rho,\theta)  d\theta \|_{L^2(d\omega)} d\rho\right\|_{L^2(dr)}.
\end{aligned}
\end{equation*}
By applying previous lemma with $m(\omega,\theta,\mu) = a(r\omega,\rho\theta,\mu)$, where $r,\rho$ should be regarded as parameters, $A = b(r,\rho)$, and $\lambda = r\rho$, we obtain
$$ \|\int_{\mathbb{S}^{d-1}} e^{ir\rho \omega\theta} a(r\omega,\rho\theta) f(\rho,\theta)  d\theta \|_{L^2(d\omega)} \leqslant C_1 (r\rho)^{-\frac{d-1}{2}} b(r,\rho) \|f(\theta,\rho)\|_{L^2(d\theta)}. $$
Thus, 
\begin{equation*}
\begin{aligned}
\|\operatorname{Op}(a)u\|_{L^2(\mathbb{R}^d)} &\leqslant C_0 C_1 \left\| \int_0^\infty b(r,\rho) \|f(\theta,\rho)\|_{L^2(d\theta)}  d\rho\right\|_{L^2(dr)} \\
&\leqslant C_0 C_1 \|B\|_{\mathcal{L}(L^2(\mathbb{R}_+))} \|f(\theta,\rho)\|_{L^2(d\theta d\rho)} = C\|B\|_{\mathcal{L}(L^2(\mathbb{R}_+))} \|u\|_{L^2(\mathbb{R}^d)}.
\end{aligned}
\end{equation*}
\end{proof}

\begin{lemma}\label{lem_tech_boundedness_of_symbol_study_in_xi}
Let $a$ be a symbol on $\mathbb{R}^d$, depending on some parameter $\lambda\in ]0,\infty[$. If for all $\alpha\in\mathbb{N}^d$,
$$ \sup_{x\in\mathbb{R}^d} \| \partial^\alpha_\xi a(x,\cdot) \|_{L^1_{\xi}} \leqslant C_\alpha \lambda^{d-|\alpha|} ,$$
the operator $\Op{a}$ is bounded on $L^2$, uniformly in $\lambda$.
\end{lemma}

\begin{proof}
Due to Lemma \ref{lem_tech_change_of_scaling}, the problem can be reduced to the case $\lambda=1$. The kernel of operator $\Op{a}$ is $K(x,y)=J(x,x-y)$, where
$$ J(x,z) = \frac{1}{(2\pi)^d} \int e^{iz\cdot\xi} a(x,\xi) d\xi. $$
We first observe that, for all $x\in\mathbb{R}^d$, $ |J(x,z)| \leqslant (2\pi)^{-d} \| a(x,\cdot) \|_{L^1_{\xi}} $, which is uniformly bounded. Then, by integration by part, we have, for all $N\in\mathbb{N}$,
\begin{equation*}
\begin{aligned}
\left| |z|^{2N}J(x,z) \right| &\lesssim \sum_{|\alpha|=2N} \left| \int e^{iz\cdot\xi} \partial_{\xi}^\alpha a(x,\xi) d\xi \right| \\
&\lesssim \sum_{|\alpha|=2N} \sup_{x\in\mathbb{R}^d} \| \partial^\alpha_\xi a(x,\cdot) \|_{L^1_{\xi}} \leqslant C_\alpha.
\end{aligned}
\end{equation*}
That is to say, $J$ is bounded and has any polynomial decay in $z$, uniformly in $x$. In particular, 
$$ |K(x,y)| = |J(x,x-y)| \lesssim \langle x-y \rangle^{-(d+1)}. $$
The conclusion follows from Schur's Lemma.
\end{proof}

\begin{lemma}[Calderon-Vaillancourt Theorem]\label{lem_tech_C-V}
For smooth symbol $a \in C^\infty(\mathbb{R}^d\times\mathbb{R}^d)$, the following estimate holds : 
\begin{equation}\label{eq_tech_C-V}
\|\Op{a}\|_{\mathcal{L}(L^2)} \lesssim \sup_{|\alpha|,|\beta|\leqslant N_d} \|\partial_x^\alpha\partial_\xi^\beta a\|_{L^\infty(\mathbb{R}^{d}_x \times \mathbb{R}^{d}_\xi)},
\end{equation}
where $N_d$ is a universal constant depending only on dimension $d$.
\end{lemma}
The earliest version of \eqref{eq_tech_C-V} was given in \cite{calderon1971boundedness} where $\alpha_j,\beta_j$ are required to be no more than $3$ for all $j=1,2,\cdots,d$. Then, in \cite{meyer1978operator}, the authors optimized it to $N_d = [\frac{d}{2}]+1$, while some other assumptions in $\alpha,\beta$ are given in the same paper. The readers may also find an alternative proof via Gabor transform in \cite{hwang1987boundedness}.

%% file: appendix/Appendix_stationary_phase_lemma.tex
\section{Stationary phase lemmas}\label{sect_stationary_phase}

\begin{lemma}\label{lem_stat_stationary_phase}
Let 
\begin{equation*}
\begin{array}{lccc}
F: & (\mathbb{S}^{d-1})^4 \times ]0,1] & \rightarrow & \mathbb{C} \\
 & (\theta,\theta',\omega,\omega',\mu) & \mapsto & F(\theta,\theta',\omega,\omega',\mu)
\end{array}
\end{equation*}
be a smooth function supported for $d(\theta',\theta)+d(\theta',\omega')<\delta'$, where $d$ is the metric on sphere $\mathbb{S}^{d-1}$ and $\delta'>0$ is a small constant. We assume that $F$ satisfies for all $\alpha,\alpha',\beta,\beta'\in\mathbb{N}$, $j,N\in\mathbb{N}$
$$ |\partial_{\theta}^{\alpha}\partial_{\theta'}^{\alpha'}\partial_{\omega}^{\beta}\partial_{\omega'}^{\beta'} \partial_{\mu}^{j} F| 
\leqslant C \mu^{-|\alpha|-|\alpha'|-|\beta|-j} \langle\frac{d(\theta',\omega)}{\mu}\rangle^{-N}. $$
For any parameter $\lambda >0$, we define the integral
$$ I_{\pm}(\theta,\omega,\omega';\lambda,\mu) = \int_{\mathbb{S}^{d-1}} e^{\pm i\lambda\theta\theta'} F(\theta,\theta',\omega,\omega',\mu) d\theta'. $$

If $\delta'$ is small enough (depending on the sphere $\mathbb{S}^{d-1}$), and $\lambda\mu^2\geqslant c>0$, then we can write
$$ I_{\pm}(\theta,\omega,\omega';\lambda,\mu) = e^{\pm i\lambda} \mu^{d-1} S^{\pm}_{-\frac{d-1}{2}}(\theta,\omega,\omega',\mu;\lambda\mu^2), $$
where $S^{\pm}_{-\frac{d-1}{2}}(\theta,\omega,\omega',\mu;\zeta)$ is a smooth function supported for $d(\theta,\omega')\leqslant 2\delta'$, satisfying for all $\alpha,\beta,\beta'\in\mathbb{N}$, $j,\gamma,N\in\mathbb{N}$
\begin{equation}\label{eq_lem_stat_stationary_phase_estimate_for_symbol}
|\partial_{\theta}^{\alpha}\partial_{\omega}^{\beta}\partial_{\omega'}^{\beta'} \partial_{\mu}^{j}\partial_{\zeta}^{\gamma} S^{\pm}_{-\frac{d-1}{2}}| 
\leqslant C \mu^{-|\alpha|-|\beta|-j} \langle\zeta\rangle^{-\frac{d-1}{2}-\gamma} \langle\frac{d(\theta,\omega)}{\mu}\rangle^{-N}.
\end{equation}

Moreover, for $|\zeta|\geqslant c>0$, $S^{\pm}_{-\frac{d-1}{2}}(\theta,\omega,\omega',\mu;\zeta)$ can be decomposed as
\begin{equation}\label{eq_lem_stat_stationary_phase_deompose_formula}
(2\pi)^{\frac{d-1}{2}} e^{\mp i\frac{\pi}{4}(d-1)} F(\theta,\theta,\omega,\omega',\mu)\zeta^{-\frac{d-1}{2}} + S^{\pm}_{-\frac{d+1}{2}}(\theta,\omega,\omega',\mu;\zeta),
\end{equation}
where $S^{\pm}_{-\frac{d+1}{2}}(\theta,\omega,\omega',\mu;\zeta)$ is smooth and supported for $d(\theta,\omega')\leqslant 2\delta'$, satisfying the estimate \eqref{eq_lem_stat_stationary_phase_estimate_for_symbol} with $d-1$ replaced by $d+1$.
\end{lemma}

\begin{lemma}\label{lem_stat_stationary_phase_remainder}
Let 
\begin{equation*}
\begin{array}{lccc}
F_0: & (\mathbb{S}^{d-1})^4 \times ]0,1] & \rightarrow & \mathbb{C} \\
 & (\theta,\theta',\omega,\omega',\mu) & \mapsto & F(\theta,\theta',\omega,\omega',\mu)
\end{array}
\end{equation*}
be a smooth function supported for $\min(d(\theta',\theta),d(\theta',-\theta))>\delta'>0$, where $d$ is the metric on sphere $\mathbb{S}^{d-1}$ and $\delta'$ is constant. We assume that $F$ satisfies for all $\alpha,\alpha',\beta,\beta'\in\mathbb{N}$, $j,N\in\mathbb{N}$
$$ |\partial_{\theta}^{\alpha}\partial_{\theta'}^{\alpha'}\partial_{\omega}^{\beta}\partial_{\omega'}^{\beta'} \partial_{\mu}^{j} F_0| 
\leqslant C \mu^{-|\alpha|-|\alpha'|-|\beta|-j} \langle\frac{d(\theta',\omega)}{\mu}\rangle^{-N}. $$

For any parameter $\lambda >0$, the integral
$$ I_{\pm}(\theta,\omega,\omega';\lambda,\mu) = \int_{\mathbb{S}^{d-1}} e^{\pm i\lambda\theta\theta'} F_0(\theta,\theta',\omega,\omega',\mu) d\theta' $$
can be written as
$$ I_{\pm}(\theta,\omega,\omega';\lambda,\mu) = e^{\pm i\lambda\theta\omega} \mu^{d-1} R^{\pm}(\theta,\omega,\omega',\mu;\lambda\mu), $$
where $R^{\pm}(\theta,\omega,\omega',\mu;\zeta)$ is a smooth function satisfying for all $\alpha,\beta,\beta'\in\mathbb{N}$, $j,\gamma,N\in\mathbb{N}$
\begin{equation}\label{eq_lem_stat_stationary_phase_remainder_estimate_for_symbol}
|\partial_{\theta}^{\alpha}\partial_{\omega}^{\beta}\partial_{\omega'}^{\beta'} \partial_{\mu}^{j}\partial_{\zeta}^{\gamma} R^{\pm}| 
\leqslant C \mu^{-|\alpha|-|\beta|-j} \langle\zeta\rangle^{-N}.
\end{equation}
\end{lemma}

For the proof of Lemma \ref{lem_stat_stationary_phase} and \ref{lem_stat_stationary_phase_remainder}, one may find a general result in \cite{delort2022microlocal}, Proposition A.1.1.

\begin{remark}\label{rmk_stat_dependence_on_para}
In Lemma \ref{lem_stat_stationary_phase} and Lemma \ref{lem_stat_stationary_phase_remainder}, the dependence on $\omega'$ is not crucial in the proof, we may eliminate the conditions and results involving $\omega'$. Meanwhile, if $F$ (resp. $F_0$) depends on some other parameters, the same condition can be inherited by $S_{-\frac{d-1}{2}}$ (resp. $R$) from $F$ (resp. $F_0$).
\end{remark}

Lemma \ref{lem_stat_stationary_phase} and \ref{lem_stat_stationary_phase_remainder}, together with Remark \ref{rmk_stat_dependence_on_para}, result in the following lemmas, which are important techniques used in the main text.
\begin{lemma}\label{lem_stat_stationary_phase_expansion}
Let
\begin{equation*}
\begin{array}{lccc}
G: & (\mathbb{S}^{d-1})^2 \times ]0,1] & \rightarrow & \mathbb{C} \\
&(\theta',\omega,\mu) & \mapsto & G(\theta',\omega,\mu)
\end{array}
\end{equation*}
be a smooth function supported for $d(\theta',\omega) < \delta'$, where $d$ is the metric on $\mathbb{S}^{d-1}$ and $\delta'>0$ is a small constant. We further assume that, for all $\alpha',\beta\in\mathbb{N}^{d-1}$, $j,N\in\mathbb{N}$,
$$ | \partial_{\theta'}^{\alpha'} \partial_{\omega}^{\beta} \partial_{\mu}^{j} G | \lesssim \mu^{-|\alpha'|-|\beta|-j} \langle \frac{d(\theta',\omega)}{\mu} \rangle^{-N}. $$
For $\lambda>0$, $\epsilon\in\{\pm\}$, we define
$$ I(\theta,\omega,\mu;\lambda) := \int_{\mathbb{S}^{d-1}} e^{\epsilon i \lambda \theta\cdot\theta'} G(\theta',\omega,\mu) d\theta'. $$
Then, under the conditions that $\delta'$ is small enough and that $\lambda\mu^2>c>0$ for some constant $c$, the integral $I$ may be written as the sum of principal terms 
$$ e^{\pm \epsilon i\lambda} \mu^{d-1} S^{\pm}_{-\frac{d-1}{2}}(\theta,\omega,\mu;\lambda\mu^2), $$
and a reminder
$$ e^{\epsilon i\lambda \theta\cdot\omega} \mu^{d-1} R(\theta,\omega,\mu;\lambda\mu), $$ 
where $S^{\pm}_{m}(\theta,\omega,\mu;\zeta)$ is a smooth function supported on $d(\theta,\pm\omega) < 2\delta'$, $\zeta>c>0$, satisfying for all $\alpha,\beta\in\mathbb{N}^{d-1}$, $j,n,N\in\mathbb{N}$,
$$ | \partial_{\theta}^{\alpha} \partial_{\omega}^{\beta} \partial_{\mu}^j \partial_{\zeta}^n S^{\pm}_{m} | \lesssim \mu^{-|\alpha|-|\beta|-j} \langle\zeta\rangle^{m-n} \langle \frac{d(\theta,\pm\omega)}{\mu} \rangle^{-N} ; $$
and $R(\theta,\omega,\mu;\zeta)$ is a smooth function satisfying for all $N\in\mathbb{N}$,
$$ |R| \lesssim \langle\zeta\rangle^{-N}.$$
Moreover, if $G$ depends on some extra parameters, the same bound for $G$ (and its derivatives in parameters) can be inherited by $S^{\pm}_{-\frac{d-1}{2}}$ and $R$.
\end{lemma}
\begin{proof}
We only give the proof for $\epsilon=+$, while the other one can be treated in the same way. To begin with, we introduce the functions
\begin{align*}
&F_{\pm}(\theta,\theta',\omega,\mu) := G(\pm\theta',\pm\omega,\mu) \chi(\theta-\theta'), \\
&F_0(\theta,\theta',\omega,\mu) := G(\theta',\omega,\mu) \left( 1 - \chi(\theta-\theta') - \chi(\theta+\theta') \right),
\end{align*}
where $\chi\in C_c^\infty(\mathbb{R}^d)$ is radial and supported in a small neighbor of zero, with value $1$ near zero. By using these functions, we may rewrite the integral $I$ into $I_+ + I_- + I_0$ with
\begin{align*}
I_\pm =& \int e^{\pm i \lambda \theta\cdot\theta'} F_{\pm}(\theta,\theta',\pm\omega,\mu) d\theta', \\
I_0 =& \int e^{i \lambda \theta\cdot\theta'} F_0(\theta,\theta',\omega,\mu) d\theta'.
\end{align*}
It is clear that $F_{\pm}$, $F_0$ verify the conditions in Lemma \ref{lem_stat_stationary_phase} and \ref{lem_stat_stationary_phase_remainder}, respectively. Then, by applying these lemmas, we can obtain the desired expressions and corresponding estimates. There remain two points to check: the support of $S^{\pm}_{-\frac{d-1}{2}}$ and the dependence on extra parameters. The latter is merely an application of Remark \ref{rmk_stat_dependence_on_para}, while the former can be shown by observing that, when $\operatorname{Supp}\chi$ is taken to be small enough, the integrand of integrals $I_{\pm}$ is supported on $d(\pm\theta',\omega) < \delta'$ and $d(\theta,\theta') < \delta'$.
\end{proof}

\begin{lemma}\label{lem_stat_int_on_angular_limit_part}
Let $\chi\in C_{c}^\infty(\mathbb{R}^d)$ and $f \in C_{c}^\infty(\mathbb{R}^d\backslash\{0\})$. Assume that $P$ is smooth on $]0,\infty[$, and $\Lambda$ is a positive function, such that
\begin{align*}
&\Lambda(\rho) \sim \rho^\sigma,\ \ \text{as }\rho\rightarrow+\infty, \\
&\left|\Lambda^{(j)}(\rho)\right| \lesssim \rho^{\sigma-j},\ \ \forall \rho>\rho_0>0, j\in\mathbb{N}, \\
&\frac{\Lambda(\rho)}{\rho^\sigma} \rightarrow \lambda_0,\ \ \text{as }\rho\rightarrow+\infty,
\end{align*}
hold for some $\sigma\in\mathbb{R}$ and $\lambda_0>0$.

Then there exists $t_0 \gg 1$ depending on $f$, $P'$, and $\Lambda$, such that for $\epsilon,\epsilon'\in \{+1,-1\}$, $\delta+\frac{\sigma}{2}\in[0,\frac{1}{2}[$, $r>c>0$, and $t>t_0$, the integral
$$ \int_{\mathbb{S}^{d-1}} e^{i\epsilon' r\rho\omega\theta} \chi\left(\frac{r\omega + \epsilon t P'(\rho)\theta}{t^{\frac{1}{2}+\delta}\Lambda(t^\frac{1}{2}\rho)}\right) f(\rho\theta) d\theta $$
can be decomposed as a principal term
$$ (2\pi)^\frac{d-1}{2} e^{i\epsilon\epsilon'\frac{\pi}{4}(d-1)} e^{-i\epsilon\epsilon' r\rho} (r\rho)^{-\frac{d-1}{2}}\chi\left(\frac{r - t P'(\rho)}{t^{\frac{1}{2}+\delta}\Lambda(t^\frac{1}{2}\rho)}\right) f(-\epsilon\rho\omega) $$
and a remainder
$$ e^{-i\epsilon\epsilon' r\rho} \mu^{d-1} S_{-\frac{d+1}{2}}(\omega,\mu,\rho,\frac{r}{t}-P'(\rho),t; r\rho\mu^2), $$
where $\mu=t^{\delta+\frac{\sigma}{2}-\frac{1}{2}}$, $S_{m}(\omega,\mu,\rho,r',t; \zeta)$ is supported for $\zeta>c>0$, $\rho\sim 1$ and 
$|r'| \lesssim \mu$ and satisfies for all $\alpha\in\mathbb{N}^{d-1}$, $j,k,l,\gamma\in\mathbb{N}$,
$$ | \partial_{\omega}^{\alpha} \partial_{\mu}^j \partial_{\rho}^{k} \partial_{r'}^{l} \partial_{\zeta}^{\gamma} S_{-\frac{d+1}{2}} |
\leqslant C \mu^{-(|\alpha|+j+l)} \langle\zeta\rangle^{m-\gamma}. $$
\end{lemma}
\begin{proof}
Using the notation $r = t(r'+P'(\rho))$, where $|r'| \lesssim \mu=t^{\delta+\frac{\sigma}{2}-\frac{1}{2}}$, we can rewrite the integral as 
\begin{equation}\label{eq_lem_stat_int_on_angular_limit_part_int_on_theta}
\int_{\mathbb{S}^{d-1}} e^{i\epsilon' t(r'+P'(\rho))\rho\omega\theta} \chi\left(\frac{r'\omega +  P'(\rho)(\omega+\epsilon\theta)}{t^{-\frac{1}{2}+\delta}\Lambda(t^\frac{1}{2}\rho)}\right) f(\rho\theta) d\theta.
\end{equation}

Consider the function
$$F(x,y,z,\mu;\rho,r',t) = \chi\left(\frac{r'x + P'(\rho)(x-y)}{\mu t^{-\frac{\sigma}{2}}\Lambda(t^\frac{1}{2}\rho)}\right) \tilde{\chi}(\frac{x-y}{\mu}) \tilde{\chi}(\frac{y-z}{\mu}) f(\rho y),$$
with $\tilde{\chi}\in C_{c}^{\infty}(\mathbb{R}^d)$ taking value $1$ in a large ball centered at the origin. By setting $\lambda = t(r'+P'(\rho))\rho>c>0$, we may rewrite the integral \eqref{eq_lem_stat_int_on_angular_limit_part_int_on_theta} as
\begin{equation*}
\int_{\mathbb{S}^{d-1}} e^{-i\epsilon\epsilon' \lambda(-\epsilon\omega)\theta} F(-\epsilon\omega,\theta,-\epsilon\omega,\mu;\rho,r',t) d\theta.
\end{equation*}
Remark that the integrand in \eqref{eq_lem_stat_int_on_angular_limit_part_int_on_theta} ensures that $|\omega+\epsilon\theta| \lesssim \mu \ll 1$, which allows us to add $\tilde{\chi}$ factors in the definition of $F$.

By using the fact that $\rho\sim1$, it is easy to verify that when $x,y,z\in\mathbb{S}^{d-1}$, for all $\alpha,\alpha',\beta\in\mathbb{N}^{d-1}$, $j,k,l,N\in\mathbb{N}$, 
$$ |\partial_{x}^{\alpha}\partial_{y}^{\alpha'}\partial_{z}^{\beta}\partial_{\mu}^{j}\partial_{\rho}^{k}\partial_{r'}^{l} F(x,y,z,\mu;\rho,r')| 
\leqslant C \mu^{-|\alpha|-|\alpha'|-|\beta|-j-l} \langle\frac{d(y,z)}{\mu}\rangle^{-N},$$
and that $F$ is supported for $\rho\sim 1$, $|r'| \lesssim \mu$, $d(x,y)\leqslant C_0 \mu \ll 1$, which ensures that $\lambda\mu^2 = t^{2(\delta+\frac{\sigma}{2})}(r'+P'(\rho))\rho>c>0$. Therefore, the conclusion follows from Lemma \ref{lem_stat_stationary_phase} with extra parameters $\rho$, $r'$, $t$ but without $\omega'$ variable.
\end{proof}

%% file: appendix/Appendix_Schrodinger_equation.tex
\section{A refined result for Schrödinger equation}\label{sect_Schrodinger}
In the case of Schrödinger equation $P(\xi)=\frac{|\xi|^2}{2}$, the structure of $P'(\xi) =\xi$ allows us to prove the parts (\romannumeral1) and (\romannumeral2) of Theorem \ref{thm_main} for non-smooth $\chi$, namely
\begin{theorem}\label{thm_main_Schrodinger}
Let $a_{\chi,\delta}$ and $E_{\chi,\delta}$ be as defined in \eqref{eq_truncation_symbol} and \eqref{eq_truncated_energy}, respectively. We assume, as in Theorem \ref{thm_main}, that $\Lambda$ is identically equal to $1$. Then, for any $u_0 \in L^2$, we have :

(\romannumeral1) if $\chi\in L^1$ and $\delta <0$,
\begin{equation}\label{eq_limit_of_energy_subcritical_Schrodinger}
\lim_{t\rightarrow+\infty} E_{\chi,\delta}(u_0,t) = \lim_{t\rightarrow+\infty} E_{\chi,\delta}(u_0,-t) = 0;
\end{equation}

(\romannumeral2) if $\chi\in L^1$ and $\delta =0$,
\begin{equation}\label{eq_limit_of_energy_critical_Schrodinger}
\lim_{t\rightarrow+\infty} E_{\chi,\delta}(u_0,t) = \lim_{t\rightarrow+\infty} E_{\chi,\delta}(u_0,-t) = \frac{1}{(2\pi)^d} \int G_\chi(\omega) |\hat{u}_0(\rho\omega)|^2 \rho^{d-1} d\rho d\omega,
\end{equation}
where $(\rho,\omega)$ is polar coordinate, and function $G_\chi(\omega)$ is defined as
\begin{equation}\label{eq_definition_of_G_Schrodinger}
G_\chi(\omega):= \frac{1}{(2\pi)^d} \left| \int_{x\cdot\omega>0} e^{i\frac{|x|^2}{2}}\chi(x)dx \right|^2;
\end{equation}
\end{theorem} 

In order to prove the limit \eqref{eq_limit_of_energy_subcritical_Schrodinger} and \eqref{eq_limit_of_energy_critical_Schrodinger}, we indicate that, compared with Theorem \ref{thm_main}, the only difficulty is the loss of regularity in $\chi$, which can be overcame by finding a bound of operator $\Op{a(t)}$ depending merely on $\|\chi\|_{L^1}$.

\subsection{An alternative bound of truncated operator}\label{subsect_boundedness_Schrodinger}

The goal of this subsection is to find a uniform bound of $\Op{a(t)}$, which involves less regularity of $\chi$ than Proposition \ref{prop_boundedness_of_truncated_operator}. Via Lemma \ref{lem_tech_change_of_scaling}, it suffices to study the symbol
$$ b(x,\xi) = \chi\left( \frac{x+\xi}{\lambda} \right) \mathbbm{1}_{|x|>|\xi|}, $$
where $\lambda = |t|^\delta \in ]0,\infty[$. To be precise, we shall prove the following proposition

\begin{proposition}\label{prop_boundedness_Schrodinger_small_parameter}
There exists constant $C>0$ independent of $\lambda$, such that, for all $\lambda>0$,
$$ \|\Op{b}\|_{\mathcal{L}(L^2)} \leqslant C \lambda^d \|\chi\|_{L^1}. $$
\end{proposition}

\begin{proof}[Proof of Proposition \ref{prop_boundedness_Schrodinger_small_parameter}]
It is easy to calculate that
\begin{align*}
\int |b(x,\xi)| dx& \leqslant \int \left| \chi\left(\frac{x+\xi}{\lambda}\right) \right| dx = \lambda^d \|\chi\|_{L^1}, \\
\int |b(x,\xi)| d\xi& \leqslant \int \left| \chi\left(\frac{x+\xi}{\lambda}\right) \right| d\xi = \lambda^d \|\chi\|_{L^1}.
\end{align*}
Thus, the desired estimate follows from Schur's Lemma and Lemma \ref{lem_tech_estimate_as_kernel}.
\end{proof}

\subsection{Calculation of limit}

By Proposition \ref{prop_boundedness_Schrodinger_small_parameter}, when $\delta<0$, we have 
$$ \| \Op{a} \|_{\mathcal{L}(L^2)} \leqslant C t^{\delta d} \|\chi\|_{L^1} \rightarrow 0,\ \ \text{as }t\rightarrow \infty, $$
which implies \eqref{eq_limit_of_energy_subcritical_Schrodinger}.

In the critical case $\delta=0$, when $\chi$ is smooth, compactly supported, and constant near zero, limit \eqref{eq_limit_of_energy_critical_Schrodinger} follows from the limit \eqref{eq_limit_of_energy_critical}. If $\chi$ is no more than an $L^1$ function, we may approximate $\chi$ by some regular function in $L^1$. To be precise, for all $n\in\mathbb{N}$, there exists $\chi_n \in C_c^\infty(\mathbb{R}^d)$ which are constant near zero, such that
$$ \|\chi-\chi_n\|_{L^1} < \frac{1}{n}. $$
To highlight the dependence in $\chi$, in the rest of this section, we will add subscript $\chi$ for concerning terms, for example,
$$ a_{\chi}(t,x,\xi) = \chi\left(\frac{x+t\xi}{|t|^{\frac{1}{2}}}\right)\mathbbm{1}_{|x|>|t||\xi|}. $$

Proposition \ref{prop_boundedness_Schrodinger_small_parameter} implies that, for all $u_0\in L^2$ and $t \neq 0$,
\begin{equation*}
\begin{aligned}
\| \Op{a_\chi}e^{itP(D_x)}u_0 - \Op{a_{\chi_n}}e^{itP(D_x)}u_0 \|_{L^2} 
&= \| \Op{a_{\chi-\chi_n}}e^{itP(D_x)}u_0 \|_{L^2} \\
&\leqslant C \|\chi-\chi_n\|_{L^1} \|e^{itP(D_x)}u_0\|_{L^2} \\
&< C \frac{1}{n} \|u_0\|_{L^2}.
\end{aligned}
\end{equation*}
Therefore, for fixed $u_0\in L^2$, the limit
$$ \Op{a_\chi}e^{itP(D_x)}u_0 \rightarrow \Op{a_{\chi_n}}e^{itP(D_x)}u_0\ \text{in }L^2,\ \ \text{as }n\rightarrow \infty $$
is uniform in $t$. We may conclude \eqref{eq_limit_of_energy_critical_Schrodinger} by passing to the limit $t\rightarrow\pm\infty$ : 
\begin{equation*}
\begin{aligned}
\lim_{t\rightarrow\pm\infty} E_{\chi,\delta}(u_0,t) &= \lim_{t\rightarrow\pm\infty} \| \Op{a_\chi}e^{itP(D_x)}u_0 \|_{L^2}^2 \\
&= \lim_{t\rightarrow\pm\infty} \lim_{n\rightarrow\infty} \| \Op{a_{\chi_n}}e^{itP(D_x)}u_0 \|_{L^2}^2 \\
&= \lim_{n\rightarrow\infty} \lim_{t\rightarrow\pm\infty} \| \Op{a_{\chi_n}}e^{itP(D_x)}u_0 \|_{L^2}^2 \\
&= \lim_{n\rightarrow\infty} \frac{1}{(2\pi)^d} \int G_{\chi_n}(\omega) |\hat{u}_0(\rho\omega)|^2 \rho^{d-1} d\rho d\omega \\
&= \frac{1}{(2\pi)^d} \int G_{\chi}(\omega) |\hat{u}_0(\rho\omega)|^2 \rho^{d-1} d\rho d\omega 
\end{aligned}
\end{equation*}
The last equality follows from Dominated Convergence Theorem and the continuity of $G$ on $\chi\in L^1$, which is obvious due to the definition \eqref{eq_definition_of_G_Schrodinger} of $G_{\chi}$.

%% file: appendix/Appendix_Klein-Gordon_classical_truncation.tex
\section{Partition of energy for Klein-Gordon equation : a classical setting}\label{sect_K-G_classical}

In this part, we shall give a proof of \eqref{eq_K-G_limit_of_energy_classical} via an alternative study of asymptotic behavior of solution to half-Klein-Gordon equation \eqref{eq_half_KG}, namely
\begin{equation*}
\left\{
\begin{aligned}
&\left(\frac{\partial_t}{i} - P(D_x)\right)u = 0, \\
&u|_{t=0} = u_0,
\end{aligned}
\right.
\end{equation*}
where $P(\xi) = \langle\xi\rangle$ is a smooth symbol. As in Section \ref{sect_Klein-Gordon}, instead of studying the solution $w$ to Klein-Gordon equation \eqref{eq_Klein-Gordon}, we turn to 
$$ u = \left( \frac{\partial_t}{i} + P(D_x) \right)w, $$
which is a solution to half-Klein-Gordon equation with initial data
$$ u_0 = u|_{t=0} =\frac{w_1}{i} + P(D_x)w_0 \in L^2. $$

Using this notation, we may rewrite the integral on the left hand side of \eqref{eq_K-G_limit_of_energy_classical} as
\begin{align*}
&\int_{r_0<\left|\frac{x}{t}\right|<r_1} \left( |\partial_t w|^2 + |\nabla w|^2 + |w|^2 \right) dx \\
=& \int \left( \left|\mathbbm{1}_{r_0<\left|\frac{x}{t}\right|<r_1}\partial_t w\right|^2 + \left|\mathbbm{1}_{r_0<\left|\frac{x}{t}\right|<r_1}\nabla w\right|^2 + \left|\mathbbm{1}_{r_0<\left|\frac{x}{t}\right|<r_1} w\right|^2 \right) dx \\
=& \int \left( \left|\mathbbm{1}_{r_0<\left|\frac{x}{t}\right|<r_1}\Imaginary u\right|^2 + \left|\mathbbm{1}_{r_0<\left|\frac{x}{t}\right|<r_1}i\frac{D_x}{\langle D_x \rangle}\Real u\right|^2 + \left|\mathbbm{1}_{r_0<\left|\frac{x}{t}\right|<r_1} \langle D_x \rangle^{-1}\Real u\right|^2 \right) dx \\
=& \int \left( \left|\Imaginary \mathbbm{1}_{r_0<\left|\frac{x}{t}\right|<r_1}u\right|^2 + \left|\Imaginary \mathbbm{1}_{r_0<\left|\frac{x}{t}\right|<r_1}\frac{D_x}{\langle D_x \rangle}u\right|^2 + \left|\Real \mathbbm{1}_{r_0<\left|\frac{x}{t}\right|<r_1} \langle D_x \rangle^{-1}u\right|^2 \right) dx \\
=& \int \Bigg( \left|\Imaginary \mathbbm{1}_{r_0<\left|\frac{x}{t}\right|<r_1} e^{itP(D_x)}u_0\right|^2 + \left|\Imaginary \mathbbm{1}_{r_0<\left|\frac{x}{t}\right|<r_1}e^{itP(D_x)}\frac{D_x}{\langle D_x \rangle}u_0\right|^2 \\
&\hspace{16em} + \left|\Real \mathbbm{1}_{r_0<\left|\frac{x}{t}\right|<r_1} e^{itP(D_x)}\langle D_x \rangle^{-1}u_0\right|^2 \Bigg) dx 
\end{align*}

Notice that the term on the right hand takes the form of
\begin{equation}\label{eq_K-G_classic_general_form_of_energy}
\int \left| \mathcal{A} g\left(\frac{x}{t}\right)v(t,x) \right|^2 dx,
\end{equation}
where $\mathcal{A}\in\{\Real,\Imaginary\}$, $g(y)=\mathbbm{1}_{r_0<|y|<r_1}$, and $v$ is a solution to half-Klein-Gordon equation, which can be written as $v(t) = e^{itP(D_x)}v_0$, with
$$ v_0 = u_0,\ \frac{D_x}{\langle D_x \rangle}u_0,\ \langle D_x \rangle^{-1}u_0, $$
which all belong to $L^2$, since $u_0\in L^2$. In order to calculate the limit of such quantity, we need to study the asymptotic behavior of $v$, which will be given in the next part.

\subsection{Asymptotic behavior}

In this part, we shall state our problem in a general setting. Let $u$ be the unique solution to \eqref{eq_frac_disper} with initial data $u_0\in L^2$. We assume the symbol $P$ is smooth except at zero, which covers all the fractional-type equations. Since $P'$ is well-defined except at zero, we may introduce $\nu$ the push forward of Lebesgue measure under $P'$, i.e.
\begin{equation}\label{eq_K-G_classic_def_of_nu}
\nu(E) := \operatorname{Leb}(P'^{-1}(E)),\ \ \forall E\subset\mathbb{R}^d\ \text{measurable}.
\end{equation}
For all function $g\in L^\infty$, we are interested in the following asymptotic formula :
\begin{equation}\label{eq_K-G_classic_asymp}
g\left(\frac{x}{t}\right)u(t) = g(-P'(D_x)) u(t) + o_{L^2}(1),\ \ \text{as }t\rightarrow \pm\infty.
\end{equation}
We denote by $\mathcal{G}$ the collection of those function $g$ satisfying this formula for all $u_0\in L^2$, namely
\begin{equation*}
\begin{aligned}
\mathcal{G} :=& \{ g\in L^\infty : \eqref{eq_K-G_classic_asymp}\ \text{holds for all } u_0\in L^2 \} \\
:=& \{ g\in L^\infty : \eqref{eq_K-G_classic_asymp}\ \text{holds for all } u_0\ \text{with }\hat{u}_0\in C_c^\infty(\mathbb{R}^d) \}.
\end{aligned}
\end{equation*}
The two definitions given above are equivalent since the multiplication with $g(x/t)$ and Fourier multiplier $g(-P'(D_x))$ are both bounded on $L^2$ uniformly in time. The equivalence then follows from the fact that the subspace $\mathcal{F}^{-1}C_c^\infty(\mathbb{R}^d\backslash\{0\})$ is dense in $L^2$, where $\mathcal{F}$ is the Fourier transform.

In order to prove \eqref{eq_K-G_limit_of_energy_classical}, we may use the following lemma
\begin{lemma}\label{lem_K-G_classic_asymp}
Let $E\subset\mathbb{R}^d$ be any measurable set whose boundary has null Lebesgue measure, namely
$$ \operatorname{Leb}(\partial E) = 0. $$
If the measure $\nu$ defined in \eqref{eq_K-G_classic_def_of_nu} is absolutely continuous w.r.t. Lebesgue measure, we have
$$ \mathbbm{1}_{E} \in \mathcal{G}. $$
\end{lemma}

We assume this lemma is true and prove it later. It is easy to check that the assumptions in Lemma \ref{lem_K-G_classic_asymp} hold true for $P(\xi)=\langle\xi\rangle$ and $E = \{r_0<|y|<r_1\}$. As a result, \eqref{eq_K-G_classic_general_form_of_energy} can be written as
\begin{align*}
\int \left| \mathcal{A} \mathbbm{1}_{E}\left(\frac{x}{t}\right)v(t,x) \right|^2 dx =& \int \left| \mathcal{A} \mathbbm{1}_{E}\left(-P'(D_x)\right)v(t,x) \right|^2 dx + o(1) \\
=& \int \left| \mathcal{A} \mathbbm{1}_{]\rho_0,\rho_1[}(|D_x|)v(t,x) \right|^2 dx + o(1),
\end{align*}
when $t\rightarrow\pm\infty$. Recall that $]\rho_0,\rho_1[ = P'^{-1}(]r_0,r_1[)$. Actually, this formula implies \eqref{eq_K-G_limit_of_energy_classical}, since as $t\rightarrow\pm\infty$, we have
\begin{align*}
&\int_{r_0<\left|\frac{x}{t}\right|<r_1} \left( |\partial_t w|^2 + |\nabla w|^2 + |w|^2 \right) dx \\
=& \int \Bigg( \left|\Imaginary \mathbbm{1}_{r_0<\left|\frac{x}{t}\right|<r_1} u(t,x)\right|^2 + \left|\Imaginary \mathbbm{1}_{r_0<\left|\frac{x}{t}\right|<r_1}\frac{D_x}{\langle D_x \rangle}u(t,x)\right|^2 \\
&\hspace{16em} + \left|\Real \mathbbm{1}_{r_0<\left|\frac{x}{t}\right|<r_1} \langle D_x \rangle^{-1}u(t,x)\right|^2 \Bigg) dx  \\
=& \int \Bigg( \left|\Imaginary \mathbbm{1}_{]\rho_0,\rho_1[}(|D_x|) u(t,x)\right|^2 + \left|\Imaginary \mathbbm{1}_{]\rho_0,\rho_1[}(|D_x|)\frac{D_x}{\langle D_x \rangle}u(t,x)\right|^2 \\
&\hspace{12em} + \left|\Real \mathbbm{1}_{]\rho_0,\rho_1[}(|D_x|) \langle D_x \rangle^{-1}u(t,x)\right|^2 \Bigg) dx + o(1)  \\
=& \left\|\mathbbm{1}_{]\rho_0,\rho_1[}(|D_x|) \partial_t w\right\|^2_{L^2} + \left\|\mathbbm{1}_{]\rho_0,\rho_1[}(|D_x|)\nabla w\right\|^2_{L^2} + \left\| \mathbbm{1}_{]\rho_0,\rho_1[}(|D_x|) w\right\|^2_{L^2} + o(1)  \\
=& \left\|\mathbbm{1}_{]\rho_0,\rho_1[}(|D_x|) \partial_t w\right\|^2_{L^2} + \left\|\mathbbm{1}_{]\rho_0,\rho_1[}(|D_x|)w\right\|^2_{H^1} + o(1)  \\
=& \| \mathbbm{1}_{]\rho_0,\rho_1[}(|D_x|) w_1\|_{H^1}^2 + \| \mathbbm{1}_{]\rho_0,\rho_1[}(|D_x|) w_0\|_{L^2}^2 + o(1).
\end{align*}

In order to complete the proof of \eqref{eq_K-G_limit_of_energy_classical}, we give now the proof of Lemma \ref{lem_K-G_classic_asymp}.

\begin{proof}[Proof of Lemma \ref{lem_K-G_classic_asymp}]
Without loss of generality, we may assume in what follows that $\hat{u}_0 \in C_c^\infty(\mathbb{R}^d)$. Now, we fix $\chi\in C_c^\infty(\mathbb{R}^d)$ which equals to $1$ in a ball centered at zero and consider the symbol
\begin{equation}\label{eq_K-G_classic_truncation_symbol}
a_0(t,x,\xi) = \chi \left( \frac{x+tP'(\xi)}{|t|^{\frac{1}{2}+\delta}} \right),
\end{equation}
where $\delta \in ]0,\frac{1}{2}[$. We have seen that Lemma \ref{lem_tech_change_of_x_and_xi} and \ref{lem_tech_boundedness_of_symbol_study_in_xi} imply the uniform-in-$t$ boundedness of $\Op{a_0(t)}$. Moreover, by integration by parts, we can prove that for all $u_0 \in L^2$
$$ \left\| u(t) - \Op{a_0(t)}u(t) \right\|_{L^2} \rightarrow 0,\ \ \text{as }t\rightarrow \pm\infty. $$

By writing
\begin{equation*}
\begin{aligned}
E_t &:= E + B(0,ct^{-\frac{1}{2}+\delta}),\\
\tilde{E}_t &:= E^c + B(0,ct^{-\frac{1}{2}+\delta}),
\end{aligned}
\end{equation*}
where $c>0$ is a constant determined by $\Supp\chi$, we may apply the uniform-in-time $L^2$-boundedness of multiplication with $\mathbbm{1}_{E}(x/t)$ and Fourier multiplier $\mathbbm{1}_{E}(-P'(D_x))$ to obtain that, as $t\rightarrow \pm\infty$,
\begin{align*}
&\mathbbm{1}_{E}\left(\frac{x}{t}\right)u(t,x) - \mathbbm{1}_{E}(-P'(D_x))u(t,x) \\
=& \mathbbm{1}_{E}\left(\frac{x}{t}\right)\Op{a_0(t)}u(t,x) - \Op{a_0(t)}\mathbbm{1}_{E}(-P'(D_x))u(t,x) + o_{L^2}(1) \\
=& \mathbbm{1}_{E}\left(\frac{x}{t}\right)\Op{a_0(t)}\mathbbm{1}_{E_t}(-P'(D_x))u(t,x) - \Op{a_0(t)}\mathbbm{1}_{E}(-P'(D_x))u(t,x) + o_{L^2}(1) \\
=& \mathbbm{1}_{E}\left(\frac{x}{t}\right)\Op{a_0(t)}\mathbbm{1}_{E_t\backslash E}(-P'(D_x))u(t,x) \\
&\hspace{8em}- \mathbbm{1}_{E^c}\left(\frac{x}{t}\right)\Op{a_0(t)}\mathbbm{1}_{E}(-P'(D_x))u(t,x) + o_{L^2}(1) \\
=& \mathbbm{1}_{E}\left(\frac{x}{t}\right)\Op{a_0(t)}\mathbbm{1}_{E_t\backslash E}(-P'(D_x))u(t,x) \\
&\hspace{8em}- \mathbbm{1}_{E^c}\left(\frac{x}{t}\right)\Op{a_0(t)}\mathbbm{1}_{E\cap\tilde{E}_t}(-P'(D_x))u(t,x) + o_{L^2}(1)
\end{align*}
Note that in the calculation above, it is possible to add extra cut-off in $-P'(D_x)$ since the symbol $a_0(t)$ is supported for
$$ \left| \frac{x}{t} - (-P'(\xi)) \right| \leqslant ct^{-\frac{1}{2}+\delta}. $$

As a consequence, the uniform boundedness of multiplication with $\mathbbm{1}_{E}(x/t)$, $\mathbbm{1}_{E^c}$ and $\Op{a_0(t)}$ implies that
\begin{align*}
&\limsup_{t\rightarrow+\infty}\left\| \mathbbm{1}_{E}\left(\frac{x}{t}\right)u(t) - \mathbbm{1}_{E}(-P'(D_x))u(t) \right\|_{L^2} & \\
\lesssim& \limsup_{t\rightarrow+\infty} \left( \left\| \mathbbm{1}_{E_t\backslash E}(-P'(D_x))u(t) \right\|_{L^2} + \left\| \mathbbm{1}_{E\cap\tilde{E}_t}(-P'(D_x))u(t) \right\|_{L^2} \right) &\\
=& (2\pi)^{-\frac{d}{2}} \limsup_{t\rightarrow+\infty} \left( \left\| \mathbbm{1}_{E_t\backslash E}(-P')\hat{u}_0 \right\|_{L^2} + \left\| \mathbbm{1}_{E\cap\tilde{E}_t}(-P')\hat{u}_0 \right\|_{L^2} \right) &(\text{Plancherel Theorem})\\
=& (2\pi)^{-\frac{d}{2}} \left( \left\| \mathbbm{1}_{\cap_{t>0}E_t\backslash E}(-P')\hat{u}_0 \right\|_{L^2} + \left\| \mathbbm{1}_{\cap_{t>0}E\cap\tilde{E}_t}(-P')\hat{u}_0 \right\|_{L^2} \right) &(\text{DCT})\\
\lesssim& \left\| \mathbbm{1}_{\partial E}(-P')\hat{u}_0 \right\|_{L^2} &
\end{align*}
The last quantity is actually zero since $\partial E$ has zero Lebesgue measure and $\nu$ is absolutely continuous w.r.t. Lebesgue measure, which ensures that $\mathbbm{1}_{\partial E}(-P'(\xi))$ vanishes almost everywhere. 
\end{proof}

\subsection{Further remarks on \texorpdfstring{$\mathcal{G}$}{G}}

In the proof of \eqref{eq_K-G_limit_of_energy_classical}, we only study the function
$$ g(y) = \mathbbm{1}_{r_0<|y|<r_1}. $$
A natural question is whether \eqref{eq_K-G_limit_of_energy_classical} holds true for other function $g$, or equivalently, which type of function is contained in class $\mathcal{G}$. In this part, we shall indicate an error in a classical result and prove that $\mathcal{G}$ contains at least the continuous functions.

The asymptotic formula \eqref{eq_K-G_classic_asymp} is, to our knowledge, firstly studied in \cite{strichartz1981asymptotic}, where the author claimed in Corollary 2.2 of \cite{strichartz1981asymptotic} that, when $\nu$ is absolutely continuous w.r.t Lebesgue measure,
\begin{equation}\label{eq_K-G_classic_conj}
L^\infty = \mathcal{G}.
\end{equation}
However, the original proof given in \cite{strichartz1981asymptotic} is false. Actually, the author managed to prove in Theorem 2.1 that for general $\nu$,
$$ \{\text{Fourier transform of finite measure}\} \subset \mathcal{G}, $$
and reduce the conjecture to 
\begin{equation}\label{eq_K-G_classic_conj'}
\mathbbm{1}_{E} \subset \mathcal{G},\ \ \text{for bounded and measurable }E\subset\mathbb{R}^d.
\end{equation}

The method used in \cite{strichartz1981asymptotic} is that, by regularity of Lebesgue measure, we may choose a series of compact sets $\{K_n\}_{n\in\mathbb{N}}$ and bounded open sets $\{U_n\}_{n\in\mathbb{N}}$, such that for all $n\in\mathbb{N}$
$$ K_n \subset K_{n+1},\ \ U_{n+1} \subset U_n,\ \ K_n \subset E \subset U_n; $$
$$ \lim_{n\rightarrow +\infty} \operatorname{Leb}(U_n \backslash K_n) = 0. $$
Then it is easy to find smooth functions $g_n\in C^\infty_c(U_n)$ which are non-negative, range in $[0,1]$, and equal to $1$ on $K_n$. Clearly, $\mathbbm{1}_{E} - g_n$ is supported in $U_n \backslash K_n$, whose measure tends to zero. By Dominated Convergence Theorem and the fact that $0\leqslant \mathbbm{1}_{E} - g_n \leqslant 1$, we have
$$ \limsup_{t\rightarrow\pm\infty}\|(\mathbbm{1}_{E} - g_n)(-P'(D_x))u \|_{L^2}^2 = (2\pi)^{-d} \|(\mathbbm{1}_{E} - g_n)(-P'(\xi))\hat{u}_0(\xi)\|_{L^2_\xi}^2 \xrightarrow{n\rightarrow+\infty} 0. $$
Remark that, to apply Dominated Convergence Theorem, we need $(\mathbbm{1}_{E} - g_n)(-P'(\xi))$ converges to zero almost everywhere, which is a consequence of absolute continuity of $\nu$ w.r.t Lebesgue measure. Therefore, in order to prove \eqref{eq_K-G_classic_conj'}, it suffices to check that
\begin{equation}\label{eq_K-G_classic_approx_x}
\limsup_{t\rightarrow\pm\infty}\|(\mathbbm{1}_{E} - g_n)(x/t)u \|_{L^2_x}^2 \xrightarrow{n\rightarrow+\infty} 0.
\end{equation}
Once it holds true, one has, for all $n\in\mathbb{N}$,
\begin{align*}
&\limsup_{t\rightarrow\pm\infty}\|\mathbbm{1}_{E}(x/t)u - \mathbbm{1}_{E}(-P'(D_x))u \|_{L^2_x} \\
\leqslant& \limsup_{t\rightarrow\pm\infty}\|g_n(x/t)u - g_n(-P'(D_x))u \|_{L^2_x} + \limsup_{t\rightarrow\pm\infty}\|(\mathbbm{1}_{E} - g_n)(x/t)u \|_{L^2_x} \\
& + \limsup_{t\rightarrow\pm\infty}\|(\mathbbm{1}_{E} - g_n)(-P'(D_x))u \|_{L^2} \\
=& \limsup_{t\rightarrow\pm\infty}\|(\mathbbm{1}_{E} - g_n)(x/t)u \|_{L^2_x}
 + \limsup_{t\rightarrow\pm\infty}\|(\mathbbm{1}_{E} - g_n)(-P'(D_x))u \|_{L^2},
\end{align*}
since $g_n$ is obviously the Fourier transform of some finite measure. The two quantities on the right hand side vanish as $n$ tends to infinity, and the desired conclusion $\mathbbm{1}_{E}\in\mathcal{G}$ follows.

The proof of \eqref{eq_K-G_classic_approx_x} given in \cite{strichartz1981asymptotic} is to find $h_n\in\mathcal{G}$ such that
$$ C \geqslant h_n \geqslant \mathbbm{1}_{O_n}\ \ \text{and}\ \ h_n \rightarrow 0,\ \text{a.e.}, $$
where $O_n =U_n \backslash K_n$ is an open set. If such $h_n$ exists, one obtains immediately that
\begin{align*}
\lim_{n\rightarrow+\infty} \limsup_{t\rightarrow\pm\infty} \| (\mathbbm{1}_{E}-g_n)(x/t) u(t) \|_{L^2}^2 
\leqslant& \lim_{n\rightarrow+\infty} \limsup_{t\rightarrow\pm\infty} \| h_n(x/t) u(t) \|_{L^2}^2 \\
\leqslant& \lim_{n\rightarrow+\infty} \limsup_{t\rightarrow\pm\infty} \| h_n(-P'(D_x)) u(t) \|_{L^2}^2 \\
=& \lim_{n\rightarrow+\infty} (2\pi)^{-d} \int h_n^2(-\nabla P(\xi)) \left| \hat{u}_{0}(\xi) \right|^2 d\xi = 0.
\end{align*}
Note that the second inequality is a consequence of $h_n \in \mathcal{G}$ and the last equality follows from Dominated Convergence Theorem. 

Since we only know that $\mathcal{G}$ contains a subset of continuous function (Fourier transform of finite measures), it is essential to assume $h_n$'s to be continuous. However, for general decreasing bounded open sets $O_n$, even if their measures decrease to zero, such continuous $h_n$'s do NOT exist. Otherwise, it is harmless to assume $h_n$'s are supported in the same large ball. Then, for one thing by Dominated Convergence Theorem,
$$ \lim_{n\rightarrow +\infty} \int h_n(y) dy = 0;  $$
and for another thing, we have
$$ \int h_n(y) dy \geqslant \int_{\overline{O_n}} h_n(y) dy = \operatorname{Leb}(\overline{O_n}) \geqslant 0, $$
where $\overline{O_n}$ is the closure of open set $O_n$. Here the first inequality is due to the continuity of $h_n$'s. As a result,
$$ \lim_{n\rightarrow +\infty} \operatorname{Leb}(\overline{O_n}) = 0. $$
The contradiction arises from the fact that $\operatorname{Leb}(O_n)$ tends to zero does not imply that $\operatorname{Leb}(\overline{O_n})$ tends to zero. For example,Let $\{r_j\}_{j\in\mathbb{N}}$ be a numeration of rational numbers in the unit ball $B=B(0,1)$ of $\mathbb{R}^d$ centered at zero. Consider the series of open sets
$$ O_n := \cup_{j\in\mathbb{N}} B(r_j,2^{-j-n}). $$
Clearly $O_n$'s are open as the union of open sets and 
$$ \operatorname{Leb}(O_n) \leqslant \sum_{j\in\mathbb{N}} \operatorname{Leb}\left(B(r_j,2^{-j-n})\right) \lesssim \sum_{j\in\mathbb{N}}2^{-(j+n)d} \sim 2^{-nd} \rightarrow 0,\ \ \text{as }n\rightarrow +\infty. $$
Since $\{r_j\}_{j\in\mathbb{N}}$ is dense in $B=B(0,1)$, the closure of each $O_n$ contains at least the unit ball $B$. As a consequence,
$$ \lim_{n\rightarrow +\infty} \operatorname{Leb}(\overline{O_n}) \geqslant \lim_{n\rightarrow +\infty} \operatorname{Leb}(B) = C_d>0. $$
We emphasize that the argument above does not falsify \eqref{eq_K-G_classic_conj} and it is still unknown whether this conjecture is true. Here we shall prove rigorously that all bounded continuous functions belong to $\mathcal{G}$.

\begin{proposition}\label{prop_K-G_classic_continuous_fct}
If the measure $\nu$ defined in \eqref{eq_K-G_classic_def_of_nu} is absolutely continuous w.r.t. Lebesgue measure, we have
$$ C_b^0(\mathbb{R}^d) \subset \mathcal{G}. $$
\end{proposition}
\begin{proof}
As in the proof of Lemma \ref{lem_K-G_classic_asymp}, we assume that $\hat{u}_0\in C_c^\infty(\mathbb{R}^d)$. To begin with, we check that the Schwartz class $\mathcal{S} \subset \mathcal{G}$. For any $g\in\mathcal{S}$,
\begin{align*}
&\left( g\left(\frac{x}{t}\right) - g(-P'(D_x)) \right) u(t,x) \\
=& \frac{1}{(2\pi)^d} \int e^{i(x\cdot\xi+tP(\xi))} \left( g\left(\frac{x}{t}\right) - g(-P'(\xi)) \right) \hat{u}_0(\xi) d\xi \\
=& \frac{1}{(2\pi)^d} \int e^{i(x\cdot\xi+tP(\xi))} \left( \frac{x}{t}+P'(\xi) \right) \int_0^1 g'\left(\frac{\tau}{t}x-(1-\tau)P'(\xi)\right) d\tau \hat{u}_0(\xi) d\xi \\
=& \frac{i}{(2\pi)^d t} \int e^{i(x\cdot\xi+tP(\xi))} \partial_\xi \left[ \int_0^1 g'\left(\frac{\tau}{t}x-(1-\tau)P'(\xi)\right) d\tau \hat{u}_0(\xi) \right] d\xi \\
=& \frac{i}{(2\pi)^d t} \int e^{i(x\cdot\xi+tP(\xi))} \int_0^1 g''\left(\frac{\tau}{t}x-(1-\tau)P'(\xi)\right) (\tau-1) d\tau P''(\xi)\hat{u}_0(\xi) d\xi \\
+&\frac{i}{(2\pi)^d t} \int e^{i(x\cdot\xi+tP(\xi))} \int_0^1 g'\left(\frac{\tau}{t}x-(1-\tau)P'(\xi)\right) d\tau \partial_\xi\hat{u}_0(\xi) d\xi.
\end{align*}
Notice that, by Lemma \ref{lem_tech_change_of_x_and_xi} and \ref{lem_tech_boundedness_of_symbol_study_in_xi}, the operator of symbol
$$ g''\left(\frac{\tau}{t}x-(1-\tau)P'(\xi)\right),\ \ g'\left(\frac{\tau}{t}x-(1-\tau)P'(\xi)\right) $$
is bounded uniformly in $t$ and $\tau$ and that functions
$$ P''(\xi)\hat{u}_0(\xi),\ \ \partial_\xi\hat{u}_0(\xi) $$
belong to $L^2$, since we have assumed $\hat{u}_0\in C_c^\infty(\mathbb{R}^d)$. As a consequence,
$$ \left\| \left( g\left(\frac{x}{t}\right) - g(-P'(D_x)) \right) u(t,x) \right\|_{L^2_x} \lesssim |t|^{-1} \xrightarrow{t\rightarrow\pm\infty} 0.$$
By noticing that $\mathcal{G}$ is closed under $L^\infty$-norm, we have
$$ C^0_0 := \{ g\in C^0 : \lim_{|y|\rightarrow+\infty} g(y)=0 \} = \overline{\mathcal{S}} \subset \mathcal{G}, $$
where $\overline{S}$ is the closure of $\mathcal{S}$ w.r.t. $L^\infty$-norm. 

It remains to pass to general continuous function $g$. In fact, we only need to consider those $g \geqslant 0$, since once may always write $g$ as the difference of two non-negative continuous functions, which are both bounded. Let us fix $\chi\in C_c^\infty$ which equals to $1$ near zero and define, for all $R>0$,
$$ \chi_R(y) := \chi\left( \frac{y}{R} \right). $$
For arbitrary $R>0$, we have
\begin{align*}
&\limsup_{t\rightarrow\pm\infty}\left\| \left( g(x/t) - g(-P'(D_x)) \right) u(t) \right\|_{L^2}  \\
\leqslant&  \limsup_{t\rightarrow\pm\infty}\left\| \left( (g\chi_R)(x/t) - (g\chi_R)(-P'(D_x)) \right) u(t) \right\|_{L^2}  \\
&+ \limsup_{t\rightarrow\pm\infty}\left\| (g(1-\chi_R))(x/t) u(t) \right\|_{L^2} + \limsup_{t\rightarrow\pm\infty}\left\| (g(1-\chi_R))(-P'(D_x)) u(t) \right\|_{L^2} \\
=& \limsup_{t\rightarrow\pm\infty}\left\| (g(1-\chi_R))(x/t) u(t) \right\|_{L^2} + \limsup_{t\rightarrow\pm\infty}\left\| (g(1-\chi_R))(-P'(D_x)) u(t) \right\|_{L^2},
\end{align*}
since $g\chi_R \in C^0_0$. The second term on the right hand side can be calculated as follows
$$ \limsup_{t\rightarrow\pm\infty}\left\| (g(1-\chi_R))(-P'(D_x)) u(t) \right\|_{L^2} = (2\pi)^{-\frac{d}{2}} \left\| (g(1-\chi_R))(-P') \hat{u}_0 \right\|_{L^2}, $$
which, due to Dominated Convergence Theorem and the absolute continuity of $\nu$ w.r.t. Lebesgue measure, converges to zero as $R \rightarrow +\infty$. As for the cut-off in $x$, we observe that,
$$ \limsup_{t\rightarrow\pm\infty}\left\| (g(1-\chi_R))(x/t) u(t) \right\|_{L^2} \leqslant \|g\|_{L^\infty} \limsup_{t\rightarrow\pm\infty}\left\| (1-\chi_{R})(x/t) u(t) \right\|_{L^2}. $$
Note that since $\chi_{R} \in \mathcal{S}\subset\mathcal{G}$, $(1-\chi_{R})(x/t) u(t)$ can be written, when $t\rightarrow\pm\infty$, as 
$$ u(t) - \chi_{R}(-P'(D_x))u(t) + o_{L^2}(1). $$
As a result, 
\begin{align*}
\limsup_{t\rightarrow\pm\infty}\left\| (1-\chi_{R})(x/t) u(t) \right\|_{L^2} =& \limsup_{t\rightarrow\pm\infty}\left\| (1-\chi_{R})(-P'(D_x)) u(t) \right\|_{L^2} \\
=& (2\pi)^{-\frac{d}{2}}\left\| (1-\chi_{R}) \hat{u}_0 \right\|_{L^2}.
\end{align*}
We have seen that the last quantity tends to zero as $R\rightarrow+\infty$. In conclusion, we have proved that, for all $u_0 \in \mathcal{F}^{-1}C_c^\infty(\mathbb{R}^d\backslash\{0\})$,
$$ \lim_{R\rightarrow+\infty}\limsup_{t\rightarrow\pm\infty}\left\| \left( g(x/t) - g(-P'(D_x)) \right) u(t) \right\|_{L^2} =0,  $$
and thus $g\in\mathcal{G}$.
\end{proof}

\subsection{Proof of \eqref{eq_K-G_classic_conj} for dispersive system}

Before ending this section, we give a proof of \eqref{eq_K-G_classic_conj} for dispersive system. To be precise, we assume that there exists some dense subspace $\mathcal{D}_0\subset L^2$, such that
\begin{equation}\label{eq_K-G_classic_dispersion_estimate}
\|e^{itP(D_x)}u_0\|_{L^\infty} \leqslant C(u_0)|t|^{-\frac{d}{2}},\ \ \forall u_0\in\mathcal{D}_0,\ \ \forall |t|>1,
\end{equation}
where $C(u_0)$ is a constant depending on $u_0$. 

\begin{theorem}
Assume that $P$ is smooth except at zero and $\nu$ defined in \eqref{eq_K-G_classic_def_of_nu} is absolutely continuous w.r.t. Lebesgue measure. Then the dispersion estimate \eqref{eq_K-G_classic_dispersion_estimate}, associated with some dense subspace $\mathcal{D}_0\subset L^2$, implies \eqref{eq_K-G_classic_conj}.
\end{theorem}
\begin{proof}
Recall that, as mentioned in previous section, it has been proved in \cite{strichartz1981asymptotic} that, when $\nu$ is absolutely continuous w.r.t. Lebesgue measure, \eqref{eq_K-G_classic_conj} is equivalent to \eqref{eq_K-G_classic_conj'}. Therefore, it suffices to check that, for all $u_0\in\mathcal{D}_0$,
$$ \left\| \mathbbm{1}_{E}\left(\frac{x}{t}\right)u(t) - \mathbbm{1}_{E}(D_x)u(t) \right\|_{L^2} \rightarrow 0,\ \ \text{as }t\rightarrow \pm\infty, $$
where $E$ is any bounded measurable set.

In Lemma \ref{lem_K-G_classic_asymp}, we have proved this result for those $E$ whose boundary has zero Lebesgue measure. The idea of treatment of general $E$ is to approximate it by finite union of cubes and control the remaining part via dispersion estimate. To begin with, we fix an arbitrarily small $\epsilon>0$. By outer regularity of Lebesgue measure, there exists open set $\tilde{E}_\epsilon \supset E$, such that
$$ \operatorname{Leb}(\tilde{E}_\epsilon\backslash E) < \frac{\epsilon}{2}. $$
Since any open set can be expressed as the union of almost disjoint closed cubes, we may find finitely many closed cubes $\{K_j\}_{j=1}^N$, such that $K_j\subset \tilde{E}_\epsilon$ and 
$$ \operatorname{Leb}(\tilde{E}_\epsilon\backslash E_\epsilon) < \frac{\epsilon}{2},\ \ \text{where }E_\epsilon=\underset{j=1}{\overset{N}{\cup}} K_j.$$
One may observe that
$$ \operatorname{Leb}(E\backslash E_\epsilon \sqcup E_\epsilon\backslash E) <\epsilon.$$

Now, for any $u_0\in\mathcal{D}_0$, we have
\begin{align*}
&\left\| \mathbbm{1}_{E}\left(\frac{x}{t}\right)u(t) - \mathbbm{1}_{E}(D_x)u(t) \right\|_{L^2} \\
\leqslant& \left\| (\mathbbm{1}_{E_\epsilon}-\mathbbm{1}_{E})\left(\frac{x}{t}\right)u(t) \right\|_{L^2} + \left\| \mathbbm{1}_{E_\epsilon}\left(\frac{x}{t}\right)u(t) - \mathbbm{1}_{E_\epsilon}(D_x)u(t) \right\|_{L^2} \\
&+ \left\| (\mathbbm{1}_{E_\epsilon}-\mathbbm{1}_{E})(D_x)u(t) \right\|_{L^2} \\
\leqslant& C(u_0)\left\| (\mathbbm{1}_{E_\epsilon}-\mathbbm{1}_{E})\left(\frac{x}{t}\right) \right\|_{L^2}|t|^{-\frac{d}{2}} + \left\| \mathbbm{1}_{E_\epsilon}\left(\frac{x}{t}\right)u(t) - \mathbbm{1}_{E_\epsilon}(D_x)u(t) \right\|_{L^2} \\
&+ (2\pi)^{-\frac{d}{2}}\left\| (\mathbbm{1}_{E_\epsilon}-\mathbbm{1}_{E})u_0 \right\|_{L^2} \\
\leqslant& C(u_0) \operatorname{Leb}(E\backslash E_\epsilon \sqcup E_\epsilon\backslash E)^{\frac{1}{2}} + \left\| \mathbbm{1}_{E_\epsilon}\left(\frac{x}{t}\right)u(t) - \mathbbm{1}_{E_\epsilon}(D_x)u(t) \right\|_{L^2} \\
&+ (2\pi)^{-\frac{d}{2}}\left\| \mathbbm{1}_{E\backslash E_\epsilon \sqcup E_\epsilon\backslash E}(\xi)\hat{u}_0(\xi) \right\|_{L^2_{\xi}} \\
<& C(u_0) \epsilon^{\frac{1}{2}} + \left\| \mathbbm{1}_{E_\epsilon}\left(\frac{x}{t}\right)u(t) - \mathbbm{1}_{E_\epsilon}(D_x)u(t) \right\|_{L^2} + (2\pi)^{-\frac{d}{2}}\left\| \mathbbm{1}_{E\backslash E_\epsilon \sqcup E_\epsilon\backslash E}(\xi)\hat{u}_0(\xi) \right\|_{L^2_{\xi}}.
\end{align*}
Due to the fact that $E_\epsilon$ is the union of finitely many closed cubes, the boundary of $E_\epsilon$ has zero Lebesgue measure. As a result, the second term on the right hand side tends to zero as $t\rightarrow \pm\infty$, i.e.
$$ \limsup_{t\rightarrow \pm\infty}\left\| \mathbbm{1}_{E}\left(\frac{x}{t}\right)u(t) - \mathbbm{1}_{E}(D_x)u(t) \right\|_{L^2} \lesssim \epsilon^{\frac{1}{2}} + \left\| \mathbbm{1}_{E\backslash E_\epsilon \sqcup E_\epsilon\backslash E}(\xi)\hat{u}_0(\xi) \right\|_{L^2_{\xi}}. $$
Since $\hat{u}_0 \in L^2$ and $\operatorname{Leb}(E\backslash E_\epsilon \sqcup E_\epsilon\backslash E) <\epsilon$, the right hand side becomes arbitrarily small, if $\epsilon>0$ is taken small enough. The desired result thus follows.
\end{proof}

The dispersion estimate \eqref{eq_K-G_classic_dispersion_estimate} holds for all symbols $P$ we deal with in the present paper, as a consequence of stationary phase lemma.
\begin{proposition}
Let $P$ be radial and smooth except at zero. If $P''>0$ or $P''<0$, the dispersion estimate \eqref{eq_K-G_classic_dispersion_estimate} holds with $\mathcal{D}_0 = \mathcal{F}^{-1}C_c^\infty(\mathbb{R}^d\backslash\{0\})$. 
\end{proposition} 
\begin{proof}
Without loss of generality, we assume that $t>1$. Let $u_0$ be any function in $\mathcal{D}_0 = \mathcal{F}^{-1}C_c^\infty(\mathbb{R}^d\backslash\{0\})$. To prove the inequality \eqref{eq_K-G_classic_dispersion_estimate}, it is equivalent by definition to prove that
$$ \left\| \int e^{i(x\cdot\xi+tP(\xi))} \hat{u}_0(\xi) d\xi \right\|_{L^\infty(dx)} \lesssim t^{-\frac{d}{2}}. $$
Since $P$ is radial, we may write the inequality above in polar system $\xi=\rho\theta$ as
$$ \left\| \int e^{it(\frac{x\cdot\theta}{t}\rho+P(\rho))} \hat{u}_0(\rho\theta) \rho^{d-1} d\rho d\theta \right\|_{L^\infty(dx)} \lesssim t^{-\frac{d}{2}}. $$
By denoting $s=x\cdot\theta/t \in \mathbb{R}$, it suffices to prove that
$$ \sup_{s\in\mathbb{R},\theta\in\mathbb{S}^{d-1}} \left| \int e^{it(s\rho+P(\rho))} \hat{u}_0(\rho\theta) \rho^{d-1} d\rho \right| \lesssim t^{-\frac{d}{2}}. $$
Since $\rho$ stays between two positive constants and $P''\neq 0$ on $]0,+\infty[$, this inequality follows from stationary phase lemma.
\end{proof}
\begin{corollary}
\eqref{eq_K-G_classic_conj} holds for all $P$ satisfying condition \eqref{hyp_fractional-type_symbol} with $p_0,p_1\neq 0$. 
\end{corollary}

%% file: appendix/Appendix_limit_of_integral.tex
\section{Proof of Proposition \ref{prop_lim_limit_of_integral}}\label{sect_proof_of_limit_of_integral}

In \cite{delort2022microlocal}, the author has proved Proposition \ref{prop_lim_limit_of_integral} for strictly convex $P$. In this part, we will explain how the same argument works for strictly concave $P$ and how to calculate the limit for $\epsilon=\pm$ respectively. Since most of calculations has been done in Section 3 of \cite{delort2022microlocal}, we will omit these details.

By definition $I(t,-\epsilon,\epsilon,-\epsilon,\epsilon;F)$ equals to
\begin{equation*}
\begin{aligned}
&\int e^{i\epsilon[ r(-\rho+\rho') - t(- P(\rho)+ P(\rho'))]} \mathbbm{1}_{\frac{r}{t}>P'(\rho),P'(\rho')}\\
&\hspace{4em}\times F(\rho,\rho',r,t;r-tP'(\rho), r-tP'(\rho'))  dr d\rho d\rho',
\end{aligned}
\end{equation*}
which can be split into $I_+$ and $I_-$, with domain of integral $\rho-\rho'>0$ and $\rho-\rho'<0$, respectively. Namely,
\begin{equation*}
\begin{aligned}
I_+=&\int e^{i\epsilon[ r(-\rho+\rho') - t(- P(\rho)+ P(\rho'))]} \mathbbm{1}_{r>tP'(\rho')} \mathbbm{1}_{\rho-\rho'>0}\\
&\hspace{4em}\times F(\rho,\rho',r,t;r-tP'(\rho), r-tP'(\rho'))  dr d\rho d\rho', \\
I_-=&\int e^{i\epsilon[ r(-\rho+\rho') - t(- P(\rho)+ P(\rho'))]} \mathbbm{1}_{r>tP'(\rho')} \mathbbm{1}_{\rho-\rho'<0}\\
&\hspace{4em}\times F(\rho,\rho',r,t;r-tP'(\rho), r-tP'(\rho'))  dr d\rho d\rho'.
\end{aligned}
\end{equation*}
In what follows, we study mainly integral $I_+$ with $I_-$ manipulated in the same way. By change of variable $r\rightarrow tr+tP'(\rho')$, $\rho' \rightarrow \rho-w$, $I_+$ reads
\begin{equation*}
\begin{aligned}
&t\int e^{i\epsilon t[ -(r+P'(\rho-w))w - (- P(\rho)+ P(\rho-w))]} \mathbbm{1}_{r>0} \mathbbm{1}_{w>0}\\
&\hspace{2em}\times F(\rho,\rho-w,t(r+P'(\rho-w)),t;tr - tP'(\rho) + tP'(\rho-w), tr)  dr d\rho dw.
\end{aligned}
\end{equation*}

We introduce the following notation:
$$ P(\rho') - P(\rho) = P'(\rho)(\rho'-\rho) + g(\rho,\rho')(\rho'-\rho)^2, $$
where $g$ is strictly negative since $P$ is concave;
$$ \tilde{F}(\rho,\rho',r,t;\zeta,\zeta') = F(\rho,\rho',tr,t;t\zeta,t\zeta'), $$
which is smooth, supported for
$$ \rho,\rho',r \sim 1,\ |\zeta|,|\zeta'| \lesssim t^{\delta'-1}, $$
and satisfies for all $j,j',k,\gamma,\gamma'\in\mathbb{N}$,
$$ | \partial_{\rho}^{j} \partial_{\rho'}^{j'} \partial_{r}^{k} \partial_{\zeta}^{\gamma} \partial_{\zeta'}^{\gamma'} \tilde{F}(\rho,\rho',r,t;\zeta,\zeta') |
\lesssim t^{(1-\delta')(k+\gamma+\gamma')}. $$
Moreover, we have the point-wise limit
$$ \lim_{t\rightarrow\infty} \tilde{F}(\rho,\rho',\frac{r}{\sqrt{t}}+P'(\rho'),t;\frac{\zeta}{\sqrt{t}},\frac{\zeta'}{\sqrt{t}}) = F_{0}(\rho,\rho'). $$

With these notations, $I_+$ can be expressed as
\begin{equation*}
\begin{aligned}
&t\int e^{-i\epsilon t[rw - g(\rho-w,\rho)w^2 ]} \mathbbm{1}_{r>0} \mathbbm{1}_{w>0}\\
&\hspace{2em}\times \tilde{F}(\rho,\rho-w,r+P'(\rho-w),t;r - P'(\rho) + P'(\rho-w), r)  dr d\rho dw.
\end{aligned}
\end{equation*}
The same calculus as in Lemma 3.1.4 of \cite{delort2022microlocal} shows that, up to some terms tending to zero, $I_+$ equals to
\begin{align*}
&t\int e^{-i\epsilon trw}\mathbbm{1}_{r>0} e^{i\epsilon tg(\rho-w,\rho)w^2}  \mathbbm{1}_{w>0}\tilde{F}(\rho,\rho-w,r+P'(\rho),t;r, r)  dr dw d\rho \\
=& \int e^{-i\epsilon rw}\mathbbm{1}_{r>0} e^{i\epsilon g(\rho-\frac{w}{\sqrt{t}},\rho)w^2}  \mathbbm{1}_{w>0}\tilde{F}(\rho,\rho-\frac{w}{\sqrt{t}},\frac{r}{\sqrt{t}}+P'(\rho),t;\frac{r}{\sqrt{t}}, \frac{r}{\sqrt{t}})  dr dw d\rho,
\end{align*}
whose formal limit by taking point-wise limit of integrand is
$$ \int e^{-i\epsilon rw}\mathbbm{1}_{r>0} e^{i\epsilon g(\rho,\rho)w^2}  \mathbbm{1}_{w>0} F_0(\rho,\rho)  dr dw d\rho. $$
The error between this integral and $I_+$ is actually $o(1)$, due the calculation of the proof of Proposition 3.1.3 of \cite{delort2022microlocal}. In conclusion, we have
\begin{align*}
I_+ =& \int e^{-i\epsilon rw}\mathbbm{1}_{r>0} e^{i\epsilon g(\rho,\rho)w^2}  \mathbbm{1}_{w>0} F_0(\rho,\rho)  dr dw d\rho + o(1),\\
I_- =& \int e^{-i\epsilon rw}\mathbbm{1}_{r>0} e^{-i\epsilon g(\rho,\rho)w^2}  \mathbbm{1}_{w<0} F_0(\rho,\rho)  dr dw d\rho + o(1).
\end{align*}

As a consequence, the limit of $I(t,-\epsilon,\epsilon,-\epsilon,\epsilon;F)$ is 
\begin{align*}
&\int e^{-i\epsilon rw}\mathbbm{1}_{r>0}dr \left( e^{i\epsilon g(\rho,\rho)w^2}  \mathbbm{1}_{w>0} + e^{-i\epsilon g(\rho,\rho)w^2}  \mathbbm{1}_{w<0} \right) dw F_0(\rho,\rho) d\rho \\
=& \int -i\epsilon(w-i\epsilon 0)^{-1} \left( \cos (\epsilon g(\rho,\rho)w^2) + \sgn{w}i \sin (\epsilon g(\rho,\rho)w^2) \right) dw F_0(\rho,\rho) d\rho \\
=& \int -i\epsilon(w-i\epsilon 0)^{-1} \cos (g(\rho,\rho)w^2) dw F_0(\rho,\rho) d\rho \\
&\hspace{12em}+ \int \frac{\sin (g(\rho,\rho)w^2)}{|w|} dw F_0(\rho,\rho) d\rho,
\end{align*}
where the integrals in $r$ and $w$ should be understood in the sense of oscillatory integral and distribution, respectively.

Due to the assumption that $P$ is strictly concave, or equivalently $P''<0$, $g$ is negative. Since $(w-i\epsilon 0)^{-1}$ is homogeneous of degree $-1$, the right hand side of equality above can be rewritten as
\begin{align*}
	&\int -i\epsilon(w-i\epsilon 0)^{-1} \cos (w^2) dw F_0(\rho,\rho) d\rho \\
	&\hspace{12em}- \int \frac{\sin (w^2)}{|w|} dw F_0(\rho,\rho) d\rho,
\end{align*}

For one thing, the distribution $(w\pm i0)^{-1}$ can be expressed as
$$ (w \pm i0)^{-1} = \mp\pi i \delta_{0} + \operatorname{P.V.}\frac{1}{w}. $$
And for another thing, due to Dirichlet integral  $\int_0^\infty \frac{\sin{y}}{y}dy = \frac{\pi}{2}$, we have
\begin{equation*}
	\int \frac{\sin (w^2)}{|w|} dw = \frac{\pi}{2}.
\end{equation*}
These identities imply that
\begin{align*}
	\lim_{t\rightarrow+\infty} I(t,-\epsilon,\epsilon,-\epsilon,\epsilon;F) 
	= & \int \left( \pi\delta_0 - i\epsilon \operatorname{P.V.}\frac{1}{w} \right) \cos (w^2) dw F_0(\rho,\rho) d\rho \\
	&\hspace{10em}+ \sgn{g} \frac{\pi}{2} \int F_0(\rho,\rho) d\rho \\
	= & \pi\int F_0(\rho,\rho) d\rho - \frac{\pi}{2} \int F_0(\rho,\rho) d\rho \\
	= & \frac{\pi}{2} \int F_0(\rho,\rho) d\rho.
\end{align*} \\

In the case of strictly convex $P$, it has been proved in \cite{delort2022microlocal} that
\begin{align*}
	\lim_{t\rightarrow+\infty} I(t,-\epsilon,\epsilon,-\epsilon,\epsilon;F) 
	= &\int i\epsilon(w+i\epsilon 0)^{-1} \cos (g(\rho,\rho)w^2) dw F_0(\rho,\rho) d\rho \\
	&\hspace{8em}- \int \frac{\sin (g(\rho,\rho)w^2)}{|w|} dw F_0(\rho,\rho) d\rho,
\end{align*}
where $g$ is positive. We may repeat the argument above and conclude that
\begin{align*}
	\lim_{t\rightarrow+\infty} I(t,-\epsilon,\epsilon,-\epsilon,\epsilon;F) 
	= & \int \left( \pi\delta_0 + i\epsilon \operatorname{P.V.}\frac{1}{w} \right) \cos (w^2) dw F_0(\rho,\rho) d\rho \\
	&\hspace{10em}- \frac{\pi}{2} \int F_0(\rho,\rho) d\rho \\
	= & \pi\int F_0(\rho,\rho) d\rho - \frac{\pi}{2} \int F_0(\rho,\rho) d\rho \\
	= & \frac{\pi}{2} \int F_0(\rho,\rho) d\rho.
\end{align*}